\documentclass[reqno]{amsart}%
\usepackage{amssymb}
\usepackage{eurosym}
\usepackage{amsfonts}
\usepackage{amsmath}
\usepackage{amsthm}
\usepackage{graphicx}%
\providecommand{\U}[1]{\protect\rule{.1in}{.1in}}
\newtheoremstyle{theorem}
{10pt}		
{10pt}
{\sl}
{\parindent}
{\bf}
{. }
{ }
{}
\theoremstyle{theorem}
\newtheorem{theorem}{Theorem}
\newtheorem{acknowledgement}{Acknowledgement}
\newtheorem{algorithm}{Lemma}[section]

\newtheorem{claim}{Claim}

\newtheorem{corollary}{Corollary}

\newtheorem{lemma}{Lemma}

\newtheorem{proposition}{Proposition}
\newtheorem{remark}{Remark}

\newtheoremstyle{defi}
{10pt}		
{10pt}
{\rm}
{\parindent}
{\bf}
{. }
{ }
{}
\theoremstyle{defi}
\newtheorem{definition}{Definition}

\makeatletter
\@namedef{subjclassname@2020}{\textup{2020} Mathematics Subject Classification}
\makeatother
\begin{document}
\author{Borka Jadrijevi\'{c}}
\address{Faculty of Science, University of Split, Croatia}
\email{borka@pmfst.hr}
\author{Kristina Mileti\'{c}}
\address{Mostar, Bosnia and Herzegovina}
\email{kristina.miletic@gmail.com}
\title{\textbf{Characterization of quadratic }$\varepsilon-$\textbf{CNS}
\textbf{polynomials (extended version)}}
\maketitle

\begin{abstract}
In this paper, we give characterization of quadratic $\varepsilon-$canonical
number system ($\varepsilon-$CNS) polynomials for all values $\varepsilon
\in\lbrack0,1)$. Our characterization provides a unified view of the
well-known characterizations of the classical quadratic CNS polynomials
($\varepsilon=0$) and quadratic SCNS polynomials ($\varepsilon=1/2$). This
result is a consequence of our new characterization results of $\varepsilon
-$shift radix systems ($\varepsilon-$SRS) in the two-dimensional case and
their relation to quadratic $\varepsilon-$CNS polynomials.

\end{abstract}

\section{Introduction}

In the last century, starting with the first examples of Knuth
\cite{Knuth-clanak} and Penney \cite{Penney} up to the early 1990s, canonical
number systems were defined as number systems that allow to represent elements
of orders (in particular, rings of integers) in number fields. The definition
of canonical number systems in number fields and an overview of the early
theory of number systems can be found for example in \cite{Akiyama-BBPT},
\cite{BHP}, \cite{Everste-Gyory}. In 1991 Peth\H{o} \cite{Petho} gave a more
general definition of canonical number systems: Let $P\in\mathbb{Z}[x]$ be a
monic polynomial and let $\mathcal{N\subset}\mathbb{Z}$ be a complete residue
system modulo $P(0)$ containing $0$. The pair $(P,\mathcal{N})$ is called a
\textit{number system} if for each $a\in\mathbb{Z}[x]$ there exist integers
$l>0$ and $d_{0},...,d_{l-1}\in\mathcal{N}$ such that%
\[
a\equiv\sum_{j=0}^{l-1}d_{j}x^{j}(\operatorname{mod}P).
\]
$P$ is called \textit{basis }of this number system and $\mathcal{N}$ is called
its \textit{set of digits}. If such a representation of $a\in\mathbb{Z}[x]$
exists, it is unique if we require $d_{l-1}\neq0$ for $a\not \equiv
0(\operatorname{mod}P)$ and take the empty expansion for $a\equiv
0(\operatorname{mod}P).$ Accordingly, we can assume that $\left\vert
P(0)\right\vert \geq2.$ Moreover, it is important to note that the fact that
$0\in\mathcal{N}$ is crucial for the unicity of the representation.

Choosing the set of digits $\mathcal{N}_{0}=\left\{  0,1,...,\left\vert
P(0)\right\vert -1\right\}  $, the number system $(P,\mathcal{N}_{0})$ is
referred to as a \textit{canonical number system} (CNS for short) and $P$ as a
\textit{CNS basis}\textbf{\ }or \textit{CNS polynomial. }Canonical number
systems $(P,\mathcal{N}_{0})$ have been extensively studied. Many papers are
devoted to the following two problems: to work out an efficient algorithm that
allows to decide whether a given polynomial $P$ is a CNS polynomial or not,
and to give the characterization of CNS polynomials by considering only their
coefficients. So far, a complete description (characterization) of CNS
polynomials is an open problem, even for polynomials of small degree. Various
variants of number systems and canonical number systems have been studied in
the literature. For example, a modification of the set of digits leads us to
the so-called \textit{symmetric CNS }introduced by Akiyama and Scheicher
\cite{Akiyama-Scheicher}. Namely, the number system $(P,\mathcal{N})$ is
called a \textit{symmetric canonical number system} (SCNS) if $P\in
\mathbb{Z}[x]$ is a monic polynomial and $\mathcal{N=}\left[  -\frac
{\left\vert P(0)\right\vert }{2},\frac{\left\vert P(0)\right\vert }{2}\right)
\cap\mathbb{Z}.$ This motivates a generalization that includes both types of
canonical number systems:

\begin{definition}
Let $P\left(  x\right)  =x^{d}+p_{d-1}x^{d-1}+...+p_{1}x+p_{0}\in
\mathbb{Z}\left[  x\right]  ,$ $\left\vert p_{0}\right\vert \geq2,$
$\varepsilon\in\left[  0,1\right)  .$ Then the \textit{number system}
$(P,\mathcal{N}_{\varepsilon})$ is called an $\varepsilon-$\textit{canonical
number system} if\textit{\ }%
\begin{equation}
\mathcal{N}_{\varepsilon}=\left[  -\varepsilon\left\vert p_{0}\right\vert
,\left(  1-\varepsilon\right)  \left\vert p_{0}\right\vert \right)
\cap\mathbb{Z}. \label{N-eps}%
\end{equation}
$P$ is called the \textit{base of the }$\varepsilon-$\textit{CNS} or
$\varepsilon-$\textit{CNS polynomial}. To $\mathcal{N}_{\varepsilon}$ we refer
as the $\varepsilon-$\textit{set of digits. }
\end{definition}

Note that the set $\mathcal{N}_{\varepsilon}$ consists of $\left\vert
p_{0}\right\vert $ consecutive integers and contains $0,\ $and that it depends
on both $\varepsilon$ and $p_{0}=P\left(  0\right)  $. Also note that the case
$\varepsilon=0$ corresponds to the usual CNS, while $\varepsilon=\frac{1}{2}$
corresponds to SCNS. A fundamental problem is to characterize all
$\varepsilon$-CNS polynomials (for example of given degree). It turns out that
this problem is strongly related to a dynamical systems, so called
$\varepsilon$-\textit{shift radix systems} (for details see
\cite{Akiyama-BBPT} and \cite{Surer}). The concept of \textit{shift radix
systems} (SRS) was introduced by Akiyama et al. \cite{Akiyama-BBPT}. Akiyama
and Scheicher \cite{Akiyama-Scheicher} presented a slight modification of SRS,
so called \textit{symmetric shift radix systems} (SSRS). Surer \cite{Surer}
constructed the following new generalization:

\begin{definition}
Let $d\geq1$ be an integer and $\varepsilon\in\left[  0,1\right)  .$ To
$\mathbf{r}=(r_{1},\ldots,r_{d})\in\mathbb{R}^{d}$ we associate the mapping
$\tau_{\mathbf{r},\varepsilon}:\mathbb{Z}^{d}\rightarrow\mathbb{Z}^{d}$ in the
following way:%
\begin{equation}
\mathbf{z}=(z_{1},\ldots,z_{d})\in\mathbb{Z}^{d}\longmapsto\tau_{\mathbf{r}%
,\varepsilon}(\mathbf{z})=(z_{2},\ldots,z_{d},-\left\lfloor \mathbf{rz+}%
\varepsilon\right\rfloor ), \label{tau-eps}%
\end{equation}
where $\mathbf{rz}=r_{1}z_{1}+...+r_{d}z_{d}$, i.e. $\mathbf{rz}$ is the inner
product of vectors $\mathbf{r}$ and $\mathbf{z}$. The mapping $\tau
_{\mathbf{r},\varepsilon}$ is called an $\varepsilon$\textit{-shift radix
system} ($\varepsilon$-SRS) if for any $\mathbf{z}\in\mathbb{Z}^{d}$ there
exists $k\in\mathbb{N}$ such that $\tau_{\mathbf{r},\varepsilon}%
^{k}(\mathbf{z})=\mathbf{0,}$ where $\mathbf{0}=(0,\ldots,0)\in\mathbb{Z}%
^{d}.$
\end{definition}

\noindent We define the following two sets related to the behavior of periods
of $\tau_{\mathbf{r},\varepsilon}$:
\begin{align*}
\mathcal{D}_{d,\varepsilon}  &  =\left\{  \mathbf{r}\in\mathbb{R}^{d}:\left(
\tau_{\mathbf{r},\varepsilon}^{n}(\mathbf{z})\right)  _{n\in\mathbb{N}}\text{
is ultimately periodic for all }\mathbf{z}\in\mathbb{Z}^{d}\right\} \\
\mathcal{D}_{d,\varepsilon}^{0}  &  =\left\{  \mathbf{r}\in\mathbb{R}^{d}%
:\tau_{\mathbf{r},\varepsilon}\text{ is }\varepsilon\text{-SRS}\right\}  .
\end{align*}
Lots of basic properties and notations concerning $\mathcal{D}_{d,\varepsilon
}$ and $\mathcal{D}_{d,\varepsilon}^{0}$ can be directly adopted from the well
analyzed case $\varepsilon=0$ and the case $\varepsilon=\frac{1}{2},$ since
$0$-SRS corresponds to classical SRS while $\frac{1}{2}$-SRS corresponds to
SSRS (see \cite{Surer}). The relation between $0$-SRS and $0$-CNS is given by
Akiyama et al. \cite[Theorem 3.1]{Akiyama-BBPT}. This theorem can easily be
generalized for any $\varepsilon\in\left[  0,1\right)  $:

\begin{theorem}
[{\cite[Theorem 4.2.]{Surer}}]\label{TRM-veza}Let\textit{\ }$\varepsilon
\in\left[  0,1\right)  $ and $P\left(  x\right)  =x^{d}+p_{d-1}x^{d-1}%
+...+p_{1}x+p_{0}\in\mathbb{Z}\left[  x\right]  $. \textit{Then }%
$P$\textit{\ is }$\varepsilon-$CNS polynomial\textit{\ }if and only if
$\left(  \frac{1}{p_{0}},\frac{p_{d-1}}{p_{0}},...,\frac{p_{1}}{p_{0}}\right)
\in\mathcal{D}_{d,\varepsilon}^{0}.$
\end{theorem}

It turned out that the description of the sets $\mathcal{D}_{d,\varepsilon}$
and $\mathcal{D}_{d,\varepsilon}^{0}$ is almost trivial for $d=1$, while
considerable difficulties already arise in the dimension $d=2$. For example,
the set $\mathcal{D}_{2,\varepsilon}^{0}\subset\mathbb{R}^{2}$ for
$\varepsilon=0$ has a very complicated structure, so that $\mathcal{D}%
_{2,0}^{0}$ cannot be described completely (see \cite{Akiyama-BBPT} and
\cite{Akiyama-BBPT-II} for several characterization results on $\mathcal{D}%
_{2,0}^{0}$). On the other hand, it turned out \cite{Akiyama-Scheicher} that
the set $\mathcal{D}_{2,\varepsilon}^{0}\subset\mathbb{R}^{2}$ for
$\varepsilon=\frac{1}{2}$ has a very simple structure, so that it can be
characterized completely (see Proposition \ref{SRSpola}).

Clearly, we have $\mathcal{D}_{d,\varepsilon}^{0}\subseteq\mathcal{D}%
_{d,\varepsilon}$, so any analysis of $\mathcal{D}_{d,\varepsilon}^{0}$ starts
with $\mathcal{D}_{d,\varepsilon}$. The set $\mathcal{D}_{d,\varepsilon}$ is,
up to the boundary, easy to describe. Namely, we have $\mathcal{E}_{d}%
\subset\mathcal{D}_{d,\varepsilon}\subset\overline{\mathcal{E}}_{d}$, where
$\mathcal{E}_{d}$ is an open bounded set characterized by several strict
inequalities. For example in dimension $d=2$, we have
\[
\mathcal{E}_{2}=\left\{  \left(  x,y\right)  \in\mathbb{R}^{2}:\left\vert
x\right\vert <1,\text{ }\left\vert y\right\vert <x+1\right\}  .
\]
For definition of the set $\mathcal{E}_{d}$ (sometimes referred to as the
Schur-Takagi-region) see for example \cite{Akiyama-BBPT} or \cite{Surer}. The
interior of $\mathcal{D}_{d,\varepsilon}$ equals $\mathcal{E}_{d}$, so it does
not depend on $\varepsilon,$ and, consequently, the interior of $\mathcal{D}%
_{d,\varepsilon}$ is equal for all $\varepsilon\in\left[  0,1\right)  $. Set
$\mathcal{D}_{d,\varepsilon}\cap\partial\mathcal{D}_{d,\varepsilon
}=\mathcal{D}_{d,\varepsilon}\diagdown\mathcal{E}_{d}$ is very hard to
describe and probably depends on $\varepsilon$. Furthermore, Surer
\cite{Surer} has shown that $\mathcal{D}_{d,\varepsilon}^{0}$, for
$\varepsilon\in\left(  0,1\right)  \diagdown\left\{  \frac{1}{2}\right\}  $
can be gained by cutting out polyhedra from $\mathcal{D}_{d,\varepsilon}$ and
has presented a method to obtain this \textit{family of cutout polyhedra} of
$\mathcal{D}_{d,\varepsilon}^{0}$ (the method is adopted from the case
$\varepsilon=0$ and slightly modified). He has also shown that $\mathcal{D}%
_{d,\varepsilon}^{0}$ is closely related to $\mathcal{D}_{d,1-\varepsilon}%
^{0}$ for $\varepsilon\in\left(  0,1\right)  \ $and stated several
characterization results of $\mathcal{D}_{d,\varepsilon}^{0}$ in the two
dimensional case. Namely, for each $\varepsilon\in\left(  0,1\right)
\diagdown\left\{  \frac{1}{2}\right\}  $ he found an explicitly given set
$D^{\ast}\left(  \varepsilon\right)  $ with $\mathcal{D}_{2,\varepsilon}%
^{0}\subset D^{\ast}\left(  \varepsilon\right)  \subset\mathcal{E}_{2}$ and
showed that $\mathcal{D}_{2,\varepsilon}^{0}$ can be obtained from $D^{\ast
}\left(  \varepsilon\right)  $ by cutting out finitely many polyhedra. Surer
used these results to give explicit characterizations of $\mathcal{D}%
_{2,\varepsilon}^{0}$ for some particular values of $\varepsilon\in\left(
0,1\right)  \diagdown\left\{  \frac{1}{2}\right\}  \ $(see \cite[Section
5]{Surer}).

In the present paper, we give characterization of quadratic\textbf{\ }%
$\varepsilon-$CNS polynomials for all $\varepsilon\in\left[  0,1\right)  .$
This characterization provides a unified view of the well-known
characterizations of the classical quadratic CNS polynomials ($\varepsilon=0$)
and quadratic SCNS polynomials ($\varepsilon=1/2$). Our result is a
consequence of some new characterization results on $\mathcal{D}%
_{2,\varepsilon}^{0}$ and Theorem \ref{TRM-veza}. Some of the consequences of
our main result fit into the general results presented in
\cite{Petho-Thuswaldner}, as well as in \cite{EGPT} (see Remark \ref{rem-PT}).

The paper is organized as follows. In Section \ref{Sect2}, we state the main
result and explain the basic idea of its proof. The main result, given by
Theorem \ref{TRM-main-manje}, provides the necessary and sufficient conditions
for a quadratic polynomial $P(x)=x^{2}+p_{1}x+p_{0}\in\mathbb{Z}\left[
x\right]  $ to be an $\varepsilon-$CNS polynomial for any $\varepsilon
\in\left[  0,1\right)  .$ To prove Theorem \ref{TRM-main-manje}, we first
define two subsets $D\left(  \varepsilon\right)  \ $and $B\left(
\varepsilon\right)  \ $of the set $D^{\ast}\left(  \varepsilon\right)  $ for
any $\varepsilon\in\left[  0,1\right)  .$ The crucial point of the proof of
Theorem \ref{TRM-main-manje} is to show that the set inclusions $B\left(
\varepsilon\right)  \subset\mathcal{D}_{2,\varepsilon}^{0}\subset D\left(
\varepsilon\right)  $ hold for all $\varepsilon\in\left[  0,1\right)  .$ In
Section \ref{Sect3}, we prove the set inclusions $B\left(  \varepsilon\right)
\subset\mathcal{D}_{2,\varepsilon}^{0}\subset D\left(  \varepsilon\right)  $
for all $\varepsilon\in\left[  0,1\right)  $ using a powerful algorithm (given
in Lemma \ref{lem-alg}) and extensive hand calculations. The set inclusion
$\mathcal{D}_{2,\varepsilon}^{0}\subset D\left(  \varepsilon\right)  $ is easy
to prove, but it turns out to be a hard problem to prove that some parts of
the set $B\left(  \varepsilon\right)  $ belong to $\mathcal{D}_{2,\varepsilon
}^{0}.$ For example, the problem for $\varepsilon\in\left(  0,\frac{1}%
{2}\right)  $ becomes more and more harder the closer $\varepsilon$ is to $0$
and the closer we get to the line $y=x+1-\varepsilon$. Section \ref{Sect4} is
devoted to the proof of Theorem \ref{TRM-main-manje}. First, we appropriately
rewrite the interval $[0,1)$ as a disjoint union of the subintervals which
allows us to characterize all pairs $\left(  p_{0},p_{1}\right)  \in
\mathbb{Z}^{2}$ with the properties $\left(  \frac{1}{p_{0}},\frac{p_{1}%
}{p_{0}}\right)  \in D\left(  \varepsilon\right)  $ and $\left(  \frac
{1}{p_{0}},\frac{p_{1}}{p_{0}}\right)  \in B\left(  \varepsilon\right)  ,$
depending on which subinterval $\varepsilon$ belongs to. This characterization
and a combination of $B\left(  \varepsilon\right)  \subset\mathcal{D}%
_{2,\varepsilon}^{0}\subset D\left(  \varepsilon\right)  $ and Theorem
\ref{TRM-veza} yield the proof of Theorem \ref{TRM-main-manje}.

\section{Main theorem and basic idea of its proof\label{Sect2}}

The characterization of classical quadratic\textbf{\ }CNS polynomials is
already given in \cite{Katai-Kovacs}, \cite{Katai-Kovacs1}, \cite{Gilbert},
\cite{Thuswaldner} and \cite{Akiyama-Rao} in several ways. Namely, we have:

\begin{theorem}
\label{TRM-car-0}\textit{Let }$P(x)=x^{2}+p_{1}x+p_{0}\in\mathbb{Z}\left[
x\right]  $. \textit{Then }$P$\textit{\ is a }$0-$CNS polynomial if and only
if\textit{\ }$-1\leq p_{1}\leq p_{0}\ $and\ $p_{0}\geq2$.
\end{theorem}

\noindent Akiyama and Scheicher \cite{Akiyama-Scheicher} gave a complete
description of the set $D_{2,\frac{1}{2}}^{0}.$

\begin{proposition}
[{\cite[Theorem 5.2]{Akiyama-Scheicher}}]\label{SRSpola}%
\begin{align*}
\mathcal{D}_{2,\frac{1}{2}}^{0}  &  =\left\{  \left(  {\small x,y}\right)
{\small \in}\mathbb{R}^{2}:\left\vert x\right\vert <\frac{1}{2},\text{
}{\small -x-}\frac{1}{2}{\small <y\leq x+}\frac{1}{2}\right\} \\
&  {\small \cup}\left\{  \left(  \frac{1}{2}{\small ,y}\right)  \in
\mathbb{R}^{2}:{\small -1<y\leq}\frac{1}{2}\text{ or }{\small y=1}\right\}  .
\end{align*}

\end{proposition}

As a consequence of Proposition \ref{SRSpola} and Theorem \ref{TRM-veza} for
$\left(  d,\varepsilon\right)  =\left(  2,\frac{1}{2}\right)  $, Akiyama and
Scheicher \cite[Theorem 2.2]{Akiyama-Scheicher} obtained the characterization
of quadratic\textbf{\ }SCNS polynomials. The statement of the following
corollary is a slight modification of the statement of \cite[Theorem
2.2]{Akiyama-Scheicher} adapted to our needs.

\begin{corollary}
\label{Corr-1}\textit{Let }$P\left(  x\right)  =x^{2}+p_{1}x+p_{0}%
\in\mathbb{Z}\left[  x\right]  $. \textit{Then }$P$\textit{\ is a }$\frac
{1}{2}\mathcal{-}$CNS polynomial if and only if
\begin{gather}
\left\vert p_{1}\right\vert \leq sgn\left(  p_{0}\right)  +\frac{\left\vert
p_{0}\right\vert -1}{2}\text{, }\left\vert p_{0}\right\vert \geq3\text{, when
}\left\vert p_{0}\right\vert \text{ is odd,}\label{car-1/2-odd}\\
\left\vert {\small p}_{1}\right\vert {\small \leq sgn}\left(  {\small p}%
_{0}\right)  {\small -1+}\frac{\left\vert {\small p}_{0}\right\vert
}{{\small 2}}\text{ \ or}\ \ {\small p}_{1}{\small =1+}\frac{{\small p}_{0}%
}{{\small 2}}{\small ,}\text{ }\left\vert {\small p}_{0}\right\vert
{\small \geq4}\text{, when }\left\vert {\small p}_{0}\right\vert
{\small >2}\text{ is even}\label{car-1/2-even}\\
-1\leq p_{1}\leq2\text{, }p_{0}=2\text{, when }\left\vert p_{0}\right\vert =2.
\label{car-1/2-2}%
\end{gather}

\end{corollary}

\begin{proof}
It is easy to see that Theorem \ref{TRM-veza} and Proposition \ref{SRSpola}
imply: $P$\textit{\ }is a\textit{\ }$\frac{1}{2}\mathcal{-}$CNS polynomial if
and only if
\begin{align*}
\left\vert p_{1}\right\vert  &  <\operatorname{sgn}\left(  p_{0}\right)
+\frac{\left\vert p_{0}\right\vert }{2}\text{ or \ }p_{1}=1+\frac{p_{0}}%
{2}\text{, when }\left\vert p_{0}\right\vert >2\text{ }\\
-1  &  \leq p_{1}\leq2,\text{ when \ }p_{0}=2.
\end{align*}
If $\left\vert p_{0}\right\vert \geq3$ is odd, then $1+\frac{p_{0}}{2}%
\notin\mathbb{Z}$ and $\left\vert p_{1}\right\vert <sgn\left(  p_{0}\right)
+\frac{\left\vert p_{0}\right\vert }{2}$ is equivalent with $\left\vert
p_{1}\right\vert \leq sgn\left(  p_{0}\right)  +\frac{\left\vert
p_{0}\right\vert -1}{2}.$ If $\left\vert p_{0}\right\vert \geq4$ is even, then
$\left\vert p_{1}\right\vert <sgn\left(  p_{0}\right)  +\frac{\left\vert
p_{0}\right\vert }{2}$ is equivalent with $\left\vert p_{1}\right\vert \leq
sgn\left(  p_{0}\right)  -1+\frac{\left\vert p_{0}\right\vert }{2}.$
\end{proof}

\noindent Our main result is given in the following theorem.

\begin{theorem}
\label{TRM-main-manje}\textit{Let }$P\left(  x\right)  =x^{2}+p_{1}x+p_{0}%
\in\mathbb{Z}\left[  x\right]  $, $\left\vert p_{0}\right\vert \geq2\ $and
$\varepsilon\in\left[  0,1\right)  .$ Let $k=\left\lfloor \mathcal{\varepsilon
}\left\vert p_{0}\right\vert \right\rfloor $. Then corresponding
$\varepsilon-$set of digits is $\mathcal{N}_{\varepsilon}=\left\{
-k,...,\left\vert p_{0}\right\vert -1-k\right\}  .$

\begin{enumerate}
\item[\textbf{i)}] Let $\varepsilon\in\left[  0,\frac{1}{2}\right)  $ or let
$\varepsilon=\frac{1}{2}$ when $\left\vert p_{0}\right\vert $ is odd.
\textit{Then }$P$\textit{\ is a }$\mathcal{\varepsilon-}$CNS polynomial if and
only if
\begin{equation}
-k-1\leq p_{1}\leq\left\vert p_{0}\right\vert -k,\text{ \ }p_{0}\geq
2\text{,}\ \text{for }\varepsilon\in\left[  0,\frac{1}{2}\right]  ,
\label{prva1}%
\end{equation}
or%
\begin{equation}
k+2-\left\vert p_{0}\right\vert \leq p_{1}\leq k-1,\text{ \ }p_{0}%
\leq-3,\ \ \text{when \ }\varepsilon\in\left[  \frac{1}{\left\vert
p_{0}\right\vert },\frac{1}{2}\right]  . \label{druga 1}%
\end{equation}

\item[\textbf{ii)}] Let $\varepsilon\in\left(  \frac{1}{2},1\right)  $ or let
$\varepsilon=\frac{1}{2}$ when $\left\vert p_{0}\right\vert $ is even.
\textit{Then }$P$\textit{\ is a }$\mathcal{\varepsilon-}$CNS polynomial if and
only if
\begin{equation}
-\left\vert p_{0}\right\vert +k\leq p_{1}\leq k+1,\text{ \ }p_{0}\geq2\text{,
}\ \ \text{for }\varepsilon\in\left[  \frac{1}{2},1\right)  , \label{prva2}%
\end{equation}
or%
\begin{equation}
-k+1\leq p_{1}\leq-k-2+\left\vert p_{0}\right\vert ,\text{ \ }p_{0}%
\leq-3,\ \ \text{when \ }\varepsilon\in\left[  \frac{1}{2},\frac{\left\vert
p_{0}\right\vert -1}{\left\vert p_{0}\right\vert }\right)  . \label{druga2}%
\end{equation}

\end{enumerate}
\end{theorem}

\begin{remark}
\textit{Let }$P\left(  x\right)  =x^{2}+p_{1}x+p_{0}\in\mathbb{Z}\left[
x\right]  $, $\left\vert p_{0}\right\vert \geq2.$

a) If $\varepsilon\in\left[  0,\frac{1}{\left\vert p_{0}\right\vert }\right)
,$ then $k=\left\lfloor \mathcal{\varepsilon}\left\vert p_{0}\right\vert
\right\rfloor =0.$ Thus $\mathcal{N}_{\varepsilon}=\left\{  0,...,\left\vert
p_{0}\right\vert -1\right\}  .$ Since $\varepsilon\in\left[  0,\frac{1}%
{2}\right)  $ and $\varepsilon\notin\left[  \frac{1}{\left\vert p_{0}%
\right\vert },\frac{1}{2}\right]  $, by Theorem \ref{TRM-main-manje}, i) we
obtain: $P$ is a $\varepsilon$-CNS polynomial if and only if $-1\leq p_{1}\leq
p_{0},$ \ $p_{0}\geq2$. In particular, for $\varepsilon=0$, we obtain the
statement of Theorem \ref{TRM-car-0}.

b) If $\varepsilon=\frac{1}{2}$, then%
\begin{equation}
k=\left\lfloor \mathcal{\varepsilon}\left\vert p_{0}\right\vert \right\rfloor
=\left\{
\begin{tabular}
[c]{l}%
$\frac{\left\vert p_{0}\right\vert }{2}\text{ if }\left\vert p_{0}\right\vert
\text{ is even}$\\
$\frac{\left\vert p_{0}\right\vert -1}{2}\text{ if }\left\vert p_{0}%
\right\vert \text{ is odd}$%
\end{tabular}
\right.  . \label{k-pola}%
\end{equation}
Thus, set of $\frac{1}{2}-$digits is
\[
\mathcal{N}_{\frac{1}{2}}=\left\{
\begin{tabular}
[c]{l}%
$\left\{  -\frac{\left\vert p_{0}\right\vert }{2},...,\frac{\left\vert
p_{0}\right\vert }{2}-1\right\}  ,$ $\text{ if }\left\vert p_{0}\right\vert
\text{ is even}$\\
$\left\{  -\frac{\left\vert p_{0}\right\vert -1}{2},...,\frac{\left\vert
p_{0}\right\vert -1}{2}\right\}  ,\text{ if }\left\vert p_{0}\right\vert
\text{ is }$odd
\end{tabular}
\right.  .
\]
If $\varepsilon=\frac{1}{2}$ and $\left\vert p_{0}\right\vert $ is odd, then,
by Theorem \ref{TRM-main-manje}, i) and (\ref{k-pola}), we easily derive: $P$
is $\frac{1}{2}$-CNS polynomial if and only if (\ref{car-1/2-odd}) holds. If
$\varepsilon=\frac{1}{2}$ and $\left\vert p_{0}\right\vert $ is even, then
Theorem \ref{TRM-main-manje}, ii) and (\ref{k-pola}) imply: $P$ is $\frac
{1}{2}$-CNS if and only if (\ref{car-1/2-even}) or (\ref{car-1/2-2}) holds.
Therefore, Theorem \ref{TRM-main-manje} implies Corollary \ref{Corr-1}.

c) If $\varepsilon\in\left[  \frac{\left\vert p_{0}\right\vert -1}{\left\vert
p_{0}\right\vert },1\right)  ,$ then $k=\left\vert p_{0}\right\vert -1$ and
corresponding $\varepsilon-$set of digits is $\mathcal{N}_{\varepsilon
}=\left\{  -\left\vert p_{0}\right\vert +1,...,0\right\}  .$ In this case
Theorem \ref{TRM-main-manje}, ii) implies: $P$ is $\varepsilon$-CNS polynomial
if and only if $-1\leq p_{1}\leq p_{0},$ \ $p_{0}\geq2$.

d) Note that if $(x^{2}+p_{1}x+p_{0},\mathcal{N}_{\varepsilon})$ is
$\varepsilon$-CNS with corresponding $\varepsilon-$set of digits
$\mathcal{N}_{\varepsilon}=\left\{  0,...,\left\vert p_{0}\right\vert
-1\right\}  $ or $\left\{  -\left\vert p_{0}\right\vert +1,...,0\right\}  ,$
than constant term $p_{0}$ of $\varepsilon$-CNS basis $P$ has to be positive
integer $\geq2$. In all other cases, i.e. when $\left\{  -1,0,1\right\}
\subseteq\mathcal{N}_{\varepsilon}=\left\{  -k,...,\left\vert p_{0}\right\vert
-1-k\right\}  ,$ constant term $p_{0}$ can be any positive or negative integer
with $\left\vert p_{0}\right\vert \geq2\ $and $p_{0}\neq-2.$
\end{remark}

\begin{remark}
\label{rem-PT}Note that for any monic quadratic polynomial $P$, $P\left(
m\right)  \geq2$ holds if $\left\vert m\right\vert $ is sufficiently large.
Therefore, it's easy to see that Theorem \ref{TRM-main-manje} (more precisely
(\ref{prva1}) and (\ref{prva2})) implies that if $\varepsilon\in\left(
0,1\right)  ,$ then $P\left(  x\pm m\right)  $ is $\varepsilon$-CNS polynomial
for all integers $m\ $large enough, while if $\varepsilon=0,$ then $P\left(
x+m\right)  $ is, up to finitely many possible exceptions, a $0$-CNS
polynomial if and only if $m>M$, where $M$ denotes a constant. Since
$\varepsilon$-CNS are a special case of \textit{generalized number
systems\textit{ over orders }}(GNS for short)\textit{ }defined by Peth\H{o}
and Thuswaldner \cite{Petho-Thuswaldner}, this fits into the general results
on GNS obtained in \cite{Petho-Thuswaldner} (see \cite[Theorem 4.1, Corollary
4.3. and Theorem 5.2.]{Petho-Thuswaldner}). Moreover, for any given monic
quadratic polynomial $P$, using Theorem \ref{TRM-main-manje}, the integers $m$
for which the polynomials $P\left(  x+m\right)  $ are $\varepsilon$-CNS
polynomials can be all effectively determined.
\end{remark}

We reprove Theorem \ref{TRM-car-0} as a contribution to the understanding of
the basic idea of the proof of Theorem \ref{TRM-main-manje}. By \cite[Example
4.7]{Akiyama-BBPT} and \cite[Corollary 2.5]{Akiyama-BBPT-II} we have
$\mathcal{D}_{2,0}^{0}\subset D\left(  0\right)  $, where $D\left(  0\right)
$ is the trapezium%
\begin{equation}
D\left(  0\right)  :=\left\{  \left(  x,y\right)  \in\mathbb{R}^{2}:-x\leq
y<x+1,\text{ \ }0\leq x<1\text{ }\right\}  . \label{D0*}%
\end{equation}
Further, Akiyama et al. \cite{Akiyama-BBPT-II} showed that set $\mathcal{D}%
_{2,0}^{0}$ has a very simple structure if $x\leq2/3$ and they completely
characterized $\mathcal{D}_{2,0}^{0}$ in that region. Precisely, we have:

\begin{proposition}
[{\cite[Theorem 5.4]{Akiyama-BBPT-II}}]\label{TRM-car-0 3-2}
\[
\mathcal{D}_{2,0}^{0}\cap R\left(  0\right)  =\left\{  \left(  x,y\right)
\in\mathbb{R}^{2}:-x\leq y<x+1,\text{ \ }0\leq x\leq\frac{2}{3}\right\}
=:B\left(  0\right)  ,
\]
where $R\left(  0\right)  $ is the half-plane $R\left(  0\right)  =\left\{
\left(  x,y\right)  \in\mathbb{R}^{2}:x\leq\frac{2}{3}\right\}  .$
\end{proposition}

Therefore, we have
\begin{equation}
B\left(  0\right)  \subset\mathcal{D}_{2,0}^{0}\subset D\left(  0\right)  .
\label{inc-0}%
\end{equation}
Using (\ref{inc-0}) and Theorem \ref{TRM-veza} we are able to reprove Theorem
\ref{TRM-car-0}.\smallskip

\begin{proof}
[Proof of Theorem \ref{TRM-car-0}]Let $-1\leq p_{1}\leq p_{0}\ $%
and\ $p_{0}\geq2.$ Then $-\frac{1}{p_{0}}\leq\frac{p_{1}}{p_{0}}\leq1<\frac
{1}{p_{0}}+1$ and $0<\frac{1}{p_{0}}\leq\frac{1}{2}<\frac{2}{3}$. Therefore
$\left(  \frac{1}{p_{0}},\frac{p_{1}}{p_{0}}\right)  \in B\left(  0\right)
\subset\mathcal{D}_{2,0}^{0},$ and by Theorem \ref{TRM-veza}, we conclude that
$P$ is a $0-$CNS polynomial.

Let $P$ be a $0-$CNS polynomial. Then, by Theorem \ref{TRM-veza}, we have
$\left(  \frac{1}{p_{0}},\frac{p_{1}}{p_{0}}\right)  \in\mathcal{D}_{2,0}%
^{0}.$ Since $\mathcal{D}_{2,0}^{0}\subset D\left(  0\right)  ,$ we have
$\left(  \frac{1}{p_{0}},\frac{p_{1}}{p_{0}}\right)  \in D\left(  0\right)  $,
which implies $0\leq\frac{1}{p_{0}}<1\ $and $-\frac{1}{p_{0}}\leq\frac{p_{1}%
}{p_{0}}<\frac{1}{p_{0}}+1.$ Thus, $p_{0}>1$ and $-1\leq p_{1}<p_{0}+1,$ or
equivalently $-1\leq p_{1}\leq p_{0}$ and $p_{0}\geq2.\smallskip$
\end{proof}

Note that it was possible to provide an alternative proof of Theorem
\ref{TRM-car-0} using the set inclusions (\ref{inc-0}) and Theorem
\ref{TRM-veza}, since the following equivalences hold for $p_{0},p_{1}%
\in\mathbb{Z}$:
\[
\left(  \frac{1}{{\small p}_{0}}{\small ,}\frac{{\small p}_{1}}{{\small p}%
_{0}}\right)  {\small \in D}\left(  {\small 0}\right)  \text{ \ }%
\Longleftrightarrow\text{ \ }{\small -1\leq p}_{1}{\small \leq p}%
_{0}{\small ,\ p}_{0}{\small \geq2}\text{ \ }\Longleftrightarrow\text{
\ }\left(  \frac{1}{{\small p}_{0}}{\small ,}\frac{{\small p}_{1}}%
{{\small p}_{0}}\right)  {\small \in B}\left(  0\right)  .
\]
Now we will generalize this idea for each $\varepsilon\in\left[  0,1\right)
.$ Let $\varepsilon\in\left[  0,1\right)  $ and let us define the set
$D^{\ast}\left(  \varepsilon\right)  $ by%
\begin{equation}
D^{\ast}(\varepsilon):={\footnotesize {%
\begin{cases}
\left\{  \left(  x,y\right)  \in\mathbb{R}^{2}:-x-\varepsilon\leq
y<x+1-\varepsilon,\ \ x<1-\varepsilon\right\}  ,\text{ for }\varepsilon
\in\left[  0,\frac{1}{2}\right) \\
\left\{  \left(  x,y\right)  \in\mathbb{R}^{2}:-x-1+\varepsilon<y\leq
x+\varepsilon,\ \ x\leq\varepsilon\right\}  ,\text{ for }\varepsilon\in\left[
\frac{1}{2},1\right)
\end{cases}
}} \label{D*}%
\end{equation}
and half-planes $L\left(  \varepsilon\right)  $ and $R\left(  \varepsilon
\right)  $ by
\[
L(\varepsilon):={\small {%
\begin{cases}
\left\{  (x,y)\in\mathbb{R}^{2}:-\varepsilon\leq x\right\}  ,\text{ for
}\varepsilon\in\left[  0,\frac{1}{2}\right) \\
\left\{  (x,y)\in\mathbb{R}^{2}:-\left(  1-\varepsilon\right)  <x\right\}
,\text{ for }\varepsilon\in\left[  \frac{1}{2},1\right)
\end{cases}
{\footnotesize \!\!\!\!\!},}}%
\]%
\[
R(\varepsilon):={\small {%
\begin{cases}
\left\{  (x,y)\in\mathbb{R}^{2}:x\leq\frac{2}{3}-\varepsilon\right\}  ,\text{
for }\varepsilon\in\left[  0,\frac{1}{2}\right) \\
\left\{  (x,y)\in\mathbb{R}^{2}:x\leq\frac{2}{3}-\left(  1-\varepsilon\right)
\right\}  ,\text{ for }\varepsilon\in\left[  \frac{1}{2},1\right)
\end{cases}
{\footnotesize \!\!\!\!\!}.}}%
\]
We know that%
\begin{equation}
\mathcal{D}_{2,0}^{0}\subset D^{\ast}\left(  \varepsilon\right)  \text{ for
all }\varepsilon\in\left[  0,1\right)  . \label{D*-sur}%
\end{equation}
Namely, Surer \cite[Section 5 and 6]{Surer} have shown that (\ref{D*-sur})
holds for all $\varepsilon\in\left(  0,1\right)  \diagdown\left\{  \frac{1}%
{2}\right\}  $. For $\varepsilon=0$ it is easy to see that (\ref{D*-sur})
holds since $\mathcal{D}_{2,0}^{0}\subset D\left(  0\right)  $ and $D\left(
0\right)  \subset D^{\ast}\left(  0\right)  ,$ where $D\left(  0\right)  $ is
a trapezium given by (\ref{D0*}). For $\varepsilon=\frac{1}{2}$ the set
inclusion (\ref{D*-sur}) follows directly from Proposition \ref{SRSpola}.

We noticed that (\ref{inc-0}), Proposition \ref{SRSpola} and examples in
\cite[Section 6]{Surer} indicate that for all $\varepsilon\in\left[
0,1\right)  $ "$\varepsilon-$SRS region" $\mathcal{D}_{2,\varepsilon}^{0}$
should be contained in the half-plane $L\left(  \varepsilon\right)  ,$ and
that sets $\mathcal{D}_{2,\varepsilon}^{0}$ and $D^{\ast}\left(
\varepsilon\right)  $ should coincide in the stripe $S\left(  \varepsilon
\right)  :=L\left(  \varepsilon\right)  \cap R\left(  \varepsilon\right)  .$
Let
\begin{equation}
D\left(  \varepsilon\right)  :=D^{\ast}\left(  \varepsilon\right)  \cap
L\left(  \varepsilon\right)  \ \ \ \text{and}\ \ \ B\left(  \varepsilon
\right)  :=D^{\ast}\left(  \varepsilon\right)  \cap S\left(  \varepsilon
\right)  . \label{def-D-B}%
\end{equation}
Therefore, since (\ref{D*-sur}) holds, to prove our conjectures it suffices to
prove that
\begin{equation}
B\left(  \varepsilon\right)  \subset\mathcal{D}_{2,\varepsilon}^{0}\subset
D\left(  \varepsilon\right)  \label{incl-main}%
\end{equation}
holds for all $\varepsilon\in\left[  0,1\right)  .$ Note that the definitions
of the sets $B\left(  \varepsilon\right)  $ and $D\left(  \varepsilon\right)
$ given by (\ref{def-D-B})$,$ imply that $B\left(  \varepsilon\right)
=D\left(  \varepsilon\right)  \cap R\left(  \varepsilon\right)  $ for all
$\varepsilon\in\left[  0,1\right)  ,$ which means that these two sets coincide
in the half-plane $R\left(  \varepsilon\right)  .$ More precisely, we have%
\begin{equation}
{\footnotesize {D(\varepsilon)=%
\begin{cases}
\left\{  (x,y)\in\mathbb{R}^{2}:-x-\varepsilon\leq y<x+1-\varepsilon
,\ \ -\varepsilon\leq x<1-\varepsilon\right\}  ,\text{ if }\varepsilon
\in\left[  0,\frac{1}{2}\right) \\
\left\{  (x,y)\in\mathbb{R}^{2}:-x-1+\varepsilon<y\leq x+\varepsilon
,\ -\left(  1-\varepsilon\right)  <x\leq\varepsilon\right\}  ,\text{ if
}\varepsilon\in\left[  \frac{1}{2},1\right)
\end{cases}
\!\!\!\!\!,}} \label{D}%
\end{equation}%
\begin{equation}
{\footnotesize {B(\varepsilon)=\!\!%
\begin{cases}
\left\{  \left(  x,y\right)  \in\mathbb{R}^{2}:-x-\varepsilon\leq
y<x+1-\varepsilon,\ -\varepsilon\leq x\leq\frac{2}{3}-\varepsilon\right\}
,\text{ if }\varepsilon\in\left[  0,\frac{1}{2}\right) \\
\left\{  \left(  x,y\right)  \in\mathbb{R}^{2}:-x-1+\varepsilon<y\leq
x+\varepsilon,\ -\left(  1-\varepsilon\right)  <x\leq\frac{2}{3}-\left(
1-\varepsilon\right)  \right\}  ,\text{ if }\varepsilon\in\left[  \frac{1}%
{2},1\right)
\end{cases}
\!\!\!\!\!.}} \label{B}%
\end{equation}

The sets $B(\varepsilon)$ and $D(\varepsilon)$ for each $\varepsilon\in\left[
0,1\right)  $ are defined by the same inequalities, with the exception of the
upper bound for $x$. Consequently, if $\left(  \frac{1}{p_{0}},\frac{p_{1}%
}{p_{0}}\right)  \in D\left(  \varepsilon\right)  ,$ then $\left(  \frac
{1}{p_{0}},\frac{p_{1}}{p_{0}}\right)  \in B\left(  \varepsilon\right)  $
except for finitely many possible positive small values of $p_{0}.$ So, if we
prove that (\ref{incl-main}) holds and characterize all $\left(  p_{0}%
,p_{1}\right)  $, $\left\vert p_{0}\right\vert \geq2$ such that $\left(
\frac{1}{p_{0}},\frac{p_{1}}{p_{0}}\right)  \in D\left(  \varepsilon\right)
,$ then using (\ref{incl-main}) and Theorem \ref{TRM-main-manje}, similarly as
in case $\varepsilon=0,$ we can characterize all $\varepsilon-$CNS polynomials
except for finitely many possible $p_{0}$.

\section{Complete characterization of set $\mathcal{D}_{2,\varepsilon}^{0}%
$\ in half-plane $R\left(  \varepsilon\right)  $\label{Sect3}}

The aim of this section is to prove the following theorem.

\begin{theorem}
\label{Trm-inclus}Let $\varepsilon\in\left[  0,1\right)  \ $and let sets
$D\left(  \varepsilon\right)  $ and $B\left(  \varepsilon\right)  $ be defined
by (\ref{D}) and (\ref{B}), respectively. Than (\ref{incl-main}) holds.
\end{theorem}

We know that (\ref{incl-main}) holds$\ $for $\varepsilon=0$ (see
(\ref{inc-0})), and for $\varepsilon=\frac{1}{2}$ the set inclusions
(\ref{incl-main}) easily follow from Proposition \ref{SRSpola}. Thus, we have
to prove that (\ref{incl-main}) holds for every $\varepsilon\in\left(
0,1\right)  \diagdown\left\{  \frac{1}{2}\right\}  .$ Also note that if
Theorem \ref{Trm-inclus} holds, then the definitions of the sets $D\left(
\varepsilon\right)  $ and $B\left(  \varepsilon\right)  $ imply $\mathcal{D}%
_{2,\varepsilon}^{0}\cap R\left(  \varepsilon\right)  =B\left(  \varepsilon
\right)  ,$ which means that we completely characterized set $\mathcal{D}%
_{2,\varepsilon}^{0}$ in the half-plane $R\left(  \varepsilon\right)  $ for
every $\varepsilon\in\left[  0,1\right)  .$ Before proving Theorem
\ref{Trm-inclus}, we give a brief overview of the notations and results that
we will use in the proof (for more details see \cite[Section 2]{Surer}).

In order to analyze the structure of the $\varepsilon-$SRS region
$\mathcal{D}_{d,\varepsilon}^{0}$, we start with $\mathcal{D}_{d,\varepsilon}$
and have to remove all points $\mathbf{r}$ where $\left(  \tau_{\mathbf{r}%
,\varepsilon}^{n}(\mathbf{z})\right)  _{n\in\mathbb{N}}$ is periodic for some
$\mathbf{z\in}\mathbb{Z}^{d},$ $\mathbf{z\neq0}.$ In particular,
$\mathbf{r}\in\mathcal{D}_{d,\varepsilon}\diagdown\mathcal{D}_{d,\varepsilon
}^{0}$ when there exist nonzero points $\mathbf{z}_{1},...,\mathbf{z}_{l}%
\in\mathbb{Z}^{d}$ with $\tau_{\mathbf{r},\varepsilon}(\mathbf{z}%
_{i})=\mathbf{z}_{i+1}$ for $i=1,...,l-1$ and $\tau_{\mathbf{r},\varepsilon
}(\mathbf{z}_{l})=\mathbf{z}_{1}.$ Since the mapping $\tau_{\mathbf{r}%
,\varepsilon}$ keeps $d-1$ entries unchanged (see (\ref{tau-eps})), it
suffices to write down only the first entry of each vector $\mathbf{z}%
_{1},...,\mathbf{z}_{l}.$ Therefore, if these vectors are given by
$\mathbf{z}_{i}=\left(  z_{i},...,z_{i+d-1}\right)  ,$ $i=1,...,l$, where the
indices have to be taken modulo $l,$ then these $l$ vectors can be represented
by $l$ integers $z_{1},...,z_{l}.$ We refer to such a sequence of integers as
a \textit{cycle of} $\tau_{\mathbf{r},\varepsilon}$ \textit{of period }$l$ and
we write it in the form $\pi=\left(  z_{1},...,z_{l}\right)  $. By definition
of the mapping $\tau_{\mathbf{r},\varepsilon}$, it is easy to see that a cycle
$\pi=\left(  z_{1},...,z_{l}\right)  \in\mathbb{Z}^{l}$ is a cycle of
$\tau_{\mathbf{r},\varepsilon}$ if and only if%
\begin{equation}
0\leq r_{1}z_{1+i}+r_{2}z_{2+i}+...+r_{d}z_{d+i}+z_{d+1+i}+\varepsilon
<1,\ \forall i\in\left\{  0,...,l-1\right\}  , \label{P-eps}%
\end{equation}
where $z_{k+i}=z_{\left(  k+i\right)  \operatorname{mod}l}$. The finite set of
inequalities (\ref{P-eps}) determines a (possibly degenerate or empty)
polyhedron $P_{\varepsilon}\left(  \pi\right)  \subset\mathbb{R}^{d}$. We call
this polyhedron the \textit{cutout polyhedron} defined by the cycle $\pi$.
Since $\mathbf{r}\in\mathcal{D}_{d,\varepsilon}^{0}$ if and only if
$\tau_{\mathbf{r},\varepsilon}$ has $(0)$ as its only period, we conclude
that
\begin{equation}
\mathcal{D}_{d,\varepsilon}^{0}=\mathcal{D}_{d,\varepsilon}\diagdown%
{\textstyle\bigcup\limits_{\pi\in\Pi}}
P_{\varepsilon}\left(  \pi\right)  , \label{poly}%
\end{equation}
where $\Pi$ is the set of all integer vectors $\left(  z_{1},...,z_{l}\right)
$ of finite dimension $l$ such that the word $z_{1}\cdots z_{l}\in N$, where
$N$ is a set of representatives of the primitive necklaces of finite length
without the word $0.$ The problem is that this representation is not very
practicable, since $\Pi$ is an infinite set. But (\ref{poly}) shows that it is
much easier to recognize a subsets of $\mathcal{D}_{d,\varepsilon}$ not to be
a subset of $\mathcal{D}_{d,\varepsilon}^{0}$ than to be a subset of
$\mathcal{D}_{d,\varepsilon}^{0}$. There is an algorithm that uses the idea of
the \textquotedblleft set of witnesses\textquotedblright, invented
independently by Brunotte \cite{Brunotte} and Scheicher and Thuswaldner
\cite{Scheicher-Thuswaldner}, which allows one to decide whether a given
vector $\mathbf{r}$ is an element of $\mathcal{D}_{d,\varepsilon}^{0}.$ Using
the "convexity property" of $\tau_{\mathbf{r},\varepsilon}$, it was shown that
a similar algorithm works even for a given closed and convex set, for example
the convex hull $\mathcal{H}$ of $\mathbf{r}_{1},...,\mathbf{r}_{k}%
\in\mathbb{R}^{d},$ which lies in $\operatorname*{Int}\mathcal{D}%
_{d,\varepsilon}=\mathcal{E}_{d}\ $(see \cite[Section 2]{Surer} and
\cite[Lemma 3.2]{Akiyama-BBPT-II}).

\begin{lemma}
\label{lem-alg}Let $\varepsilon\in\left[  0,1\right)  $ and $\mathbf{r}%
_{1},...,\mathbf{r}_{k}$ be points of $\mathcal{D}_{d,\varepsilon}$. Denote by
$\mathcal{H}$ the convex hull of $\mathbf{r}_{1},...,\mathbf{r}_{k}$. We
assume that $\mathcal{H}$ is contained in $\mathcal{E}_{d}$. If the diameter
of $\mathcal{H}$ is sufficiently small, then there exists an algorithm to
create a \textit{set of witnesses} $\mathcal{V}\left(  \mathcal{H}\right)
\subset\mathbb{Z}^{d}$ for $\mathcal{H}\ $and an algorithm to create a finite
\textit{directed graph} $G_{\varepsilon}(\mathcal{H}):=\mathcal{V}\left(
\mathcal{H}\right)  \times E_{\varepsilon}$ with set of vertices
$\mathcal{V}\left(  \mathcal{H}\right)  $ and set of edges $E_{\varepsilon
}\subset\mathcal{V}\left(  \mathcal{H}\right)  \times\mathcal{V}\left(
\mathcal{H}\right)  $ having the following properties:

\begin{enumerate}
\item[1)] $\pm\mathbf{e}_{1},...,\pm\mathbf{e}_{d}\in\mathcal{V}\left(
\mathcal{H}\right)  $, where $\mathbf{e}_{i}$ is a $d-$dimensional standard
unit vector $\mathbf{e}_{i}=(0,...,1,0,...,0).$

\item[2)] For each $\mathbf{z}=\left(  z_{1},z_{2},...,z_{d}\right)
\in\mathcal{V}\left(  \mathcal{H}\right)  $, then $\left(  z_{2}%
,...,z_{d},z_{d+1}\right)  \in\mathcal{V}\left(  \mathcal{H}\right)  $ if and
only if%
\[
z_{d+1}\in\lbrack\min_{1\leq i\leq k}\{\lfloor-\mathbf{r}_{i}\mathbf{z}%
\rfloor\},-\min_{1\leq i\leq k}\{\lfloor\mathbf{r}_{i}\mathbf{z}\rfloor
\}]\cap\mathbb{Z}.
\]

\item[3)] Let $\mathbf{z,z}^{\prime}\in\mathcal{V}\left(  \mathcal{H}\right)
,$ where $\mathbf{z}=\left(  z_{1},z_{2},...,z_{d}\right)  $. We put a
\textit{direct }edge $\left(  \mathbf{z,z}^{\prime}\right)  \in E_{\varepsilon
}$ (or $\mathbf{z\rightarrow z}^{\prime})$ if and only if $\mathbf{z}^{\prime
}=\left(  z_{2},...,z_{d},z_{d+1}\right)  $ and%
\begin{equation}
z_{d+1}\in\lbrack-\max_{1\leq i\leq k}\{\lfloor\mathbf{r}_{i}\mathbf{z}%
+\varepsilon\rfloor\},-\min_{1\leq i\leq k}\{\lfloor\mathbf{r}_{i}%
\mathbf{z}+\varepsilon\rfloor\}]\cap\mathbb{Z}. \label{edg}%
\end{equation}

\end{enumerate}
\end{lemma}

Note that the definition of the graph $G_{\varepsilon}(\mathcal{H})$ is
meaningful because a set of witnesses $\mathcal{V}\left(  \mathcal{H}\right)
$ of $\mathcal{H}$ is closed under the application of $\tau_{\mathbf{r}%
,\varepsilon}$ (see \cite[Lemma 2.9]{Surer}). If $\mathcal{H}$ is small
enough, the procedure will terminate, yielding finite set $\mathcal{V}\left(
\mathcal{H}\right)  $. The finiteness condition on $\mathcal{V}\left(
\mathcal{H}\right)  $ assures the finiteness of the graph $G_{\varepsilon
}(\mathcal{H})$. We are interested in the (directed) cycles of the graph
$G_{\varepsilon}(\mathcal{H})$. By the definition of the edges of the graph
$G_{\varepsilon}(\mathcal{H}),$ the graph-cycle $\mathbf{z}_{1}%
\mathbf{\rightarrow z}_{2}\mathbf{\rightarrow}...\mathbf{\rightarrow z}%
_{l}\mathbf{\rightarrow z}_{1}\ $is uniquely determined by the $l$ integers
that form the first entries of the vectors $\mathbf{z}_{1}\mathbf{,...,z}_{l}%
$. Thus, the graph-cycles, similar as cycles of $\tau_{\mathbf{r},\varepsilon
}$, can also be written as $\pi=\left(  z_{1},...,z_{l}\right)  \in
\mathbb{Z}^{l}.$ The graph-cycles and the cycles of $\tau_{\mathbf{r}%
,\varepsilon}$ for $\mathbf{r\in}\mathcal{H}$ are closely related. More
precisely, we have

\begin{proposition}
[{\cite[Theorem 2.11]{Surer}}]\label{prop-conv-H}Let $\varepsilon\in\left[
0,1\right)  $ and let $\mathcal{H}\subset\mathcal{E}_{d}$ be the convex hull
of $\mathbf{r}_{1},...,\mathbf{r}_{k}\in\mathcal{E}_{d}$ with a finite set of
witnesses $\mathcal{V}\left(  \mathcal{H}\right)  .$ Further, let $\Lambda$ be
the set of primitive graph-cycles of $G_{\varepsilon}(\mathcal{H})$ without
the trivial one $(0)$. Then $\mathcal{H}\cap\mathcal{D}_{d,\varepsilon}%
^{0}=\mathcal{H}\diagdown%
{\textstyle\bigcup\limits_{\pi\in\Lambda}}
P_{\varepsilon}\left(  \pi\right)  .$
\end{proposition}

Since the number of primitive cycles of a finite directed graph is finite,
Proposition \ref{prop-conv-H} gives an efficient way to determine a subregion
of $\mathcal{D}_{d,\varepsilon}^{0}$ contained in a given convex polyhedron
$\mathcal{H}$ lying in $\operatorname*{Int}\mathcal{D}_{d,\varepsilon
}=\mathcal{E}_{d}.$ Note that if $G_{\varepsilon}(\mathcal{H})$ has only
trivial cycle $(0)$ or if $%
{\textstyle\bigcup\limits_{\pi\in\Lambda}}
P_{\varepsilon}\left(  \pi\right)  \cap\mathcal{H}=\emptyset,$ then
$\mathcal{H}\subset\mathcal{D}_{d,\varepsilon}^{0}.$ If the algorithm to
create a set of witnesses $\mathcal{V}\left(  \mathcal{H}\right)  $ does not
terminate in a certain prescribed time, we have to subdivide $\mathcal{H}$
into several convex hulls and perform the algorithm for each of them. In fact,
Lemma \ref{lem-alg} and Proposition \ref{prop-conv-H} provide an effective
algorithm to obtain $\mathcal{F}\cap\mathcal{D}_{d,\varepsilon}^{0}$ for any
compact set $\mathcal{F}$ contained in the set $\mathcal{E}_{d}.$ Namely, if
$\mathcal{F}$ is compact and contained in $\mathcal{E}_{d}$, there exists a
finite covering of $\mathcal{F}$ by sufficiently small convex polyhedra, each
of which is contained in the interior of $\mathcal{E}_{d}$. In particular, if
$\mathcal{F}$ is a closed set and $\mathcal{D}_{d,\varepsilon}^{0}%
\mathcal{\subset F\subset E}_{d}$, then $\mathcal{D}_{d,\varepsilon}^{0}$ can
be characterized by cutting out finitely many polyhedra from $\mathcal{F}$.

Surer \cite{Surer} has shown that $\mathcal{D}_{2,\varepsilon}^{0}\subseteq
D^{\ast}\left(  \varepsilon\right)  $ and $\overline{D^{\ast}\left(
\varepsilon\right)  }\subset\mathcal{E}_{2}$ for all $\varepsilon\in\left(
0,1\right)  \diagdown\left\{  \frac{1}{2}\right\}  .$ Therefore,
$\varepsilon-$SRS region $\mathcal{D}_{2,\varepsilon}^{0}$ can be obtained
from $D^{\ast}\left(  \varepsilon\right)  $ by cutting out finitely many
polygons. Note that $D\left(  \varepsilon\right)  =D^{\ast}\left(
\varepsilon\right)  \diagdown T\left(  \varepsilon\right)  ,$ where $T\left(
\varepsilon\right)  $ is triangle given by%

\begin{equation}
T\left(  \varepsilon\right)  {\footnotesize {=\!\!%
\begin{cases}
\left\{  \left(  x,y\right)  \in\mathbb{R}^{2}:-x-\varepsilon\leq
y<x+1-\varepsilon,\ -\frac{1}{2}\leq x<-\varepsilon\right\}  ,\text{ if
}\varepsilon\in\left[  0,\frac{1}{2}\right) \\
\left\{  \left(  x,y\right)  \in\mathbb{R}^{2}:-x-1+\varepsilon<y\leq
x+\varepsilon,\ -\frac{1}{2}\leq x\leq-\left(  1-\varepsilon\right)  \right\}
,\text{ if }\varepsilon\in\left[  \frac{1}{2},1\right)
\end{cases}
\!\!\!\!\!.}} \label{T}%
\end{equation}
To prove $\mathcal{D}_{2,\varepsilon}^{0}\subset D\left(  \varepsilon\right)
,$ it is therefore sufficient to prove that $T\left(  \varepsilon\right)  $
can be cut out from $D^{\ast}\left(  \varepsilon\right)  ,$ i.e. that
$T\left(  \varepsilon\right)  \subseteq D^{\ast}\left(  \varepsilon\right)
\diagdown\mathcal{D}_{d,\varepsilon}^{0}.$ Therefore, to prove
(\ref{incl-main}), we have to apply the algorithm of Lemma \ref{lem-alg} to
$\overline{D^{\ast}\left(  \varepsilon\right)  }$, precisely to the closed
convex subsets $\overline{T\left(  \varepsilon\right)  }$ and $\overline
{B\left(  \varepsilon\right)  }\ $of the set $\overline{D^{\ast}\left(
\varepsilon\right)  }.$

Surer \cite[Corollary 2.3 ]{Surer} has shown that for $\varepsilon\in\left(
0,1\right)  $ the sets $\mathcal{D}_{d,\varepsilon}^{0}$ and $\mathcal{D}%
_{d,1-\varepsilon}^{0}$ differ only by a set of measure zero with respect to
the $d-$dimensional Lebesgue measure. This result is a direct consequence of
(\ref{poly}) and \cite[Lemma 2.2]{Surer}, which states that for each cycle
$\pi$ we have $\operatorname{Int}P_{\varepsilon}\left(  \pi\right)
=\operatorname{Int}P_{1-\varepsilon}\left(  -\pi\right)  $ for all
$\varepsilon\in\left(  0,1\right)  $. From (\ref{D*})\ it is easy to see that
for all $\varepsilon\in\left(  0,1\right)  $ the sets $D^{\ast}\left(
\varepsilon\right)  $ and $D^{\ast}\left(  1-\varepsilon\right)  $ are equal
up to the boundary and the boundaries are reversed. The same symmetry also
applies to the sets $D\left(  \varepsilon\right)  $ and $D\left(
1-\varepsilon\right)  ,$ to the sets $T\left(  \varepsilon\right)  $ and
$T\left(  1-\varepsilon\right)  ,$ and also to the sets $B\left(
\varepsilon\right)  $ and $B\left(  1-\varepsilon\right)  ,$ except on a part
of the boundary of these sets where $x=\frac{2}{3}-\varepsilon\ $(see
(\ref{D}), (\ref{T}) and (\ref{B})). Consequently, the sets $B\left(
\varepsilon\right)  \cap\mathcal{D}_{2,\varepsilon}^{0}$ and $B\left(
1-\varepsilon\right)  \cap\mathcal{D}_{2,1-\varepsilon}^{0}$ differ only by a
set of measure zero. The same applies to the sets $T\left(  \varepsilon
\right)  \cap\mathcal{D}_{2,\varepsilon}^{0}$ and $T\left(  1-\varepsilon
\right)  \cap\mathcal{D}_{2,1-\varepsilon}^{0}.$ Therefore, we will first
prove Theorem \ref{Trm-inclus} for $\varepsilon\in\left(  0,\frac{1}%
{2}\right)  $ and then use the symmetry described above to prove it for
$\varepsilon\in\left(  \frac{1}{2},1\right)  .$

\begin{remark}
\label{rem-cyc}According to Lemma \ref{lem-alg}, for a given (sufficiently
small) convex hull $\mathcal{H}$ of $\mathbf{r}_{1},...,\mathbf{r}_{k}$ and
$\varepsilon\in\left[  0,1\right)  ,$ the set of vertices $\mathcal{V}\left(
\mathcal{H}\right)  $ of the corresponding graph $G_{\varepsilon}%
(\mathcal{H})$ does not depend on $\varepsilon$, but the set of edges
$E_{\varepsilon}$ does. Therefore, the graphs $G_{\varepsilon}(\mathcal{H})$
and $G_{1-\varepsilon}(\mathcal{H})$ have the same sets of vertices for each
$\varepsilon\in\left(  0,1\right)  $, but the sets of edges may differ. From
the definition of the set of witnesses $\mathcal{V}\left(  \mathcal{H}\right)
$ also follows: if $\mathbf{z\in}\mathcal{V}\left(  \mathcal{H}\right)  ,$
then $-\mathbf{z\in}\mathcal{V}\left(  \mathcal{H}\right)  .$ Thus, if $\pi
\in\Pi\ $represents a non-zero primitive cycle $\mathbf{z}_{1}%
\mathbf{\rightarrow z}_{2}\mathbf{\rightarrow}...\mathbf{\rightarrow z}%
_{l}\mathbf{\rightarrow z}_{1}$ of the graph $G_{\varepsilon}(\mathcal{H})$,
then for each $l=1,...,l$ the pair $\left(  \mathbf{z}_{i}\mathbf{,z}%
_{i+1}\right)  $ is an edge of the graph $G_{\varepsilon}(\mathcal{H})$ and
$-\mathbf{z}_{1}\mathbf{,-z}_{2}\mathbf{,}...\mathbf{,-z}_{l}$ are vertexes of
the graph $G_{1-\varepsilon}(\mathcal{H})$, but each $\left(  -\mathbf{z}%
_{i}\mathbf{,-z}_{i+1}\right)  $ does not necessarily belong to the set of
edges $E_{1-\varepsilon}$ of the graph $G_{1-\varepsilon}(\mathcal{H}).$
Consequently, $-\pi$ is not necessarily a non-zero primitive cycle of the
graph $G_{1-\varepsilon}(\mathcal{H}).$ However, if for a non-zero primitive
cycle $\pi$ of the graph $G_{\varepsilon}(\mathcal{H})$ we have
$\operatorname{Int}P_{\varepsilon}\left(  \pi\right)  \cap\mathcal{H\neq
\varnothing}$, then we can expect $-\pi$ to be a non-zero primitive cycle of
the graph $G_{1-\varepsilon}(\mathcal{H}),$ since $\operatorname{Int}%
P_{\varepsilon}\left(  \pi\right)  =\operatorname{Int}P_{1-\varepsilon}\left(
-\pi\right)  $. Further, since sets $P_{\varepsilon}\left(  \pi\right)  $ and
$P_{1-\varepsilon}\left(  -\pi\right)  $ are equal up to the boundary and
their boundaries are reversed, then if there is a cycle $\pi$ of the graph
$G_{\varepsilon}(\mathcal{H})$ such that $\operatorname{Int}P_{\varepsilon
}\left(  \pi\right)  \cap\mathcal{H=\varnothing}$ and $P_{\varepsilon}\left(
\pi\right)  \cap\mathcal{H\neq\varnothing}$, then probably $-\pi$ is not a
cycle of the graph $G_{1-\varepsilon}(\mathcal{H})$.
\end{remark}

In the following, we use $\square\left(  Q_{1},...,Q_{k}\right)  $ to denote
the convex hull of the points $Q_{1},...,Q_{k}\in\mathbb{R}^{2}.$ In
particular, we use $\Delta\left(  Q_{1},Q_{2},Q_{3}\right)  $ to denote closed
plane triangle with vertices $Q_{1},Q_{2},Q_{3}$ and $\overline{Q_{1}Q_{2}}$
to denote the closed line segment with endpoints $Q_{1}$ and $Q_{2}.$

First assume $\varepsilon\in\left(  0,\frac{1}{2}\right)  $. We define the
following sets:

\noindent$\Delta_{1}\left(  \varepsilon\right)  =\Delta(A_{\varepsilon
},B_{\varepsilon},C_{\varepsilon}),$ where $A_{\varepsilon}=(-\frac{1}%
{2},\frac{1}{2}-\varepsilon),B_{\varepsilon}=(-\varepsilon,0),C_{\varepsilon
}=(-\varepsilon,1-2\varepsilon),$

\noindent$\Delta_{2}\left(  \varepsilon\right)  =\Delta(B_{\varepsilon
},C_{\varepsilon},D_{\varepsilon}),$ where $D_{\varepsilon}=(\frac{2}%
{3}-\varepsilon,0),$

\noindent$\Delta_{3}\left(  \varepsilon\right)  =\square(B_{\varepsilon
},D_{\varepsilon},F_{\varepsilon},E_{\varepsilon}),$ where $E_{\varepsilon
}=(\frac{1}{3}-\varepsilon,-\frac{1}{3}),$ $F_{\varepsilon}=(\frac{2}%
{3}-\varepsilon,-\frac{1}{3}),$

\noindent$\Delta_{4}\left(  \varepsilon\right)  =\Delta(E_{\varepsilon
},F_{\varepsilon},G_{\varepsilon}),$ where $G_{\varepsilon}=(\frac{1}%
{2}-\varepsilon,-\frac{1}{2}),$

\noindent$\Delta_{5}\left(  \varepsilon\right)  =\Delta(G_{\varepsilon
},F_{\varepsilon},H_{\varepsilon}),$ where $H_{\varepsilon}=(\frac{2}%
{3}-\varepsilon,-\frac{2}{3}),$

\noindent$\Delta_{6}\left(  \varepsilon\right)  =\Delta(D_{\varepsilon
},I_{\varepsilon},K_{\varepsilon}),$ where $I_{\varepsilon}=(\frac{2}%
{3}-\varepsilon,\frac{1}{2}-\varepsilon),K_{\varepsilon}=(\frac{1}%
{3}-\varepsilon,\frac{1}{2}-\varepsilon),$

\noindent$\Delta_{7}\left(  \varepsilon\right)  =\Delta(I_{\varepsilon
},K_{\varepsilon},U_{\varepsilon}),$ where $U_{\varepsilon}=(\frac{1}%
{2}-\varepsilon,\frac{3}{4}-\frac{3}{2}\varepsilon),$

\noindent$\Delta_{8}\left(  \varepsilon\right)  =\Delta(I_{\varepsilon
},J_{\varepsilon},U_{\varepsilon}),$ where $J_{\varepsilon}=(\frac{2}%
{3}-\varepsilon,1-2\varepsilon),$

\noindent$\Delta_{9}\left(  \varepsilon\right)  =\Delta(J_{\varepsilon
},K_{\varepsilon},L_{\varepsilon}),$ where $L_{\varepsilon}=(\frac{1}%
{3}-\varepsilon,1-2\varepsilon),$

\noindent$\Delta_{10}\left(  \varepsilon\right)  =\Delta(C_{\varepsilon
},K_{\varepsilon},L_{\varepsilon}),$

\noindent$\Delta_{11}\left(  \varepsilon\right)  =\Delta(J_{\varepsilon
},N_{\varepsilon},O_{\varepsilon}),$ where $N_{\varepsilon}=(\frac{1}%
{2}-\varepsilon,1-2\varepsilon),O_{\varepsilon}=(\frac{2}{3}-\varepsilon
,\frac{7}{6}-2\varepsilon),$

\noindent$\Delta_{12}\left(  \varepsilon\right)  =\square(M_{\varepsilon
},N_{\varepsilon},O_{\varepsilon},P_{\varepsilon},R_{\varepsilon}),$ where
$M_{\varepsilon}=(\frac{1}{6}-\varepsilon,1-2\varepsilon),$ $P_{\varepsilon
}=(\frac{2}{3}-\varepsilon,\frac{4}{3}-2\varepsilon),$ $R_{\varepsilon}%
=(\frac{1}{2}-\varepsilon,\frac{4}{3}-2\varepsilon),$

\noindent$\Delta_{13}\left(  \varepsilon\right)  =\Delta(P_{\varepsilon
},R_{\varepsilon},S_{\varepsilon}),$ where $S_{\varepsilon}=(\frac{2}%
{3}-\varepsilon,\frac{3}{2}-2\varepsilon),$

\noindent$\Delta_{14}\left(  \varepsilon\right)  =\square(C_{\varepsilon
},M_{\varepsilon},S_{\varepsilon},Z_{\varepsilon}),$ where $Z_{\varepsilon
}=(\frac{2}{3}-\varepsilon,\frac{5}{3}-2\varepsilon).$

\noindent Then $\overline{T\left(  \varepsilon\right)  }=\Delta_{1}\left(
\varepsilon\right)  $ and $\overline{B\left(  \varepsilon\right)  }%
=\bigcup\limits_{i=2}^{14}\Delta_{i}\left(  \varepsilon\right)  $ (see Figure
\ref{slika180}). In what follows, we will use the abbreviations:
$A:=A_{\varepsilon},...,Z:=Z_{\varepsilon}\ $and $\Delta_{i}:=\Delta
_{i}\left(  \varepsilon\right)  ,$ $i=1,...,14.$%

\begin{figure}[h]
\centering
\includegraphics{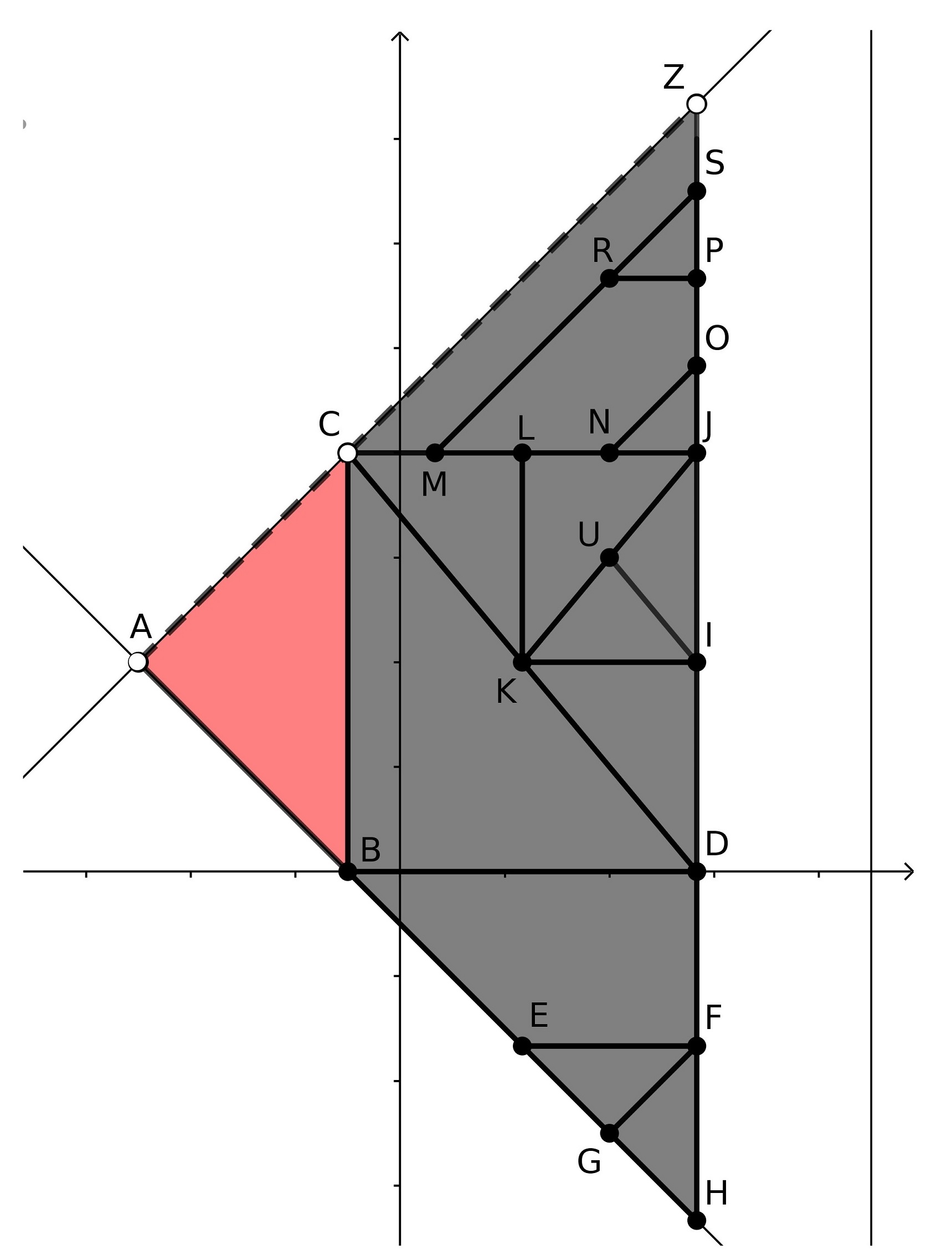}
\caption{\small The sets $\overline{B\left( \varepsilon\right) }$ and $\overline{T\left(  \varepsilon\right) }.$}
\label{slika180}
\end{figure}

\begin{lemma}
\label{lemma-delta1}Let $\varepsilon\in\left(  0,\frac{1}{2}\right)  .$ Then
$\Delta_{1}\cap\mathcal{D}_{2,\varepsilon}^{0}=\mathcal{F}_{1},\ $where
\[
\mathcal{F}_{1}\mathcal{=}\{(-\varepsilon,y)\in\mathbb{R}^{2}:0\leq
y<1-2\varepsilon\}=\overline{BC}\setminus\left\{  C\right\}  .
\]

\end{lemma}

\begin{proof}
The algorithm of Lemma \ref{lem-alg} for $\mathcal{H}=\Delta_{1}$ and all
$\varepsilon\in\left(  0,\frac{1}{2}\right)  $ leads to the graph
$G_{\varepsilon}(\Delta_{1})$ of $7$ vertices and $10$ or $11$ edges as
follows:\\[0.02in]$\left(  0,0\right)  ,\pm\left(  1,0\right)  ,\pm\left(
0,1\right)  ,\pm\left(  1,-1\right)  ;$\\[0.02in]$\left(  0,0\right)
\rightarrow\left(  0,0\right)  ,$ $\left(  1,0\right)  \rightarrow\left(
0,1\right)  ,$ $\left(  1,0\right)  \rightarrow\left(  0,0\right)  ,$ $\left(
-1,0\right)  \rightarrow\left(  0,0\right)  ,$ $\left(  0,1\right)
\rightarrow\left(  1,0\right)  ,$ $\left(  0,-1\right)  \rightarrow\left(
-1,0\right)  ,$ $\left(  -1,1\right)  \rightarrow\left(  0,1\right)  ,$
$\left(  -1,1\right)  \rightarrow\left(  1,-1\right)  ,$ $\left(  1,-1\right)
\rightarrow\left(  -1,0\right)  ,$ $\left(  1,-1\right)  \rightarrow\left(
-1,1\right)  \ $for all $\varepsilon\in\left(  0,\frac{1}{2}\right)  $, and
also $\left(  0,-1\right)  \rightarrow\left(  -1,1\right)  $ if $\varepsilon
\in\left(  0,\frac{1}{3}\right)  .$ \newline Therefore, the graph
$G_{\varepsilon}(\Delta_{1})$ contains two non-zero primitive cycles for all
$\varepsilon\in\left(  0,\frac{1}{2}\right)  $: $\pi_{1}=(1,0)\ $and $\pi
_{2}=(-1,1)\ $which give cutout polygons
\begin{align*}
P_{\varepsilon}(\pi_{1})  &  =\{(x,y)|-1\leq x+\varepsilon<0,0\leq
y+\varepsilon<1\}\\
P_{\varepsilon}(\pi_{2})  &  =\{(x,y)|1\leq-x+y+\varepsilon<2,-1\leq
x-y+\varepsilon<0\}.
\end{align*}
Since $\Delta_{1}\cap P_{\varepsilon}(\pi_{1})=\Delta_{1}\backslash
\overline{BC}$ and $\Delta_{1}\cap P_{\varepsilon}(\pi_{2})=\overline{AC}$,
then $\Delta_{1}\cap(P(\pi_{1})\cup P(\pi_{2}))=\Delta_{1}\backslash
\mathcal{F}_{1},$ which implies$\ \Delta_{1}\cap\mathcal{D}_{2,\varepsilon
}^{0}=\mathcal{F}_{1}$.
\end{proof}

\begin{lemma}
\label{Lemma-delta-0}Let $\varepsilon\in\left(  0,\frac{1}{2}\right)  .$ Then
$\Delta_{i}\subset\mathcal{D}_{2,\varepsilon}^{0}$ for
$i=3,4,5,6,7,8,9,11,12,13.$
\end{lemma}

\begin{proof}
The algorithm of Lemma \ref{lem-alg} for $\mathcal{H}=\Delta_{3}$ and every
$\varepsilon\in\left(  0,\frac{1}{2}\right)  $ leads to the graph
$G_{\varepsilon}(\Delta_{3}),$ which is a subgraph of the graph $G^{\prime
}(\Delta_{3})$ with $9$ vertices and $16$ edges (for all $\varepsilon$'s, the
graph $G_{\varepsilon}(\Delta_{3})$ has $9$ vertices but the number of edges
depends on which subinterval of $\left(  0,\frac{1}{2}\right)  $ $\varepsilon$
belongs to). Since the graph $G^{\prime}(\Delta_{3})$ only contains the
trivial cycle $(0)$, then, for each $\varepsilon\in\left(  0,\frac{1}%
{2}\right)  ,$ the corresponding graph $G_{\varepsilon}(\Delta_{3})$ also only
has the trivial cycle $(0).$ This implies $\Delta_{3}\subset\mathcal{D}%
_{2,\varepsilon}^{0}$ for all $\varepsilon\in\left(  0,\frac{1}{2}\right)  .$

Similarly, for any $\Delta_{i},$ $i=4,5,9,11,13,$ the algorithm of Lemma
\ref{lem-alg} yields the graph $G_{\varepsilon}(\Delta_{i}),$ which only has
the trivial cycle $(0)$ for all $\varepsilon\in\left(  0,\frac{1}{2}\right)
.$ More precisely, we have:

\begin{itemize}
\item For $\Delta_{4}$ we distinguish between two cases, depending on which
subinterval $\varepsilon$ belongs to:

\begin{itemize}
\item if $\varepsilon\in\left[  \frac{1}{3},\frac{1}{2}\right)  ,$ then the
graph $G_{\varepsilon}(\Delta_{4})$ has $7$ vertices, $8$ edges and only the
trivial cycle $(0);$

\item if $\varepsilon\in\left(  0,\frac{1}{3}\right)  ,$ then the graph
$G_{\varepsilon}(\Delta_{4})$ is subgraph of the graph $G^{\prime}(\Delta
_{4})$ with $9$ vertices, $12$ edges and only the trivial cycle $(0);$
\end{itemize}

\item For $\Delta_{5}$ we distinguish between three cases:

\begin{itemize}
\item if $\varepsilon\in\left[  \frac{1}{3},\frac{1}{2}\right)  ,$ then the
graph $G_{\varepsilon}(\Delta_{5})$ has $7$ vertices, $9$ edges and only the
trivial cycle $(0);$

\item if $\varepsilon\in\left[  \frac{1}{6},\frac{1}{3}\right)  ,$ then the
graph $G_{\varepsilon}(\Delta_{5})$ is subgraph of the graph $G^{\prime
}(\Delta_{5})$ with $15$ vertices, $21$ edges and only the trivial cycle
$(0);$

\item if $\varepsilon\in\left(  0,\frac{1}{6}\right)  ,$ then the graph
$G_{\varepsilon}(\Delta_{5})$ is subgraph of the graph $G^{\prime}(\Delta
_{5})$ with $21$ vertices and $30$ edges and only the trivial cycle $(0).$
\end{itemize}

\item For $\Delta_{9}$ we distinguish between two cases:

\begin{itemize}
\item if $\varepsilon\in\left(  \frac{1}{3},\frac{1}{2}\right)  ,$ then the
graph $G_{\varepsilon}(\Delta_{9})$ is subgraph of the graph $G^{\prime
}(\Delta_{9})$ with $9$ vertices, $10$ edges and only the trivial cycle $(0);$

\item if $\varepsilon\in\left(  0,\frac{1}{3}\right]  ,$ then the graph
$G_{\varepsilon}(\Delta_{9})$ is subgraph of the graph $G^{\prime\prime
}(\Delta_{9})$ with $7$ vertices and $10$ edges and only the trivial cycle
$(0).$
\end{itemize}

\item For $\Delta_{11}$ we distinguish between three cases:

\begin{itemize}
\item if $\varepsilon\in\left[  \frac{1}{3},\frac{1}{2}\right)  ,$ then graph
$G_{\varepsilon}(\Delta_{11})$ is subgraph of the graph $G^{\prime}%
(\Delta_{11})$ with $9$ vertices and $12$ edges and only the trivial cycle
$(0);$

\item if $\varepsilon\in\left[  \frac{1}{12},\frac{1}{3}\right)  ,$ then graph
$G_{\varepsilon}(\Delta_{11})$ has $7$ vertices and $10$ edges and only the
trivial cycle $(0);$

\item if $\varepsilon\in\left(  0,\frac{1}{12}\right)  ,$ then the graph
$G_{\varepsilon}(\Delta_{11})$ is subgraph of the graph $G^{\prime\prime
}(\Delta_{11})$ with $33$ vertices and $49$ edges and only the trivial cycle
$(0).$
\end{itemize}

\item For $\Delta_{13}$ we distinguish between three cases:

\begin{itemize}
\item if $\varepsilon\in\left[  \frac{1}{4},\frac{1}{2}\right)  ,$ then the
graph $G_{\varepsilon}(\Delta_{13})$ has $7$ vertices and $9$ edges and only
the trivial cycle $(0);$

\item if $\varepsilon\in\left[  \frac{1}{9},\frac{1}{4}\right)  ,$ then the
graph $G_{\varepsilon}(\Delta_{13})$ has $15$ vertices and $19$ edges and only
the trivial cycle $(0);$

\item if $\varepsilon\in\left(  0,\frac{1}{9}\right)  ,$ then the graph
$G_{\varepsilon}(\Delta_{13})$ is subgraph of the graph $G^{\prime}%
(\Delta_{13})$ with $37$ vertices and $57$ edges and only the trivial cycle
$(0).$
\end{itemize}
\end{itemize}

This proves lemma for $i=4,5,9,11,13$.

The algorithm of Lemma \ref{lem-alg} for $\mathcal{H}=\Delta_{i}$, $i=6,7$ and
each$\ \varepsilon\in\left(  0,\frac{1}{2}\right)  $ leads to the graphs
$G_{\varepsilon}(\Delta_{i})$ which are subgraphs of the graphs with one
non-zero cycle\ $\pi_{3}=(1,1,-1)$. More precisely, we have:

\begin{itemize}
\item For $\Delta_{6}$ we distinguish between two cases:

\begin{itemize}
\item if $\varepsilon\in\left[  \frac{1}{12},\frac{1}{2}\right)  ,$ then the
graph $G_{\varepsilon}(\Delta_{6})$ is subgraph of the graph $G^{\prime
}(\Delta_{6})$ with $9$ vertices, $15$ edges and one non-zero cycle\ $\pi_{3}$.

\item if $\varepsilon\in\left(  0,\frac{1}{12}\right)  $ then the graph
$G_{\varepsilon}(\Delta_{6})$ is subgraph of the graph $G^{\prime\prime
}(\Delta_{6})$ with $39$ vertices, $79$ edges and one non-zero cycle\ $\pi
_{3}$.
\end{itemize}

\item For $\Delta_{7}$ we distinguish between two cases:

\begin{itemize}
\item if $\varepsilon\in\left[  \frac{1}{10},\frac{1}{2}\right)  $ then the
graph $G_{\varepsilon}(\Delta_{7})$ is subgraph of the graph $G^{\prime
}(\Delta_{7})$ with $9$ vertices and $15$ edges and one non-zero cycle
$\pi_{3};$

\item if $\varepsilon\in\left(  0,\frac{1}{10}\right)  $ then the graph
$G_{\varepsilon}(\Delta_{7})$ is subgraph of the graph $G^{\prime\prime
}(\Delta_{7})$ with $21$ vertices and $33$ edges and one non-zero cycle
$\pi_{3}.$
\end{itemize}
\end{itemize}

The graphs $G_{\varepsilon}(\Delta_{6})$ and $G_{\varepsilon}(\Delta_{7}) $
for all $\varepsilon\in\left(  0,\frac{1}{2}\right)  $ contain either only the
trivial cycle $(0)$ or the trivial cycle $(0)$ and a non-zero cycle $\pi
_{3}=(1,1,-1).$ The corresponding cutout polygon $P_{\varepsilon}(\pi_{3})$ is
empty for all $\varepsilon\in\left(  0,\frac{1}{2}\right)  $, so we conclude
$\Delta_{6},\Delta_{7}\subset\mathcal{D}_{2,\varepsilon}^{0}.$

\begin{itemize}
\item For $\Delta_{8}$ we distinguish between three cases:

\begin{itemize}
\item if $\varepsilon\in\left[  \frac{2}{9},\frac{1}{2}\right)  $ then the
graph $G_{\varepsilon}(\Delta_{8})$ has $9$ vertices and $13$ edges and only a
trivial cycle $(0);$

\item if $\varepsilon\in\left[  \frac{1}{9},\frac{2}{9}\right)  $ then the
graph $G_{\varepsilon}(\Delta_{8})$ has $21$ vertices and $36$ edges and one
non-zero cycle $\pi_{3}=(1,1,-1);$

\item if $\varepsilon\in\left(  0,\frac{1}{9}\right)  $ then the graph
$G_{\varepsilon}(\Delta_{8})$ has $37$ vertices and $66$ edges and three
non-zero cycles\ $\ \pi_{3}=(1,1,-1),\ \pi_{4}=(1,-2,2)$ and $\pi
_{5}=(-1,3,-2)$.
\end{itemize}

For all $\varepsilon\in\left(  0,\frac{1}{2}\right)  $ the corresponding
cutout polygons $P_{\varepsilon}(\pi_{3})$ are $P_{\varepsilon}(\pi_{4})$ are
empty, while $P_{\varepsilon}(\pi_{5})$ is not empty but $\Delta_{8}\cap
P_{\varepsilon}(\pi_{5})=\varnothing.$ This proves the lemma for $i=8$.

\item For $\Delta_{12}$ we distinguish between two cases:

\begin{itemize}
\item if $\varepsilon\in\left[  \frac{1}{6},\frac{1}{2}\right)  ,$ then the
graph $G_{\varepsilon}(\Delta_{12})$ has $7$ vertices and $11$ edges and one
non-zero cycle $\pi_{6}=(-1,0,1)$;

\item if $\varepsilon\in\left(  0,\frac{1}{6}\right)  ,$ then the graph
$G_{\varepsilon}(\Delta_{12})$ is a subgraph of the graph $G^{\prime}%
(\Delta_{12})$ with $27$ vertices and $53$ edges and one non-zero
cycle\ $\ \pi_{7}=\left(  0,-1,2,-2,2\right)  .$

For all $\varepsilon\in\left(  0,\frac{1}{2}\right)  $ we have $\Delta
_{12}\cap P_{\varepsilon}(\pi_{6})=\varnothing,$ while $P_{\varepsilon}%
(\pi_{7})$ is empty, so that we conclude $\Delta_{12}\subset\mathcal{D}%
_{2,\varepsilon}^{0}\ $for all $\varepsilon\in\left(  0,\frac{1}{2}\right)  $.
\end{itemize}
\end{itemize}

This completes the proof of the lemma.
\end{proof}

\begin{lemma}
\label{Lemma-delta-C}Let $\varepsilon\in\left(  0,\frac{1}{2}\right)  .$ Then
$\Delta_{i}\cap\mathcal{D}_{2,\varepsilon}^{0}=\Delta_{i}\setminus\left\{
C\right\}  $ for $i=2,10.$
\end{lemma}

\begin{proof}
The algorithm of Lemma \ref{lem-alg} for $\mathcal{H}=\Delta_{2}$ and each
$\varepsilon\in\left(  0,\frac{1}{2}\right)  $ yields the graph
$G_{\varepsilon}(\Delta_{2})$ which is a subgraph of the graph $G^{\prime
}(\Delta_{2})$ with $9$ vertices, $15$ edges and one non-zero cycle\ $\pi
_{2}=(-1,1)$. For $\mathcal{H}=\Delta_{10}$ and each $\varepsilon\in\left(
0,\frac{1}{2}\right)  $, the graph $G_{\varepsilon}(\Delta_{10})$ is subgraph
of the graph $G^{\prime}(\Delta_{10})$ with $7$ vertices, $11$ edges and one
non-zero cycle\ $\pi_{2}=(-1,1)$. Since the cycle\ $\pi_{2}\ $is\ also the
cycle$\ $of\ the graphs $G_{\varepsilon}(\Delta_{2})$ and $G_{\varepsilon
}(\Delta_{10})$ for each $\varepsilon\in\left(  0,\frac{1}{2}\right)  $ and
for the associated cutout polygon $P_{\varepsilon}(\pi_{2})$ we have
$\Delta_{2}\cap P_{\varepsilon}(\pi_{2})=\left\{  C\right\}  \ $and
$\Delta_{10}\cap P_{\varepsilon}(\pi_{2})=\left\{  C\right\}  ,$ the lemma is proved.
\end{proof}

\begin{lemma}
\label{Lemma-delta-CZ}Let $\varepsilon\in\left(  0,\frac{1}{2}\right)  .$ Then
$\Delta_{14}\cap\mathcal{D}_{d,\varepsilon}^{0}=\Delta_{14}\backslash
\mathcal{F}_{2},$ where%
\[
\mathcal{F}_{2}\mathcal{=}\{(x,y)\in\mathbb{R}^{2}:-\varepsilon\leq x\leq
\frac{2}{3}-\varepsilon,y=x+1-\varepsilon\}=\overline{CZ}.
\]

\end{lemma}

\begin{proof}
If $\varepsilon\in\left(  0,\frac{1}{3}\right)  ,$ the construction of the set
of witnesses $\mathcal{V}\left(  \Delta_{14}\right)  $ does not seem to
converge. Therefore, in this case, we are forced to subdivide $\Delta
_{14}=\square(C,M,S,Z)$ into two or more parts and perform the algorithm for
each of these parts. First, for each $\varepsilon\in\left(  0,\frac{1}%
{2}\right)  ,$ we define the following sets: $\Delta_{15}=\square(C,M,S,W)$
and $\Delta_{16}=\Delta(W,T,S),$ where $W=(\varepsilon,1)$ and $T=(\frac{2}%
{3}-\varepsilon,\frac{3}{2}-\varepsilon).$ Note that the point $W$ belongs to
the line $y=x+1-\varepsilon$ and the point $T$ belongs to the line $x=\frac
{2}{3}-\varepsilon,$ but these points do not necessarily belong to the
boundary of $\overline{B\left(  \varepsilon\right)  }$ for all $\varepsilon
\in\left(  0,\frac{1}{2}\right)  $. More precisely, $W\notin\overline{B\left(
\varepsilon\right)  }$ if $\varepsilon\in\left(  \frac{1}{3},\frac{1}%
{2}\right)  $ and $T\notin\overline{B\left(  \varepsilon\right)  }$ if
$\varepsilon\in\left(  \frac{1}{6},\frac{1}{2}\right)  .$ Consequently,
$\Delta_{15}$ and $\Delta_{16}$ do not necessarily lie in $\overline{B\left(
\varepsilon\right)  }$, but $\Delta_{15}\subset\overline{D^{\ast}\left(
\varepsilon\right)  }\subset\mathcal{E}_{d}$ and $\Delta_{16}\subset
\mathcal{E}_{d}$ for all $\varepsilon\in\left(  0,\frac{1}{2}\right)  ,$ so
that we can apply the algorithm of Lemma \ref{lem-alg} to these sets. We have:
$\Delta_{14}\subseteq\Delta_{15}$ for $\varepsilon\in\left[  \frac{1}{3}%
,\frac{1}{2}\right)  ,$ $\Delta_{14}=\Delta_{15}\cup\Delta_{17}$ for
$\varepsilon\in\left[  \frac{1}{6},\frac{1}{3}\right)  $ and $\Delta
_{14}=\Delta_{15}\cup\Delta_{16}\cup\Delta_{18}$\ for $\varepsilon\in\left(
0,\frac{1}{6}\right)  ,$ where $\Delta_{17}=\Delta(W,S,Z)$ and $\Delta
_{18}=\Delta(W,T,Z)$ (see Figure \ref{Slika2}).

\ \ \ \ \ \
\begin{figure}
{\parbox[b]{1.8273in}{\begin{center}
\includegraphics{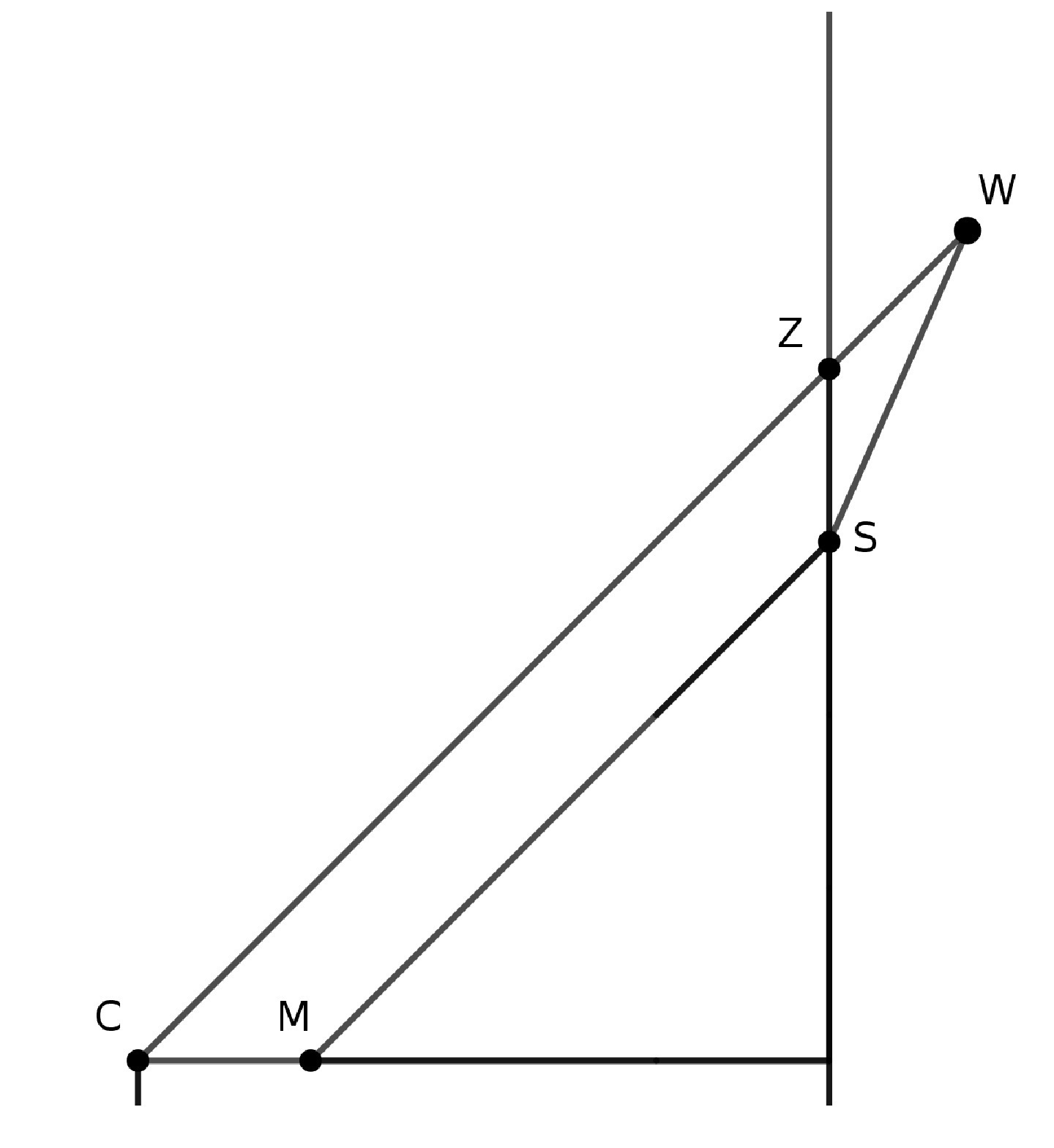}%
\\
{\protect\small The case }$\varepsilon\in\left[  \frac{1}{3}{\protect\small ,}%
\frac{1}{2}\right)  $%
\end{center}}}
\ \ \ \ \ \ \
{\parbox[b]{1.9527in}{\begin{center}
\includegraphics{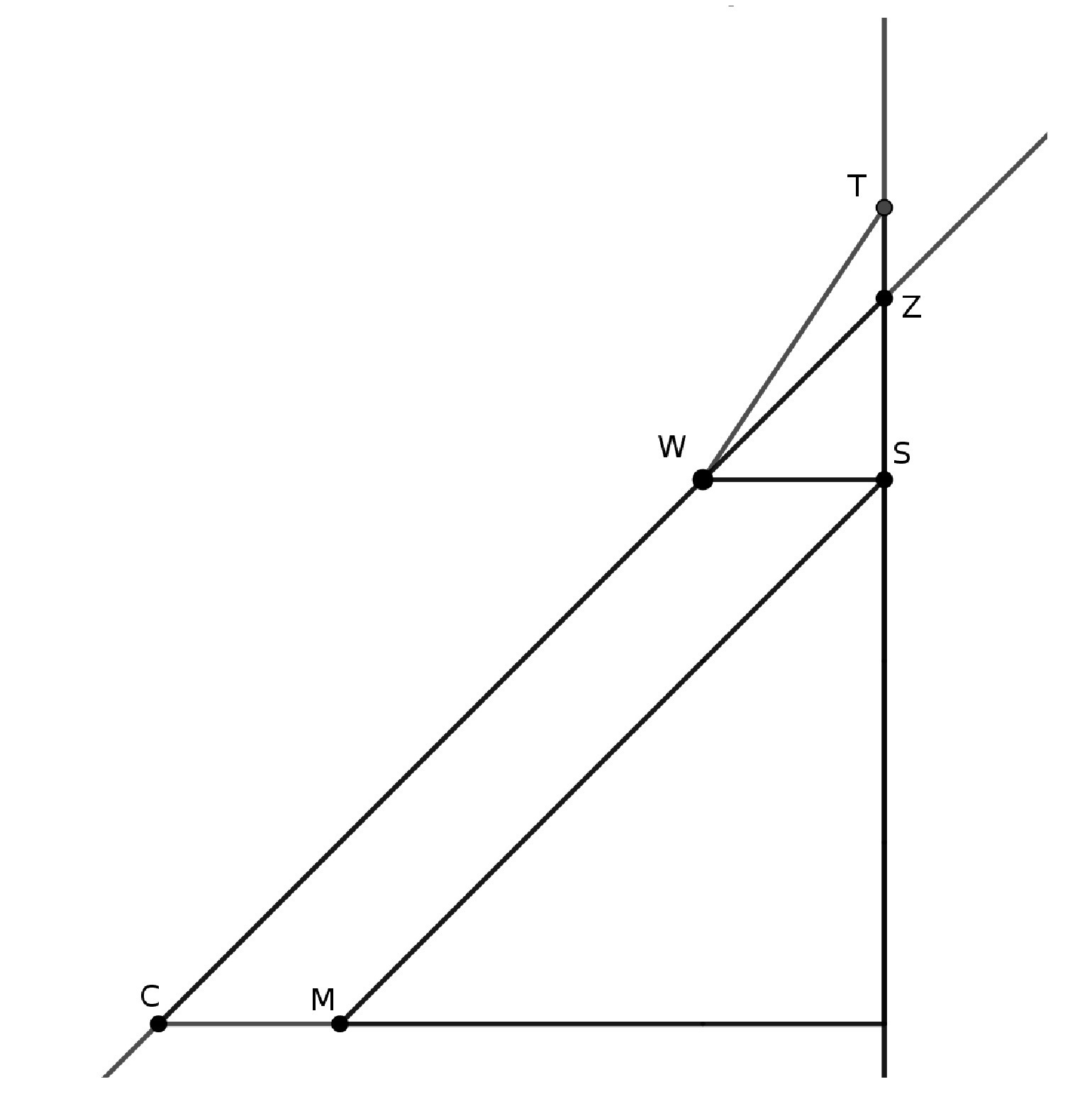}%
\\
{\protect\small The case }$\varepsilon\in\left[  \frac{1}{6}{\protect\small ,}%
\frac{1}{3}\right)  $%
\end{center}}}

\ \ \
{\parbox[b]{2.0176in}{\begin{center}
\includegraphics{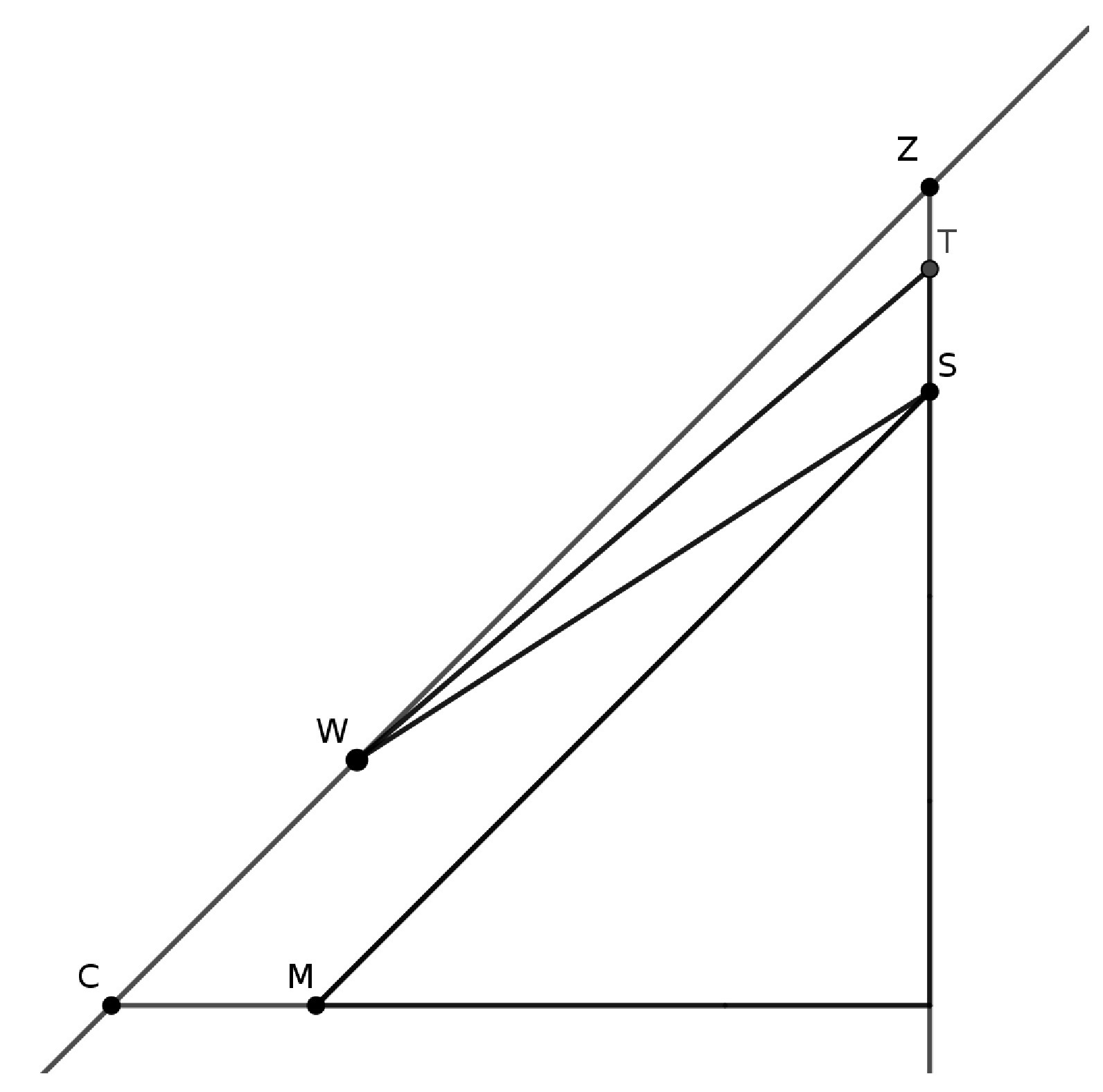}%
\\
{\protect\small The case }$\varepsilon\in\left(  {\protect\small 0,}\frac
{1}{6}\right)  $%
\end{center}}}
\ \ \ \
{\parbox[b]{2.1266in}{\begin{center}
\includegraphics{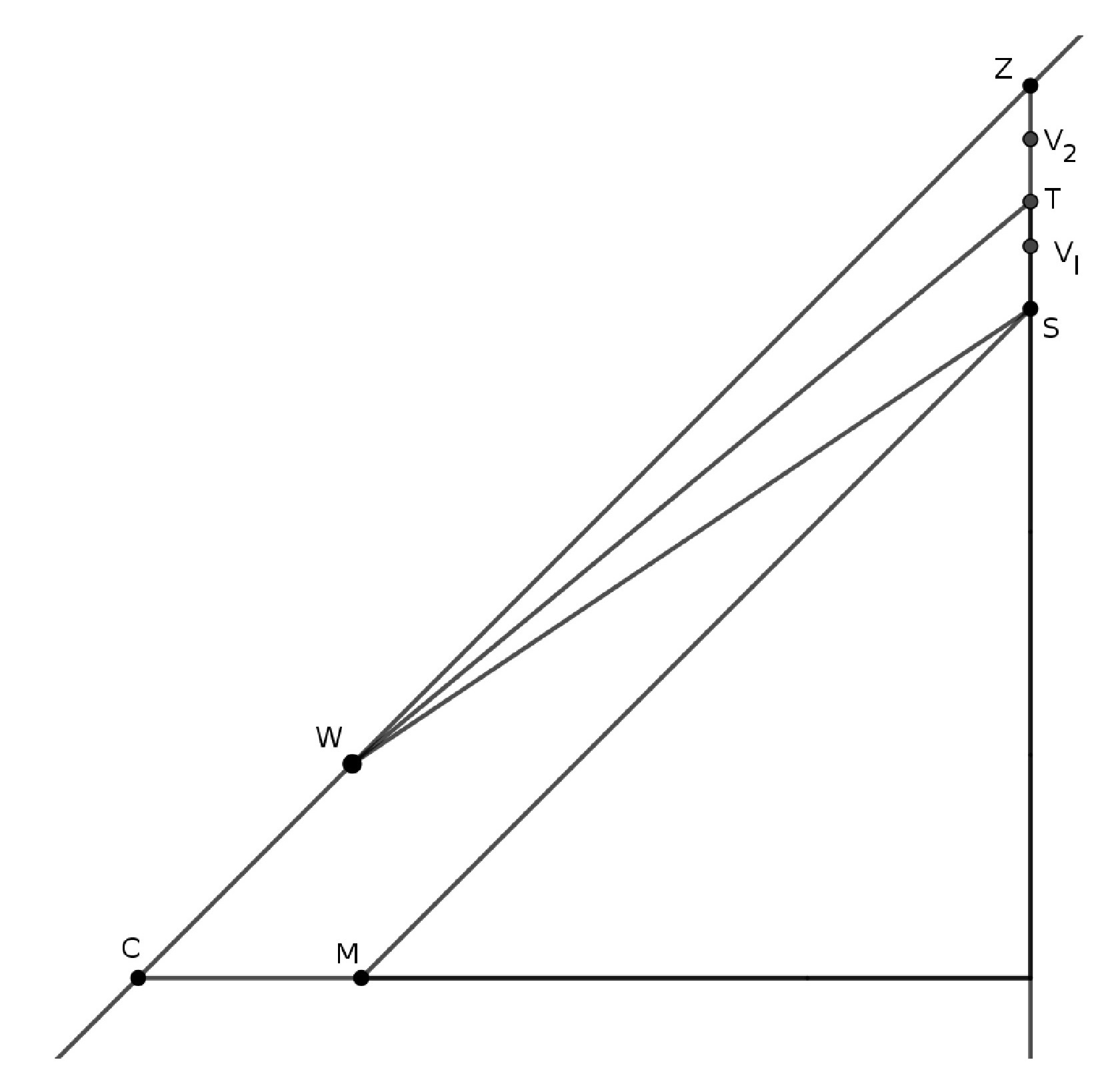}%
\\
{\protect\small The case }$\varepsilon\in\left(  {\protect\small 0,}\frac
{2}{21}\right)  $%
\end{center}}}
\caption{Convex hull $\Delta _{14}$.}
\label{Slika2}
\end{figure}

In the algorithm of Lemma \ref{lem-alg} for $\mathcal{H}=\Delta_{15}$ we
distinguish between four cases, depending on which subinterval of $\left(
0,\frac{1}{2}\right)  $ $\varepsilon$ belongs to:

\begin{itemize}
\item if $\varepsilon\in\left[  \frac{1}{4},\frac{1}{2}\right)  $ then the
graph $G_{\varepsilon}(\Delta_{15})$ is a subgraph of the graph $G^{\prime
}(\Delta_{15})$ with $7$ vertices, $11$ edges and one non-trivial cycle
$\pi_{2}=(-1,1);$

\item if $\varepsilon\in\left[  \frac{1}{9},\frac{1}{4}\right)  $ then the
graph $G_{\varepsilon}(\Delta_{15})$ is a subgraph of the graph $G^{\prime
\prime}(\Delta_{15})$ with $13$ vertices, $22$ edges and one non-trivial cycle
$\pi_{2}=(-1,1);$

\item if $\varepsilon\in\left[  \frac{1}{24},\frac{1}{9}\right)  $ then the
graph $G_{\varepsilon}(\Delta_{15})$ has $23$ vertices, $39$ edges and one
non-trivial cycle $\pi_{2}=(-1,1);$

\item if $\varepsilon\in\left(  0,\frac{1}{24}\right)  $ then the graph
$G_{\varepsilon}(\Delta_{15})$ has $35$ vertices, $61$ edges and one
non-trivial cycle $\pi_{2}=(-1,1).$
\end{itemize}

\noindent Since the cycle\ $\pi_{2}\ $is\ also the cycle$\ $of\ the graph
$G_{\varepsilon}(\Delta_{15})$ for each $\varepsilon\in\left(  0,\frac{1}%
{2}\right)  $ and for the corresponding cutout polygon $P_{\varepsilon}%
(\pi_{2})$ we have $\Delta_{15}\cap P(\pi_{2})=\overline{CW}$, we obtain
$\Delta_{15}\cap\mathcal{D}_{2,\varepsilon}^{0}=\Delta_{15}\diagdown
\overline{CW}$ for all $\varepsilon\in\left(  0,\frac{1}{2}\right)  .$

In the algorithm of Lemma \ref{lem-alg} for $\mathcal{H}=\Delta_{16}$ and
$\varepsilon\in\left(  0,\frac{1}{3}\right)  ,$ we distinguish between two
cases, depending on which subinterval of $\left(  0,\frac{1}{3}\right)  $
$\varepsilon$ belongs to:

\begin{itemize}
\item if $\varepsilon\in\left[  \frac{1}{6},\frac{1}{3}\right)  $ then the
graph $G_{\varepsilon}(\Delta_{16})$ is a subgraph of the graph $G^{\prime
}(\Delta_{16})$ with $25$ vertices, $41$ edges and one non-trivial cycle
$\pi_{2}=(-1,1);$

\item if $\varepsilon\in\left(  0,\frac{1}{6}\right)  $ then the graph
$G_{\varepsilon}(\Delta_{16})$ is a subgraph of the graph $G^{\prime\prime
}(\Delta_{16})$ with $37$ vertices, $61$ edges and one non-trivial cycle
$\pi_{2}=(-1,1).$
\end{itemize}

\noindent The cycle\ $\pi_{2}\ $is\ also the cycle$\ $of\ the graph
$G_{\varepsilon}(\Delta_{16})$ for each $\varepsilon\in\left(  0,\frac{1}%
{2}\right)  $ and for the corresponding cutout polygon $P_{\varepsilon}%
(\pi_{2})$ we have
\[
\Delta_{16}\cap\mathcal{D}_{2,\varepsilon}^{0}=\Delta_{16}\diagdown
P_{\varepsilon}\left(  \pi_{2}\right)  =\left\{
\begin{array}
[c]{cc}%
\Delta_{16}\diagdown\Delta(W,T,Z) & \text{if }\varepsilon\in\left[  \frac
{1}{6},\frac{1}{3}\right) \\
\Delta_{16}\diagdown\left\{  W\right\}  & \text{if }\varepsilon\in\left(
0,\frac{1}{6}\right)
\end{array}
\right.  .
\]

\textbf{Case} $\varepsilon\in\left[  \frac{1}{3},\frac{1}{2}\right)  .$ Since
for each $\varepsilon\in\left[  \frac{1}{3},\frac{1}{2}\right)  $ we have
$\Delta_{14}\subseteq\Delta_{15},$ $\Delta_{15}\cap\mathcal{D}_{2,\varepsilon
}^{0}=\Delta_{15}\diagdown\overline{CW}$ and $\overline{CW}\cap\Delta
_{14}=\overline{CZ}=\mathcal{F}_{2},$ then $\Delta_{14}\cap\mathcal{D}%
_{2,\varepsilon}^{0}=\Delta_{14}\backslash\mathcal{F}_{2}.$ Thus, we proved
lemma for $\varepsilon\in\left[  \frac{1}{3},\frac{1}{2}\right)  .$

\textbf{Case} $\varepsilon\in\left[  \frac{1}{6},\frac{1}{3}\right)  .$ In
this case we have $\Delta_{14}=\Delta_{15}\cup\Delta_{17}$ and $\Delta
_{16}=\Delta_{17}\cup\Delta_{18}.$ Since $\Delta_{16}\cap\mathcal{D}%
_{2,\varepsilon}^{0}=\Delta_{16}\diagdown\Delta(W,T,Z),$ $\Delta_{17}%
\subseteq\Delta_{16}$ and $\Delta(W,T,Z)\cap\Delta_{17}=\overline{WZ},$ then
$\Delta_{17}\cap\mathcal{D}_{2,\varepsilon}^{0}=\Delta_{17}\backslash
\overline{WZ}.$ Consequently, we obtain
\[
\Delta_{14}\cap\mathcal{D}_{2,\varepsilon}^{0}=\left(  \Delta_{15}%
\cap\mathcal{D}_{2,\varepsilon}^{0}\right)  \cup\left(  \Delta_{17}%
\cap\mathcal{D}_{2,\varepsilon}^{0}\right)  =\Delta_{14}\backslash\left(
\overline{CW}\cup\overline{WZ}\right)  =\Delta_{14}\backslash\overline{CZ}.
\]
Therefore, $\Delta_{14}\cap\mathcal{D}_{2,\varepsilon}^{0}=\Delta
_{14}\backslash\mathcal{F}_{2}$ and we proved lemma for $\varepsilon\in\left[
\frac{1}{6},\frac{1}{3}\right)  .$

\textbf{Case} $\varepsilon\in\left(  0,\frac{1}{6}\right)  .$ In this case we
have $\Delta_{14}=\Delta_{15}\cup\Delta_{16}\cup\Delta_{18}.$ Therefore, to
complete the proof of the lemma we need to find the set $\Delta_{18}%
\cap\mathcal{D}_{2,\varepsilon}^{0}$ for all $\varepsilon\in\left(  0,\frac
{1}{6}\right)  .$

Let $\varepsilon\in\left[  \frac{2}{21},\frac{1}{6}\right)  .$ In this case,
the algorithm of Lemma \ref{lem-alg} for $\mathcal{H}=\Delta_{18}$ yields the
graph $G_{\varepsilon}(\Delta_{18}),$ which has $51$ vertices and $82$ edges.
The only non-trivial cycle in this graph is $\pi_{2}=(-1,1).$ Since
$\Delta_{18}\cap P_{\varepsilon}\left(  \pi_{2}\right)  =\overline{WZ},$ we
obtain $\Delta_{18}\cap\mathcal{D}_{2,\varepsilon}^{0}=\Delta_{18}%
\backslash\overline{WZ}.$ Thus,
\[
\Delta_{14}\cap\mathcal{D}_{2,\varepsilon}^{0}=\Delta_{14}\backslash\left(
\overline{CW}\cup\left\{  W\right\}  \cup\overline{WZ}\right)  =\Delta
_{14}\backslash\overline{CZ},
\]
or equivalently $\Delta_{14}\cap\mathcal{D}_{d,\varepsilon}^{0}=\Delta
_{14}\backslash\mathcal{F}_{2}$, so that we complete the proof of the lemma
for all $\varepsilon\in\left[  \frac{2}{21},\frac{1}{6}\right)  .$

If $\varepsilon\in\left(  0,\frac{2}{21}\right)  ,$ in order to find the set
$\Delta_{18}\cap\mathcal{D}_{2,\varepsilon}^{0}$ we are forced to subdivide
the triangle $\Delta_{18}:=\Delta(W,T,Z)$ into two or more triangles,
depending on which subinterval of $\left(  0,\frac{2}{21}\right)  $
$\varepsilon$ belongs to, and perform the algorithm for each of these parts.
If $\varepsilon\in\left(  0,\frac{2}{21}\right)  $, then there exists unique
integer $n\geq4$ such that $\varepsilon\in\Big[\frac{2}{3(2n+1)},\frac
{2}{3(2n-1)}\Big).$ So, let's assume $\varepsilon\in\Big[\frac{2}%
{3(2n+1)},\frac{2}{3(2n-1)}\Big)$ and $n\geq4.$ Define $l=\left\lfloor
\frac{2n^{2}-3n}{4}\right\rfloor +1\ $and $V_{k}=(\frac{2}{3}-\varepsilon
,\frac{5}{3}-2\varepsilon-\frac{k}{n}\varepsilon)$, $k=0,1,..,l.$ Then%
\begin{equation}
\Delta_{17}=\bigcup\limits_{s=0}^{l-1}\Delta(W,V_{s},V_{s+1})\cup\Delta_{16}
\label{unija}%
\end{equation}
Note that $V_{0}=Z\ $and $l$ is a minimal positive integer so that
(\ref{unija}) holds. Also note that $\Delta_{18}\subseteq\bigcup
\limits_{s=0}^{l-1}\Delta(W,V_{s},V_{s+1})=\Delta(W,Z,V_{l}).$ Define sets
$\Delta_{18}^{\left(  s\right)  }=\Delta(W,V_{s},V_{s+1})\ $for $s=0,...,l-1.$

The algorithm of Lemma \ref{lem-alg} for $\mathcal{H}=\Delta_{18}^{\left(
0\right)  }\cup\Delta_{18}^{\left(  1\right)  }=\Delta(W,Z,V_{2})=:\Delta
_{19}$ for each $\varepsilon\in\Big[\frac{2}{3(2n+1)},\frac{2}{3(2n-1)}\Big)$
leads to the graph $G_{\varepsilon}(\Delta_{19}\mathcal{)}$ which is a
subgraph of the graph $G^{\prime}(\Delta_{19})$ with $18n-3$ vertices and
$32n-12$ edges. It can be shown by induction that the graph $G^{\prime}%
(\Delta_{19})$ has only one non-trivial cycle $\pi_{2}=(-1,1)$ (see Lemma
\ref{Lemma-18-prvi}). Since the cycle\ $\pi_{2}\ $is\ also the cycle$\ $%
of\ the graph $G_{\varepsilon}(\Delta_{19}\mathcal{)}$ for each $\varepsilon
\in\Big[\frac{2}{3(2n+1)},\frac{2}{3(2n-1)}\Big)$, then $G_{\varepsilon
}(\Delta_{19})$ also has only one non-trivial cycle $\pi_{2}$ for all
$\varepsilon\in\Big[\frac{2}{3(2n+1)},\frac{2}{3(2n-1)}\Big)$. Since
$\Delta_{19}\cap P(\pi_{2})=\overline{WZ},$ we get $\Delta_{19}\cap
\mathcal{D}_{d,\varepsilon}^{0}=\Delta_{19}\backslash\overline{WZ}.$

The algorithm of Lemma \ref{lem-alg} for $\mathcal{H}=\Delta_{18}^{\left(
s\right)  },$ all $s=2,...,l-1$ and all $\varepsilon\in\Big[\frac{2}%
{3(2n+1)},\frac{2}{3(2n-1)}\Big)$ leads to the graph $G_{\varepsilon}%
(\Delta_{18}^{\left(  s\right)  })$ which is subgraph of the graph $G^{\prime
}(\Delta_{18}^{\left(  s\right)  })$ with $8a_{0}^{\left(  s\right)  }%
+2a_{1}^{\left(  s\right)  }-1$ vertices and $14a_{0}^{\left(  s\right)
}+4a_{1}^{\left(  s\right)  }-6$ edges, where ${\small a}_{0}^{\left(
s\right)  }=\left\lceil \frac{2n^{2}}{n+s}\right\rceil $ and ${\small a}%
_{1}^{\left(  s\right)  }=\left\lceil \frac{n^{2}}{n+s}\right\rceil .$ It can
be shown by induction that the graph $G^{\prime}(\Delta_{18}^{\left(
s\right)  })$ has only one non-trivial cycle $\pi_{2}=(-1,1)$ for all
$s=2,...,l-1$ (see Lemma \ref{Lemma-18-s}). Since the cycle\ $\pi_{2}%
\ $is\ also the cycle$\ $of\ the graph $G_{\varepsilon}(\Delta_{18}^{\left(
s\right)  }\mathcal{)}$ for all $\varepsilon\in\Big[\frac{2}{3(2n+1)},\frac
{2}{3(2n-1)}\Big)$ and all $s=2,..,l-1$, then $G_{\varepsilon}(\Delta
_{18}^{\left(  s\right)  })$ also has only one non-trivial cycle $\pi_{2}$.
Since $\Delta_{18}^{\left(  s\right)  }\cap P_{\varepsilon}(\pi_{2})=\left\{
W\right\}  $, we obtain $\Delta_{18}^{\left(  s\right)  }\cap\mathcal{D}%
_{d,\varepsilon}^{0}=\Delta_{18}^{\left(  s\right)  }\backslash\left\{
W\right\}  \ $for all $s=2,...,l-1\ $and all $\varepsilon\in\Big[\frac
{2}{3(2n+1)},\frac{2}{3(2n-1)}\Big)$.

Since $\Delta(W,Z,V_{l})=\bigcup\limits_{s=2}^{l-1}\Delta_{18}^{\left(
s\right)  }\cup\Delta_{19}\ $then, by De Morgan's law, we obtain
$\Delta(W,Z,V_{l})\cap\mathcal{D}_{d,\varepsilon}^{0}=\Delta(W,Z,V_{l}%
)\backslash\overline{WZ}$. Finally, since $\Delta_{14}=\Delta_{15}\cup
\Delta_{16}\cup\Delta(W,Z,V_{l}),$ we derive $\Delta_{14}\cap\mathcal{D}%
_{d,\varepsilon}^{0}=\Delta_{14}\backslash\left(  \overline{CW}\cup\left\{
W\right\}  \cup\overline{WZ}\right)  =\Delta_{14}\backslash\overline{CZ}$ for
all $\varepsilon\in\Big[\frac{2}{3(2n+1)},\frac{2}{3(2n-1)}\Big)$ and all
$n\geq4,$ which completes the proof of the lemma.
\end{proof}

\begin{remark}
All graphs given in the proofs of Lemmas \ref{Lemma-delta-0},
\ref{lemma-delta1}, \ref{Lemma-delta-C} and \ref{Lemma-delta-CZ} were obtained
by hand calculation. The use of a computer to perform the algorithm of Lemma
\ref{lem-alg} was not possible, since the vertices of the sets $\Delta
_{i}:=\Delta_{i}\left(  \varepsilon\right)  $ and also (\ref{edg}) depend on
$\varepsilon.$ Since most of these graphs become quite large the closer
$\varepsilon$ is to $0,$ we have not written them all down. The proofs that
the graphs $G^{\prime}(\Delta_{19})$ and $G^{\prime}(\Delta_{18}^{\left(
s\right)  })$ have only one non-trivial cycle $\pi_{2}$ are in Apendix
\ref{Sect5}.
\end{remark}

\begin{proposition}
\label{prop-inc-1}Let $\varepsilon\in\left(  0,\frac{1}{2}\right)  .$ Then
$B\left(  \varepsilon\right)  \subset\mathcal{D}_{2,\varepsilon}^{0}\subset
D\left(  \varepsilon\right)  .$
\end{proposition}

\begin{proof}
Since $T\left(  \varepsilon\right)  \subset D^{\ast}\left(  \varepsilon
\right)  ,$ $D\left(  \varepsilon\right)  =D^{\ast}\left(  \varepsilon\right)
\diagdown T\left(  \varepsilon\right)  \ $and $T\left(  \varepsilon\right)
=\Delta_{1}\backslash\mathcal{F}_{1},$ then (\ref{D*-sur}) and Lemma
\ref{lemma-delta1} imply $T\left(  \varepsilon\right)  \subseteq D^{\ast
}\left(  \varepsilon\right)  \diagdown\mathcal{D}_{2,\varepsilon}^{0}$, and
consequently $\mathcal{D}_{2,\varepsilon}^{0}\subset D\left(  \varepsilon
\right)  .$ Since $\overline{B\left(  \varepsilon\right)  }=\bigcup
\limits_{i=2}^{14}\Delta_{i}$, then by Lemmas \ref{Lemma-delta-0},
\ref{Lemma-delta-C} and \ref{Lemma-delta-CZ} we obtain $\overline{B\left(
\varepsilon\right)  }\cap\mathcal{D}_{d,\varepsilon}^{0}=\overline{B\left(
\varepsilon\right)  }\backslash\overline{CZ}.$ By definition (\ref{B}) of the
set $B\left(  \varepsilon\right)  $ we have $B\left(  \varepsilon\right)
=\overline{B\left(  \varepsilon\right)  }\backslash\overline{CZ}$. Therefore,
$\overline{B\left(  \varepsilon\right)  }\cap\mathcal{D}_{d,\varepsilon}%
^{0}=B\left(  \varepsilon\right)  ,$ and thus $B\left(  \varepsilon\right)
\subset\mathcal{D}_{2,\varepsilon}^{0}.$ This proves the proposition.
\end{proof}

For all $\varepsilon\in\left(  0,1\right)  $, we have: $\overline{B\left(
\varepsilon\right)  }=\overline{B\left(  1-\varepsilon\right)  }\ $and
$\overline{T\left(  \varepsilon\right)  }=\overline{T\left(  1-\varepsilon
\right)  }.$ So if we define $\tilde{A}_{\varepsilon}:=A_{1-\varepsilon
},...,\tilde{Z}_{\varepsilon}:=Z_{1-\varepsilon}\ $and $\widetilde{\Delta}%
_{i}\left(  \varepsilon\right)  :=\Delta_{i}\left(  1-\varepsilon\right)  ,$
$i=1,...,14$ for each $\varepsilon\in\left(  \frac{1}{2},1\right)  $, we find
$\overline{T\left(  \varepsilon\right)  }=\widetilde{\Delta}_{1}\left(
\varepsilon\right)  $\ and\ $\overline{B\left(  \varepsilon\right)  }%
=\bigcup\limits_{i=2}^{14}\widetilde{\Delta}_{i}\left(  \varepsilon\right)  $
for all $\varepsilon\in\left(  \frac{1}{2},1\right)  .$ According to Remark
\ref{rem-cyc}, we can expect that for each $\varepsilon\in\left(  \frac{1}%
{2},1\right)  $ sets $\widetilde{\Delta}_{i}\left(  \varepsilon\right)
\cap\mathcal{D}_{d,\varepsilon}^{0}$ and $\Delta_{i}\left(  1-\varepsilon
\right)  \cap\mathcal{D}_{d,1-\varepsilon}^{0}$ are equal up to the boundary.
Only the boundaries will change a little. Also, $\mathcal{V}\left(
\widetilde{\Delta}_{i}\left(  \varepsilon\right)  \right)  =\mathcal{V}\left(
\Delta_{i}\left(  1-\varepsilon\right)  \right)  \ $for all $\varepsilon
\in\left(  \frac{1}{2},1\right)  .$

In what follows, we use the abbreviations: $\tilde{A}:=\tilde{A}_{\varepsilon
},...,\tilde{Z}:=\tilde{Z}_{\varepsilon}\ $and $\tilde{\Delta}_{i}%
:=\tilde{\Delta}_{i}\left(  \varepsilon\right)  ,$ $i=1,...,14$.

\begin{proposition}
\label{prop-inc-2}Let $\varepsilon\in\left(  \frac{1}{2},1\right)  .$ Then
$B\left(  \varepsilon\right)  \subset\mathcal{D}_{2,\varepsilon}^{0}\subset
D\left(  \varepsilon\right)  .$
\end{proposition}

\begin{proof}
We show this very analogously to Proposition \ref{prop-inc-1}. In particular,
the algorithm of Lemma \ref{lem-alg} for $\mathcal{H}=\widetilde{\Delta}_{1}$
and each $\varepsilon\in\left(  \frac{1}{2},1\right)  $ leads to the graph
$G_{\varepsilon}(\widetilde{\Delta}_{1})$ with $7$ vertices and $10$ or $11$
edges as follows:\\[0.02in]$\left(  0,0\right)  ,\pm\left(  1,0\right)
,\pm\left(  0,1\right)  ,\pm\left(  1,-1\right)  ;$ \\[0.02in]$\left(
0,0\right)  \rightarrow\left(  0,0\right)  ,$ $\left(  1,0\right)
\rightarrow\left(  0,0\right)  ,$ $\left(  -1,0\right)  \rightarrow\left(
0,-1\right)  ,$ $\left(  0,1\right)  \rightarrow\left(  1,0\right)  ,$
$\left(  0,-1\right)  \rightarrow\left(  -1,0\right)  ,$ $\left(  -1,1\right)
\rightarrow\left(  1,-1\right)  ,$ $\left(  1,-1\right)  \rightarrow\left(
-1,0\right)  \ $for all $\varepsilon\in\left(  \frac{1}{2},1\right)  $, and
also $\left(  0,1\right)  \rightarrow\left(  1,-1\right)  $ if $\varepsilon
\in\left(  \frac{2}{3},1\right)  .$ \\[0.02in]Therefore, the graph
$G_{\varepsilon}(\widetilde{\Delta}_{1})$ contains one non-zero primitive
cycle $-\pi_{1}=(-1,0)$ for all $\varepsilon\in\left(  \frac{1}{2},1\right)
.$ We obtain $\widetilde{\Delta}_{1}\cap P_{\varepsilon}(-\pi_{1}%
)=\widetilde{\Delta}_{1}$, so we conclude $\widetilde{\Delta}_{1}%
\cap\mathcal{D}_{2,\varepsilon}^{0}=\varnothing.$ Since $\overline{T\left(
\varepsilon\right)  }=\widetilde{\Delta}_{1}$, then $T\left(  \varepsilon
\right)  \cap\mathcal{D}_{2,\varepsilon}^{0}=\varnothing$ which implies
$T\left(  \varepsilon\right)  \subset D^{\ast}\left(  \varepsilon\right)
\diagdown\mathcal{D}_{d,\varepsilon}^{0}$, and consequently $\mathcal{D}%
_{2,\varepsilon}^{0}\subset D\left(  \varepsilon\right)  .$

Furthermore, in the same style as in Lemmas \ref{Lemma-delta-0}%
-\ref{Lemma-delta-CZ}, we show that for all $\varepsilon\in\left(  \frac{1}%
{2},1\right)  $: $\tilde{\Delta}_{i}\subset\mathcal{D}_{d,\varepsilon}^{0}$ if
$i=6,7,8,9,11,12,13,$ $\tilde{\Delta}_{2}\cap\mathcal{D}_{d,\varepsilon}%
^{0}=\tilde{\Delta}_{2}\backslash\overline{\tilde{C}\tilde{B}},$
$\tilde{\Delta}_{3}\cap\mathcal{D}_{d,\varepsilon}^{0}=\tilde{\Delta}%
_{3}\backslash\overline{\tilde{B}\tilde{E}},$ $\tilde{\Delta}_{4}%
\cap\mathcal{D}_{d,\varepsilon}^{0}=\tilde{\Delta}_{4}\backslash
\overline{\tilde{E}\tilde{G}},$ $\tilde{\Delta}_{5}\cap\mathcal{D}%
_{d,\varepsilon}^{0}=\tilde{\Delta}_{5}\backslash\overline{\tilde{G}\tilde{H}%
},$ $\tilde{\Delta}_{14}\cap\mathcal{D}_{d,\varepsilon}^{0}=\tilde{\Delta
}_{14}\backslash\left\{  \tilde{C}\right\}  ,$ $\tilde{\Delta}_{10}%
\cap\mathcal{D}_{d,\varepsilon}^{0}=\tilde{\Delta}_{10}\backslash\left\{
\tilde{C}\right\}  .$ Since $B\left(  \varepsilon\right)  =\overline{B\left(
\varepsilon\right)  }\backslash\left(  \overline{\tilde{C}\tilde{B}}%
\cup\overline{\tilde{B}\tilde{H}}\right)  ,$ $\overline{B\left(
\varepsilon\right)  }=\bigcup\limits_{i=2}^{14}\widetilde{\Delta}_{i}\left(
\varepsilon\right)  $ and $\overline{\tilde{B}\tilde{H}}=\overline{\tilde
{B}\tilde{E}}\cup\overline{\tilde{E}\tilde{G}}\cup\overline{\tilde{G}\tilde
{H}}$, then $\overline{B\left(  \varepsilon\right)  }\cap\mathcal{D}%
_{2,\varepsilon}^{0}=B\left(  \varepsilon\right)  ,$ and consequently
$B\left(  \varepsilon\right)  \subset\mathcal{D}_{2,\varepsilon}^{0}.$
\end{proof}

\begin{proof}
[Proof of Theorem \ref{Trm-inclus}]Directly from Propositions \ref{prop-inc-1}
and \ref{prop-inc-2}.
\end{proof}

Since Theorem \ref{Trm-inclus} holds, then $\mathcal{D}_{2,\varepsilon}%
^{0}\cap R\left(  \varepsilon\right)  =B\left(  \varepsilon\right)  $ for all
$\varepsilon\in\left[  0,1\right)  $ which means that we completely
characterize the sets $\mathcal{D}_{2,\varepsilon}^{0}$ for $x\leq\frac{2}%
{3}-\varepsilon$ if $\varepsilon\in\left[  0,\frac{1}{2}\right)  \ $and for
$x<\frac{2}{3}{\small -}\left(  {\small 1-\varepsilon}\right)  $ if
$\varepsilon\in\left[  \frac{1}{2},1\right)  .$

\begin{remark}
\label{Rem 2}In the case $\varepsilon=0$, the above-mentioned range of $x$
($x\leq\frac{2}{3}$) is the best possible range of $x$ in which $\mathcal{D}%
_{2,\varepsilon}^{0}$ and $D\left(  \varepsilon\right)  $ coincide. Namely,
for $\varepsilon=0$ the range for $x$ cannot go beyond $\frac{2}{3}$ since the
points $(\frac{2}{3},-\frac{1}{3})$ and $(\frac{2}{3},\frac{4}{3})$ are on the
boundary of a cutout polygons but not contained in them (see \cite[p.
55]{Akiyama-BBPT-II} ). For some other $\varepsilon\in\left(  0,1\right)  $
this is not the case. For example, for $\varepsilon=\frac{1}{2}$ the best
possible range of $x$ where $\mathcal{D}_{2,\varepsilon}^{0}$ and $D\left(
\varepsilon\right)  $ coincide is $x<\frac{1}{2}$ (see Proposition
\ref{SRSpola}), for $\varepsilon=\frac{1}{5}$ the best possible range is
$x<\frac{2}{3}-\frac{\varepsilon}{3}=\frac{3}{5},$ while for $\varepsilon
=\frac{1}{10}$ it is $x\leq\frac{2}{3}\ $(see \cite[Theorems 6.1 and
6.2]{Surer}).

We believe that the result we obtained can still be improved. More precisely,
we believe that it can be shown that the sets $\mathcal{D}_{2,\varepsilon}%
^{0}$ and $D\left(  \varepsilon\right)  $ coincide if $x<\frac{2}{3}%
-\frac{\varepsilon}{3}$ for all $\varepsilon\in\left(  0,\frac{1}{2}\right)  $
and if $x\leq\frac{2}{3}-\frac{1-\varepsilon}{3}$ for all $\varepsilon
\in\left(  \frac{1}{2},1\right)  $. Note that the best possible range of $x$
in which $\mathcal{D}_{2,\varepsilon}^{0}$ and $D\left(  \varepsilon\right)  $
coincide can be written as $x\leq\frac{2}{3}-\frac{\varepsilon}{3}$ for
$\varepsilon=0$ and as $x<\frac{2}{3}-\frac{1-\varepsilon}{3}$ for
$\varepsilon=\frac{1}{2}.$ If we define the set $B\left(  \varepsilon\right)
$ according to these new upper bounds for $x$, to prove our assumption we
should prove that such a $B\left(  \varepsilon\right)  $ is contained in the
$\varepsilon-$SRS region $\mathcal{D}_{2,\varepsilon}^{0}$ for all
$\varepsilon\in\left(  0,\frac{1}{2}\right)  \cup\left(  \frac{1}{2},1\right)
.$
\end{remark}

\section{Characterization of quadratic $\varepsilon-$CNS polynomials
\label{Sect4}}

Let $P\left(  x\right)  =x^{d}+p_{d-1}x^{d-1}+...+p_{1}x+p_{0}\in
\mathbb{Z}\left[  x\right]  ,$ $\left\vert p_{0}\right\vert \geq2\ $and
$\varepsilon\in\left[  0,1\right)  .$ Let us write the interval $[0,1)$ as a
disjoint union of the subintervals as follows%
\[
\left[  0,1\right)  =\left[  0,\frac{1}{\left\vert p_{0}\right\vert }\right)
\cup\left[  \frac{1}{\left\vert p_{0}\right\vert },\frac{2}{\left\vert
p_{0}\right\vert }\right)  \cup...\cup\left[  \frac{k}{\left\vert
p_{0}\right\vert },\frac{k+1}{\left\vert p_{0}\right\vert }\right)
\cup...\cup\left[  \frac{\left\vert p_{0}\right\vert -1}{\left\vert
p_{0}\right\vert },1\right)  .
\]
If $\varepsilon\in\left[  \frac{k}{\left\vert p_{0}\right\vert },\frac
{k+1}{\left\vert p_{0}\right\vert }\right)  $, $k=0,...,\left\vert
p_{0}\right\vert -1,$ then $k=\left\lfloor \varepsilon\left\vert
p_{0}\right\vert \right\rfloor $ and the corresponding $\varepsilon-$set of
digits (\ref{N-eps}) has the form $\mathcal{N}_{\varepsilon}\mathbb{=}\left\{
-k,...,\left\vert p_{0}\right\vert -1-k\right\}  .$ Therefore, $\mathcal{N}%
_{\varepsilon}=\mathcal{N}_{\varepsilon_{k}},$ where $\varepsilon_{k}=\frac
{k}{\left\vert p_{0}\right\vert }.$ Note that $\varepsilon_{k}=\frac
{k}{\left\vert p_{0}\right\vert }\leq\mathcal{\varepsilon<}\varepsilon
_{k+1}=\frac{k+1}{\left\vert p_{0}\right\vert }.$ In particular, if
$\varepsilon=\frac{1}{2}$, then the corresponding $k\ $is given by
(\ref{k-pola}), which implies%
\[
\varepsilon_{k}=\left\{
\begin{tabular}
[c]{l}%
$\frac{1}{2},$ $\text{ if }\left\vert p_{0}\right\vert \text{ is even}$\\
$\frac{1}{2}-\frac{1}{2\left\vert p_{0}\right\vert },\text{ \ if }\left\vert
p_{0}\right\vert \text{ is }$odd
\end{tabular}
\ \right.  .
\]

\begin{corollary}
\label{TRM-veza 2}Let\textit{\ }$P\left(  x\right)  =x^{d}+p_{d-1}%
x^{d-1}+...+p_{1}x+p_{0}\in\mathbb{Z}\left[  x\right]  ,$ $\left\vert
p_{0}\right\vert \geq2,$ $\varepsilon\in\left[  0,1\right)  $ and
$k=\left\lfloor \varepsilon\left\vert p_{0}\right\vert \right\rfloor .$
\textit{Then }$P$\textit{\ is }$\varepsilon-$CNS polynomial\textit{\ }if and
only if $P$\textit{\ is }$\varepsilon_{k}-$CNS polynomial, where
$\varepsilon_{k}=\frac{k}{\left\vert p_{0}\right\vert }.$
\end{corollary}

\begin{proof}
By definition of $\mathcal{\varepsilon-}$CNS polynomials, it is obvious that
$P$ is\textit{\ }$\mathcal{\varepsilon-}$CNS polynomial\textit{\ }if and only
if $P$ is\textit{\ }$\mathcal{\varepsilon}_{k}\mathcal{-}$CNS polynomial since
$\mathcal{N}_{\varepsilon}\mathbb{=}\mathcal{N}_{\varepsilon_{k}}.$
\end{proof}

\begin{proposition}
\label{Lem1}Let $\left(  p_{0},p_{1}\right)  \in\mathbb{Z}^{2},$ $\left\vert
p_{0}\right\vert \geq2$, $\varepsilon\in\left[  0,\frac{1}{2}\right)  $ or let
$\varepsilon=\frac{1}{2}$ if $\left\vert p_{0}\right\vert $ is odd. Let
$k=\left\lfloor \mathcal{\varepsilon}\left\vert p_{0}\right\vert \right\rfloor
$ and let $D\left(  \varepsilon\right)  $ be given by (\ref{D}). Then $\left(
\frac{1}{p_{0}},\frac{p_{1}}{p_{0}}\right)  \in D\left(  \varepsilon\right)  $
if and only if the pair $\left(  p_{0},p_{1}\right)  $ satisfies the
conditions (\ref{prva1}) or (\ref{druga 1}). Consequently, $\left(  \frac
{1}{p_{0}},\frac{p_{1}}{p_{0}}\right)  \in D\left(  \varepsilon\right)  $ if
and only if $\left(  \frac{1}{p_{0}},\frac{p_{1}}{p_{0}}\right)  \in D\left(
\varepsilon_{k}\right)  ,$ where $\varepsilon_{k}=\frac{k}{\left\vert
p_{0}\right\vert }.$
\end{proposition}

\begin{proof}
\textit{i) }By the definition of the set $D\left(  \varepsilon\right)  $ given
by (\ref{D}), we see that we have to consider the case $\varepsilon\in\left[
0,\frac{1}{2}\right)  $ and the case $\varepsilon=\frac{1}{2}$ and $\left\vert
p_{0}\right\vert $ odd separately. Note also that if $\varepsilon\in\left[
0,\frac{1}{2}\right)  $ or if $\varepsilon=\frac{1}{2} $ and $\left\vert
p_{0}\right\vert $ is odd, then $k\in K_{\left\vert p_{0}\right\vert },$ where%
\begin{equation}
K_{\left\vert p_{0}\right\vert }=\left\{
\begin{tabular}
[c]{ll}%
$\left\{  0,1,...,\frac{\left\vert p_{0}\right\vert -1}{2}\right\}  ,$ &
$\text{if }\left\vert p_{0}\right\vert \text{ is odd}$\\
$\left\{  0,1,...,\frac{\left\vert p_{0}\right\vert }{2}-1\right\}  ,$ &
$\text{if }\left\vert p_{0}\right\vert \text{ is }$even
\end{tabular}
\right.  \text{.} \label{Kp0}%
\end{equation}
In particular, $\varepsilon\in\left[  \frac{1}{\left\vert p_{0}\right\vert
},\frac{1}{2}\right)  $ implies $k\in K_{\left\vert p_{0}\right\vert
}\diagdown\left\{  0\right\}  .$ Also, $K_{\left\vert p_{0}\right\vert
}=\left\{  0\right\}  $ if and only if $\left\vert p_{0}\right\vert =2.$

\textit{Case }$\varepsilon\in\left[  0,\frac{1}{2}\right)  .$ Let
$\varepsilon\in\left[  0,\frac{1}{2}\right)  $. First assume $\left(  \frac
{1}{p_{0}},\frac{p_{1}}{p_{0}}\right)  \in D\left(  \varepsilon\right)  .$
Then%
\begin{equation}
-\varepsilon\leq\frac{1}{p_{0}}<1-\varepsilon\text{ \ \ and \ \ }-\frac
{1}{p_{0}}-\varepsilon\leq\frac{p_{1}}{p_{0}}<\frac{1}{p_{0}}+1-\varepsilon.
\label{n1}%
\end{equation}
Note that the first inequality in (\ref{n1}) implies $-\varepsilon\left\vert
p_{0}\right\vert \leq sgn\left(  p_{0}\right)  <\left(  1-\varepsilon\right)
\left\vert p_{0}\right\vert .$ So, if $\varepsilon\in\left[  0,\frac
{1}{\left\vert p_{0}\right\vert }\right)  $, then $p_{0}$ must be a positive
integer, while if $\varepsilon\in\left[  \frac{1}{\left\vert p_{0}\right\vert
},\frac{1}{2}\right)  $, then $p_{0}$ can take positive and negative values.
The first inequality in (\ref{n1}) implies%
\begin{equation}
-\frac{k+1}{\left\vert p_{0}\right\vert }\leq-\varepsilon_{k+1}<-\varepsilon
\leq\frac{1}{p_{0}}<1-\varepsilon\leq1-\varepsilon_{k}=1-\frac{k}{\left\vert
p_{0}\right\vert }, \label{prv-n1}%
\end{equation}
while the second inequality in (\ref{n1}) implies%
\begin{equation}
-\frac{1}{p_{0}}-\frac{k+1}{\left\vert p_{0}\right\vert }<-\frac{1}{p_{0}%
}-\varepsilon\leq\frac{p_{1}}{p_{0}}<\frac{1}{p_{0}}+1-\varepsilon\leq\frac
{1}{p_{0}}+1-\frac{k}{\left\vert p_{0}\right\vert }. \label{dug-n1}%
\end{equation}
Let $p_{0}>0.$ Then (\ref{prv-n1}) implies $0<\frac{1}{p_{0}}<1-\frac{k}%
{p_{0}}$, which again implies $p_{0}>k+1$, i.e. $p_{0}\geq k+2\geq2.$ Since
$p_{0}>0,$ then (\ref{dug-n1}) implies $-k-2<p_{1}<p_{0}-k+1$, or equivalently
$-k-1\leq p_{1}\leq p_{0}-k.$ Therefore, we obtain that the pair $\left(
p_{0},p_{1}\right)  $ satisfies the condition (\ref{prva1}).

Let $p_{0}<0.$ Then (\ref{prv-n1}) implies $-\frac{k+1}{\left\vert
p_{0}\right\vert }<\frac{1}{p_{0}}<0$ , which again implies $k>0$, i.e.
$k\geq1.$ Since $K_{\left\vert p_{0}\right\vert }=\left\{  0\right\}  $ only
if $\left\vert p_{0}\right\vert =2,$ then $k\geq1$ implies $p_{0}\leq-3.$
Since $p_{0}\ $is a negative integer, then multiplying the inequality
(\ref{dug-n1}) by $p_{0}$, we obtain $-\left\vert p_{0}\right\vert
+k+1<p_{1}<k,$ which implies $-\left\vert p_{0}\right\vert +k+2\leq p_{1}\leq
k-1.$ Note that $k\geq1$ if and only if $\varepsilon\in\left[  \frac
{1}{\left\vert p_{0}\right\vert },\frac{1}{2}\right)  .$ Thus, we obtain that
the pair $\left(  p_{0},p_{1}\right)  $ satisfies the conditions in
(\ref{druga 1}). Therefore, we conclude: if $\left(  \frac{1}{p_{0}}%
,\frac{p_{1}}{p_{0}}\right)  \in D\left(  \varepsilon\right)  $, then the pair
$\left(  p_{0},p_{1}\right)  $ satisfies the conditions (\ref{prva1}) or
(\ref{druga 1})$.$

Suppose now that the pair $\left(  p_{0},p_{1}\right)  $ fulfills the
conditions (\ref{prva1}) or (\ref{druga 1})$.$

Let us first assume $\left(  p_{0},p_{1}\right)  $ fulfills condition
(\ref{prva1}). Since for every $p_{0}\geq2,$ inequality $p_{0}\geq k+2$ holds
for all $k\in K_{\left\vert p_{0}\right\vert }$, then%
\[
p_{0}\geq2\Longrightarrow p_{0}\geq k+2\Longrightarrow1\geq\frac{k+1}{p_{0}%
}+\frac{1}{p_{0}}=\varepsilon_{k+1}+\frac{1}{p_{0}}>\varepsilon+\frac{1}%
{p_{0}}.
\]
Therefore, we obtain $0<\frac{1}{p_{0}}<1-\varepsilon.$ Since $\varepsilon$ is
a nonnegative integer, we have $-\varepsilon\leq\frac{1}{p_{0}}<1-\varepsilon
.$ On the other hand, since $p_{0}\geq2$ is a positive integer, then $-k-1\leq
p_{1}\leq p_{0}-k$ implies%
\[
-\varepsilon-\frac{1}{p_{0}}\leq-\frac{k}{p_{0}}-\frac{1}{p_{0}}\leq
\frac{p_{1}}{p_{0}}\leq1-\frac{k+1}{p_{0}}+\frac{1}{p_{0}}<1-\varepsilon
+\frac{1}{p_{0}}.
\]
Therefore, if $\left(  p_{0},p_{1}\right)  $ satisfies condition
(\ref{prva1}), then $\left(  \frac{1}{p_{0}},\frac{p_{1}}{p_{0}}\right)  \in
D\left(  \varepsilon\right)  $.

Let us assume that $\left(  p_{0},p_{1}\right)  $ fulfills the condition
(\ref{druga 1}). First note that $\varepsilon\in\left[  \frac{1}{\left\vert
p_{0}\right\vert },\frac{1}{2}\right)  \ $implies $k\geq1$ and $1-\varepsilon
>0$. Therefore, as $p_{0}\leq-3$ is a negative integer, we have%
\[
k\geq1\Longrightarrow-\frac{k}{\left\vert p_{0}\right\vert }\leq\frac
{-1}{\left\vert p_{0}\right\vert }\Longrightarrow-\varepsilon\leq
-\varepsilon_{k}\leq\frac{-1}{\left\vert p_{0}\right\vert }=\frac{1}{p_{0}%
}<1-\varepsilon.
\]
Also, since $p_{0}\leq-3$ is a negative integer, then $k+2-\left\vert
p_{0}\right\vert \leq p_{1}\leq k-1$ implies%
\[
-\varepsilon-\frac{1}{p_{0}}\leq-\frac{k}{\left\vert p_{0}\right\vert }%
-\frac{1}{p_{0}}\leq\frac{p_{1}}{p_{0}}\leq\frac{1}{p_{0}}+1-\frac
{k+1}{\left\vert p_{0}\right\vert }<\frac{1}{p_{0}}+1-\varepsilon
\]
So, if $\left(  p_{0},p_{1}\right)  $ satisfies condition (\ref{druga 1}),
then $\left(  \frac{1}{p_{0}},\frac{p_{1}}{p_{0}}\right)  \in D\left(
\varepsilon\right)  $. This concludes the proof for $\varepsilon\in\left[
0,\frac{1}{2}\right)  .$

\textit{Case }$\varepsilon=\frac{1}{2}\ $\textit{and }$\left\vert
p_{0}\right\vert $\textit{\ odd.}\textbf{\ }Let $\varepsilon=\frac{1}{2}$ and
let $\left\vert p_{0}\right\vert $ be odd. Suppose that $\left(  \frac
{1}{p_{0}},\frac{p_{1}}{p_{0}}\right)  \in D\left(  \frac{1}{2}\right)  .$
Then
\begin{equation}
-\frac{1}{2}<\frac{1}{p_{0}}\leq\frac{1}{2}\text{ \ \ and \ \ }-\frac{1}%
{p_{0}}-\frac{1}{2}<\frac{p_{1}}{p_{0}}\leq\frac{1}{p_{0}}+\frac{1}{2}.
\label{n1-pola}%
\end{equation}
Note that $-\frac{1}{2}<\frac{1}{p_{0}}\leq\frac{1}{2}$ implies $p_{0}\geq2$
or $p_{0}\leq-3.$

Let $p_{0}\geq2$. Since $p_{0}$ is an odd positive integer, then
$k=\frac{p_{0}-1}{2}$, and the second inequality in (\ref{n1-pola}) implies%
\begin{equation}
-k-1-\frac{1}{2}=-1-\frac{p_{0}}{2}<p_{1}\leq1+\frac{p_{0}}{2}=p_{0}%
-k+\frac{1}{2}\text{.} \label{nej-pola}%
\end{equation}
Since $p_{0}$, $p_{1}$ and $k$ are integers, (\ref{nej-pola}) implies
$-k-1\leq p_{1}\leq p_{0}-k.$

Let $p_{0}\leq-3.$ Since $p_{0}$ is an odd negative integer, then
$k=\frac{-p_{0}-1}{2},$ and the second inequality in (\ref{n1-pola}) implies
\begin{equation}
-\left\vert p_{0}\right\vert +k+1+\frac{1}{2}=1+\frac{p_{0}}{2}\leq
p_{1}<-1-\frac{p_{0}}{2}=k-\frac{1}{2}\text{.} \label{nej-pola-2}%
\end{equation}
From (\ref{nej-pola-2}) we get $-\left\vert p_{0}\right\vert +k+2\leq
p_{1}\leq k-1.\ $Note also that $\frac{1}{2}\in\left[  \frac{1}{\left\vert
p_{0}\right\vert },\frac{1}{2}\right]  $. Therefore, $\left(  \frac{1}{p_{0}%
},\frac{p_{1}}{p_{0}}\right)  \in D\left(  \frac{1}{2}\right)  ,\ $where
$\left\vert p_{0}\right\vert $\textbf{\ }is\textbf{\ }odd, implies that the
pair $\left(  p_{0},p_{1}\right)  $ satisfies the conditions (\ref{prva1}) or
(\ref{druga 1})$.$

Suppose the pair $\left(  p_{0},p_{1}\right)  \in\mathbb{Z}^{2}$, where
$\left\vert p_{0}\right\vert $\textbf{\ }is\textbf{\ }odd, satisfies
conditions in (\ref{prva1}) or in (\ref{druga 1})$\ $for $\varepsilon=\frac
{1}{2}.$ Then%
\[
p_{0}\geq2\text{ \ or \ }p_{0}\leq-3\text{ }\Longrightarrow-\frac{1}{2}%
<\frac{1}{p_{0}}\leq\frac{1}{2}.
\]
Since $k=\frac{\left\vert p_{0}\right\vert -1}{2}$, then (\ref{prva1}) and
(\ref{druga 1}) are of the form
\begin{align*}
-\frac{1}{2}-\frac{p_{0}}{2}  &  \leq p_{1}\leq\frac{p_{0}}{2}+\frac{1}%
{2},\text{ \ }p_{0}\geq2\text{ }\\
\frac{1}{2}p_{0}+\frac{3}{2}  &  \leq p_{1}\leq-\frac{1}{2}p_{0}-\frac{3}%
{2},\text{ \ }p_{0}\leq-3,
\end{align*}
respectively. From the above inequalities, we derive the following:
\begin{align*}
-\frac{1}{p_{0}}-\frac{1}{2}  &  <-\frac{1}{2}-\frac{1}{2p_{0}}\leq\frac
{p_{1}}{p_{0}}\leq\frac{1}{2}+\frac{1}{2p_{0}}<\frac{1}{2}+\frac{1}{p_{0}%
},\text{ if \ }p_{0}\geq2\text{ }\\
-\frac{1}{p_{0}}-\frac{1}{2}  &  <-\frac{1}{2}-\frac{3}{2p_{0}}\leq\frac
{p_{1}}{p_{0}}\leq\frac{1}{2}+\frac{3}{2p_{0}}\leq\frac{1}{2}+\frac{1}{p_{0}%
},\text{ if \ }p_{0}\leq-3,
\end{align*}
so we conclude $\left(  \frac{1}{p_{0}},\frac{p_{1}}{p_{0}}\right)  \in
D\left(  \frac{1}{2}\right)  .$

\textit{ii)} We that note if $\varepsilon\in\left[  0,\frac{1}{2}\right)  $ or
if $\varepsilon=\frac{1}{2}$ and $\left\vert p_{0}\right\vert $ is odd, then
$\varepsilon_{k}\in\left[  0,\frac{1}{2}\right)  $ and $\left\lfloor
\mathcal{\varepsilon}_{k}\left\vert p_{0}\right\vert \right\rfloor =k.$
Therefore, by \textit{i)} we have: $\left(  \frac{1}{p_{0}},\frac{p_{1}}%
{p_{0}}\right)  \in D\left(  \varepsilon\right)  $ if and only if $\left(
\frac{1}{p_{0}},\frac{p_{1}}{p_{0}}\right)  \in D\left(  \varepsilon
_{k}\right)  .$
\end{proof}

\begin{proposition}
\label{Lem3}Let $\left(  p_{0},p_{1}\right)  \in\mathbb{Z}^{2},$ $\left\vert
p_{0}\right\vert \geq2.$ Let $\varepsilon\in\left(  \frac{1}{2},1\right)  $ or
let $\varepsilon=\frac{1}{2}$ if $\left\vert p_{0}\right\vert $ is even. Let
$k=\left\lfloor \mathcal{\varepsilon}\left\vert p_{0}\right\vert \right\rfloor
$ and let $D\left(  \varepsilon\right)  $ be given by (\ref{D}). Then $\left(
\frac{1}{p_{0}},\frac{p_{1}}{p_{0}}\right)  \in D\left(  \varepsilon\right)  $
if and only if the pair $\left(  p_{0},p_{1}\right)  $ satisfies the
conditions (\ref{prva2}) or (\ref{druga2}). Consequently, $\left(  \frac
{1}{p_{0}},\frac{p_{1}}{p_{0}}\right)  \in D\left(  \varepsilon\right)  $ if
and only if $\left(  \frac{1}{p_{0}},\frac{p_{1}}{p_{0}}\right)  \in D\left(
\varepsilon_{k}\right)  ,$ where $\varepsilon_{k}=\frac{k}{\left\vert
p_{0}\right\vert }.$
\end{proposition}

\begin{proof}
\textit{i)} Note first that in this case $k\in K_{\left\vert p_{0}\right\vert
}^{\prime}$, where%
\[
K_{\left\vert p_{0}\right\vert }^{\prime}=\left\{
\begin{tabular}
[c]{ll}%
$\left\{  \frac{\left\vert p_{0}\right\vert }{2},...,\left\vert p_{0}%
\right\vert -1\right\}  ,$ & if $\varepsilon\in\left[  \frac{1}{2},1\right)  $
and $\left\vert p_{0}\right\vert $ is even\\
$\left\{  \frac{\left\vert p_{0}\right\vert +1}{2},...,\left\vert
p_{0}\right\vert -1\right\}  ,$ & $\text{if }\varepsilon\in\left[  \frac{1}%
{2}+\frac{1}{2\left\vert p_{0}\right\vert },1\right)  $ and $\left\vert
p_{0}\right\vert \text{ is }$odd\\
$\left\{  \frac{\left\vert p_{0}\right\vert -1}{2}\right\}  ,$ & $\text{if
}\varepsilon\in\left(  \frac{1}{2},\frac{1}{2}+\frac{1}{2\left\vert
p_{0}\right\vert }\right)  $ and $\left\vert p_{0}\right\vert $ is odd
\end{tabular}
\ \right.  .
\]
We note that $k\geq1$ for all $\left\vert p_{0}\right\vert \geq2.$ In
particular, if $\left\vert p_{0}\right\vert =2,$ then $K_{\left\vert
p_{0}\right\vert }^{\prime}=\left\{  1\right\}  .$ Consequently,
$\varepsilon_{k}\in\left[  \frac{1}{2},1\right)  $ except if $\varepsilon
\in\left(  \frac{1}{2},\frac{1}{2}+\frac{1}{2\left\vert p_{0}\right\vert
}\right)  $ and $\left\vert p_{0}\right\vert $ is odd. In this case,
$\varepsilon_{k}=\frac{1}{2}-\frac{1}{2\left\vert p_{0}\right\vert }\in\left[
0,\frac{1}{2}\right)  .$

Let $\left(  \frac{1}{p_{0}},\frac{p_{1}}{p_{0}}\right)  \in D\left(
\varepsilon\right)  .$ Then%
\begin{equation}
-\left(  1-\varepsilon\right)  <\frac{1}{p_{0}}\leq\varepsilon\text{ \ \ and
\ \ }-\frac{1}{p_{0}}-1+\varepsilon<\frac{p_{1}}{p_{0}}\leq\frac{1}{p_{0}%
}+\varepsilon. \label{n1-2}%
\end{equation}
Note that (\ref{n1-2}) implies $-\left(  1-\varepsilon\right)  \left\vert
p_{0}\right\vert <sgn\left(  p_{0}\right)  \leq\varepsilon\left\vert
p_{0}\right\vert .$ Therefore, if $\varepsilon\in\left[  \frac{\left\vert
p_{0}\right\vert -1}{\left\vert p_{0}\right\vert },1\right)  $, then $p_{0}$
must be a positive integer, while if $\varepsilon\in\left[  \frac{1}{2}%
,\frac{\left\vert p_{0}\right\vert -1}{\left\vert p_{0}\right\vert }\right)
$, then $p_{0}$ can take positive and negative values.

Let $p_{0}>0.$ Then the first inequality in (\ref{n1-2}) implies $0<\frac
{1}{p_{0}}\leq\varepsilon.$ It is obvious that $0<\frac{1}{p_{0}}%
\leq\varepsilon$ holds for all $\varepsilon\in\left[  \frac{1}{2},1\right)  $
and all $p_{0}\geq2.$ Therefore, the first inequality in (\ref{n1-2}) implies
$p_{0}\geq2.$ The second inequality in (\ref{n1-2}) yields%
\[
-\frac{1}{p_{0}}-1+\frac{k}{p_{0}}\leq-\frac{1}{p_{0}}-1+\varepsilon
<\frac{p_{1}}{p_{0}}\leq\frac{1}{p_{0}}+\varepsilon<\frac{1}{p_{0}}+\frac
{k+1}{p_{0}},
\]
which implies $-p_{0}+k-1<p_{1}<k+2,$ or equivalently $-p_{0}+k\leq p_{1}\leq
k+1.$

Let $p_{0}<0.$ Then the first inequality in (\ref{n1-2}) implies
$-1+\varepsilon<\frac{1}{p_{0}}<0,$ and we derive
\[
-1-\frac{k}{p_{0}}=-1+\varepsilon_{k}\leq-1+\varepsilon<\frac{1}{p_{0}}<0,
\]
which implies $k<-p_{0}-1=\left\vert p_{0}\right\vert -1,$ or equivalently
$k\leq\left\vert p_{0}\right\vert -2.$ Note that $k\leq\left\vert
p_{0}\right\vert -2$ implies $\varepsilon\in\left[  \frac{1}{2},\frac
{\left\vert p_{0}\right\vert -1}{\left\vert p_{0}\right\vert }\right)  .$
Since $k\geq1$, then $k\leq\left\vert p_{0}\right\vert -2$ also implies
$\left\vert p_{0}\right\vert \geq3$, i.e. $p_{0}\leq-3.$ Since $p_{0}\ $is a
negative integer$,$ from the second inequality in (\ref{n1-2}), we obtain%
\[
-\frac{1}{p_{0}}-1-\frac{k}{p_{0}}\leq-\frac{1}{p_{0}}-1+\varepsilon
<\frac{p_{1}}{p_{0}}\leq\frac{1}{p_{0}}+\varepsilon<\frac{1}{p_{0}}-\frac
{k+1}{p_{0}}%
\]
which implies $-k<p_{1}<-1-p_{0}-k,$ or $-k+1\leq p_{1}\leq\left\vert
p_{0}\right\vert -k-2.$ Thus, if $\left(  \frac{1}{p_{0}},\frac{p_{1}}{p_{0}%
}\right)  \in D\left(  \varepsilon\right)  ,$ then the pair $\left(
p_{0},p_{1}\right)  $ satisfies the conditions (\ref{prva2}) or (\ref{druga2}%
)$.$

Let $\varepsilon\in\left[  \frac{1}{2},1\right)  $ and let the pair $\left(
p_{0},p_{1}\right)  \in\mathbb{Z}^{2},$ $\left\vert p_{0}\right\vert \geq2$
satisfies the conditions (\ref{prva2}) or (\ref{druga2})$.$

Suppose that $\left(  p_{0},p_{1}\right)  $ fulfills the condition
(\ref{prva2}). Since $p_{0}\geq2\ $and $\varepsilon\in\left[  \frac{1}%
{2},1\right)  $, then we have $-\left(  1-\varepsilon\right)  <\frac{1}{p_{0}%
}\leq\frac{1}{2}\leq\varepsilon.$ On the other hand, $-p_{0}+k\leq p_{1}\leq
k+1$ implies%
\[
-\frac{1}{p_{0}}-1+\varepsilon<-\frac{1}{p_{0}}-1+\frac{k+1}{p_{0}}\leq
\frac{p_{1}}{p_{0}}\leq\frac{1}{p_{0}}+\frac{k}{p_{0}}\leq\frac{1}{p_{0}%
}+\varepsilon.
\]
So, if $\left(  p_{0},p_{1}\right)  $ satisfies the condition (\ref{prva2}),
then $\left(  \frac{1}{p_{0}},\frac{p_{1}}{p_{0}}\right)  \in D\left(
\varepsilon\right)  $.

Suppose that $\left(  p_{0},p_{1}\right)  $ fulfills the condition
(\ref{druga2}). First note that $\varepsilon\in\left[  \frac{1}{2}%
,\frac{\left\vert p_{0}\right\vert -1}{\left\vert p_{0}\right\vert }\right)
\ $implies $\frac{1}{\left\vert p_{0}\right\vert }<1-\varepsilon\leq\frac
{1}{2}.$ Therefore, $p_{0}\leq-3\ $and $\varepsilon\in\left[  \frac{1}%
{2},\frac{\left\vert p_{0}\right\vert -1}{\left\vert p_{0}\right\vert
}\right)  $ implies
\[
-\left(  1-\varepsilon\right)  <-\frac{1}{\left\vert p_{0}\right\vert }%
=\frac{1}{p_{0}}\leq\varepsilon.
\]
Since $p_{0}\ $is a negative integer$,$ then $-k+1\leq p_{1}\leq
-k-2+\left\vert p_{0}\right\vert $ implies%
\[
-\frac{1}{p_{0}}-1+\varepsilon<-\frac{1}{p_{0}}-1-\frac{k+1}{p_{0}}\leq
\frac{p_{1}}{p_{0}}\leq-\frac{k}{p_{0}}+\frac{1}{p_{0}}=\varepsilon_{k}%
+\frac{1}{p_{0}}\leq\frac{1}{p_{0}}+\varepsilon.
\]
So, if $\left(  p_{0},p_{1}\right)  $ satisfies the condition (\ref{druga2}),
then $\left(  \frac{1}{p_{0}},\frac{p_{1}}{p_{0}}\right)  \in D\left(
\varepsilon\right)  $.

\textit{ii)} For all $\varepsilon\in\left(  \frac{1}{2},1\right)  $ or for
$\varepsilon=\frac{1}{2}$ and $\left\vert p_{0}\right\vert $ is even, we have
$\varepsilon_{k}\in\left[  \frac{1}{2},1\right)  $ except for $\varepsilon
\in\left(  \frac{1}{2},\frac{1}{2}+\frac{1}{2\left\vert p_{0}\right\vert
}\right)  $ and $\left\vert p_{0}\right\vert $ odd. Therefore, by \textit{i)}
in all cases where $\varepsilon_{k}\in\left[  \frac{1}{2},1\right)  ,$ we
have: $\left(  \frac{1}{p_{0}},\frac{p_{1}}{p_{0}}\right)  \in D\left(
\varepsilon\right)  $ if and only if $\left(  \frac{1}{p_{0}},\frac{p_{1}%
}{p_{0}}\right)  \in D\left(  \varepsilon_{k}\right)  .$

Let $\varepsilon\in\left(  \frac{1}{2},\frac{1}{2}+\frac{1}{2\left\vert
p_{0}\right\vert }\right)  $ and $\left\vert p_{0}\right\vert $ be odd. Then
$\varepsilon_{k}=\frac{1}{2}-\frac{1}{2\left\vert p_{0}\right\vert }$ and
$k=\frac{\left\vert p_{0}\right\vert -1}{2}.$ Since $\varepsilon_{k}\in\left[
0,\frac{1}{2}\right)  $ and also $\varepsilon_{k}\in\left[  \frac
{1}{\left\vert p_{0}\right\vert },\frac{1}{2}\right]  $ for all $p_{0}\leq-3,$
then by Proposition \ref{Lem1}, we have: $\left(  \frac{1}{p_{0}},\frac{p_{1}%
}{p_{0}}\right)  \in D\left(  \varepsilon_{k}\right)  $ if and only if
$-k-1\leq p_{1}\leq p_{0}-k,$ \ $p_{0}\geq2$ \ or $k+2-\left\vert
p_{0}\right\vert \leq p_{1}\leq k-1,$ \ $p_{0}\leq-3,\ $where $k=\frac
{\left\vert p_{0}\right\vert -1}{2}.$ On the other hand, if $\varepsilon
\in\left(  \frac{1}{2},\frac{1}{2}+\frac{1}{2\left\vert p_{0}\right\vert
}\right)  $, then $\varepsilon\in\left[  \frac{1}{2},1\right)  $ and also
$\varepsilon\in\left[  \frac{1}{2},\frac{\left\vert p_{0}\right\vert
-1}{\left\vert p_{0}\right\vert }\right)  $ for all $p_{0}\leq-3.$ Now, by
\textit{i)} it follows: $\left(  \frac{1}{p_{0}},\frac{p_{1}}{p_{0}}\right)
\in D\left(  \varepsilon\right)  $ if and only if $-p_{0}+k\leq p_{1}\leq
k+1,$ \ $p_{0}\geq2$ or $-k+1\leq p_{1}\leq-k-2+\left\vert p_{0}\right\vert ,$
\ $p_{0}\leq-3$ $\ $where $k=\left\lfloor \mathcal{\varepsilon}\left\vert
p_{0}\right\vert \right\rfloor =\frac{\left\vert p_{0}\right\vert -1}{2}.$
Since, for $k=\frac{\left\vert p_{0}\right\vert -1}{2}$ and $\left\vert
p_{0}\right\vert $ odd, the conditions $-k-1\leq p_{1}\leq p_{0}-k,\ p_{0}%
\geq2\ $and $-p_{0}+k\leq p_{1}\leq k+1,p_{0}\geq2$ both become%
\[
-\frac{p_{0}+1}{2}\leq p_{1}\leq\frac{p_{0}+1}{2},\text{ \ }p_{0}\geq3\text{
,}%
\]
and $k+2-\left\vert p_{0}\right\vert \leq p_{1}\leq k-1,$ $p_{0}\leq-3\ $and
$-k+1\leq p_{1}\leq-k-2+\left\vert p_{0}\right\vert ,$ $p_{0}\leq-3$ both
become%
\[
-\frac{\left\vert p_{0}\right\vert +3}{2}\leq p_{1}\leq\frac{\left\vert
p_{0}\right\vert -3}{2},\text{ \ }p_{0}\leq-3,
\]
we conclude that in this case we also have: $\left(  \frac{1}{p_{0}}%
,\frac{p_{1}}{p_{0}}\right)  \in D\left(  \varepsilon\right)  $ if and only if
$\left(  \frac{1}{p_{0}},\frac{p_{1}}{p_{0}}\right)  \in D\left(
\varepsilon_{k}\right)  .$
\end{proof}

\begin{proposition}
\label{Lem2}Let $\left(  p_{0},p_{1}\right)  \in\mathbb{Z}^{2},$ $\left\vert
p_{0}\right\vert \geq2$ and $\varepsilon\in\left[  0,1\right)  $. Let
$k=\left\lfloor \mathcal{\varepsilon}\left\vert p_{0}\right\vert \right\rfloor
,$ $\varepsilon_{k}=\frac{k}{\left\vert p_{0}\right\vert }\ $ and let the sets
$B\left(  \varepsilon\right)  $ and $D\left(  \varepsilon\right)  $ be given
by (\ref{B}) and (\ref{D}) respectively. Then we have:\ 

\begin{enumerate}
\item[i)] If $\left(  \frac{1}{p_{0}},\frac{p_{1}}{p_{0}}\right)  \in D\left(
\varepsilon\right)  $ and $p_{0}\neq2,3,4,5,$ than $\left(  \frac{1}{p_{0}%
},\frac{p_{1}}{p_{0}}\right)  \in B\left(  \varepsilon\right)  .$

\item[ii)] If $\left(  \frac{1}{p_{0}},\frac{p_{1}}{p_{0}}\right)  \in
D\left(  \varepsilon\right)  $, than $\left(  \frac{1}{p_{0}},\frac{p_{1}%
}{p_{0}}\right)  \in B\left(  \varepsilon_{k}\right)  ,$ except for $p_{0}=2$
or $4$ and $\varepsilon\in\left[  \frac{1}{2},\frac{1}{2}+\frac{1}{\left\vert
p_{0}\right\vert }\right)  .$
\end{enumerate}
\end{proposition}

\begin{proof}
i) For each $\varepsilon\in\left[  0,1\right)  ,$ the sets $B(\varepsilon)$
and $D(\varepsilon)$ are defined by the same inequalities except for the upper
bound for $x$. Then $\left(  \frac{1}{p_{0}},\frac{p_{1}}{p_{0}}\right)  \in
D\left(  \varepsilon\right)  \setminus B\left(  \varepsilon\right)  $ if and
only if
\begin{gather}
\frac{2}{3}-\varepsilon<\frac{1}{p_{0}}<1-\varepsilon,\text{ \ for
\ }\varepsilon\in\left[  0,\frac{1}{2}\right)  ,\label{nejizm}\\
\frac{2}{3}-\left(  1-\varepsilon\right)  <\frac{1}{p_{0}}\leq\varepsilon
,\text{ \ for }\varepsilon\in\left[  \frac{1}{2},1\right)  . \label{nejizm2}%
\end{gather}
Since, (\ref{nejizm}) or (\ref{nejizm2}) hold only for $p_{0}=2$ \ and
$\varepsilon\in\left(  \frac{1}{6},\frac{5}{6}\right)  ,\ $or for $p_{0}=3$
and $\varepsilon\in\left(  \frac{1}{3},\frac{2}{3}\right)  ,$ or $p_{0}=4$
\ and $\varepsilon\in\left(  \frac{5}{12},\frac{7}{12}\right)  ,$ or $p_{0}=5$
\ and $\varepsilon\in\left(  \frac{7}{15},\frac{8}{15}\right)  $, we obtain
the statement i).

ii) First, let us recall that by Propositions \ref{Lem1} and \ref{Lem3} we
have $\left(  \frac{1}{p_{0}},\frac{p_{1}}{p_{0}}\right)  \in D\left(
\varepsilon\right)  $ if and only if $\left(  \frac{1}{p_{0}},\frac{p_{1}%
}{p_{0}}\right)  \in D\left(  \varepsilon_{k}\right)  $ for all $\varepsilon
\in\left[  0,1\right)  .$ Since the sets $B(\varepsilon_{k})$ and
$D(\varepsilon_{k})$ are defined by the same inequalities except for the upper
bound for $x,$ it follows that $\left(  \frac{1}{p_{0}},\frac{p_{1}}{p_{0}%
}\right)  \in D\left(  \varepsilon_{k}\right)  \setminus B\left(
\varepsilon_{k}\right)  $ if and only if
\begin{align}
\frac{2}{3}-\frac{k}{\left\vert p_{0}\right\vert }  &  <\frac{1}{p_{0}%
}<1-\frac{k}{\left\vert p_{0}\right\vert },\text{ \ for }\varepsilon_{k}%
\in\left[  0,\frac{1}{2}\right)  ,\label{ne1}\\
\frac{2}{3}-\left(  1-\frac{k}{\left\vert p_{0}\right\vert }\right)   &
<\frac{1}{p_{0}}\leq\frac{k}{\left\vert p_{0}\right\vert },\text{ \ for
}\varepsilon_{k}\in\left[  \frac{1}{2},1\right)  . \label{ne2}%
\end{align}
Also recall that for each $\varepsilon\in\left[  0,\frac{1}{2}\right)  $
corresponding $\varepsilon_{k}\in\left[  0,\frac{1}{2}\right)  ,\ $and for
each $\varepsilon\in\left[  \frac{1}{2},1\right)  $ corresponding
$\varepsilon_{k}\in\left[  \frac{1}{2},1\right)  $ except for $\varepsilon
\in\left[  \frac{1}{2},\frac{1}{2}+\frac{1}{2\left\vert p_{0}\right\vert
}\right)  $ and odd $\left\vert p_{0}\right\vert $. In this case
$\varepsilon_{k}=\frac{1}{2}-\frac{1}{2\left\vert p_{0}\right\vert }\in\left[
0,\frac{1}{2}\right)  .$

a) Let $\varepsilon\in\left[  0,\frac{1}{2}\right)  .$ Then $\varepsilon
_{k}=\frac{k}{\left\vert p_{0}\right\vert }\in\left[  0,\frac{1}{2}\right)
,\ $and (\ref{ne1}) implies that $p_{0}$ is a positive integer and
$k+1<p_{0}<\frac{3}{2}k+\frac{3}{2}.$ In this case $k\in K_{\left\vert
p_{0}\right\vert }$, where $K_{\left\vert p_{0}\right\vert }$ is given by
(\ref{Kp0}). Since for each $p_{0}\geq2,$ we have $p_{0}\geq\frac{3}{2}%
k+\frac{3}{2}$ for all $k\in K_{\left\vert p_{0}\right\vert },$ there is no
$p_{0}$ that satisfies (\ref{ne1}). Thus, statement ii) holds if
$\varepsilon\in\left[  0,\frac{1}{2}\right)  .$

b) Let $\varepsilon\in\left[  \frac{1}{2},\frac{1}{2}+\frac{1}{2\left\vert
p_{0}\right\vert }\right)  $ and let $\left\vert p_{0}\right\vert $ be odd.
Then $\varepsilon_{k}\in\left[  0,\frac{1}{2}\right)  $ and $k=\frac
{\left\vert p_{0}\right\vert -1}{2}$. In this case, (\ref{ne1}) becomes%
\[
\frac{1}{6}+\frac{1}{2\left\vert p_{0}\right\vert }<\frac{1}{p_{0}}<\frac
{1}{2}+\frac{1}{2\left\vert p_{0}\right\vert }%
\]
which holds only for $p_{0}=2.$ Since $\left\vert p_{0}\right\vert $ is odd,
we conclude that statement ii) holds in this case too.

c) Let $\varepsilon\in\left[  \frac{1}{2}+\frac{1}{2\left\vert p_{0}%
\right\vert },1\right)  $ and $\left\vert p_{0}\right\vert $ be odd or let
$\varepsilon\in\left[  \frac{1}{2},1\right)  $ and $\left\vert p_{0}%
\right\vert $ be even. Then $\varepsilon_{k}=\frac{k}{\left\vert
p_{0}\right\vert }\in\left[  \frac{1}{2},1\right)  .$ Since $\frac
{k}{\left\vert p_{0}\right\vert }\in\left[  \frac{1}{2},1\right)  ,$ then
(\ref{ne2}) implies that $p_{0}$ is positive integer and $3k-3<p_{0},$ where
$k\geq1.$

c.1) If $\varepsilon\in\left[  \frac{1}{2}+\frac{1}{2\left\vert p_{0}%
\right\vert },1\right)  $ and $\left\vert p_{0}\right\vert $ is odd, then
$k\in K_{\left\vert p_{0}\right\vert }^{\prime}=\left\{  \frac{\left\vert
p_{0}\right\vert +1}{2},...,\left\vert p_{0}\right\vert -1\right\}  $. Since
for each $p_{0}\geq3,$ we have $p_{0}\leq3k-3$ for all $k\in K_{\left\vert
p_{0}\right\vert }^{\prime}$, there is no $p_{0}$ that fulfills (\ref{ne2}) in
this case. Therefore in this case statement ii) holds.

c.2) If $\varepsilon\in\left[  \frac{1}{2},1\right)  $ and $\left\vert
p_{0}\right\vert $ is even, then $k\in K_{\left\vert p_{0}\right\vert
}^{\prime}=\left\{  \frac{\left\vert p_{0}\right\vert }{2},...,\left\vert
p_{0}\right\vert -1\right\}  .$ Since for each $p_{0}\geq6,$ we have
$p_{0}\leq3k-3$ for all $k\in K_{\left\vert p_{0}\right\vert }^{\prime}$ then
there is no even integer $p_{0}\geq6$ that satisfies (\ref{ne2}).

If $p_{0}=4$ and $\varepsilon\in\left[  \frac{1}{2},1\right)  ,$ then
$K_{\left\vert p_{0}\right\vert }^{\prime}=\left\{  2,3\right\}  $. Precisely,
if $\varepsilon\in\left[  \frac{3}{4},1\right)  ,$ then $k=3,$ while if
$\varepsilon\in\left[  \frac{1}{2},\frac{3}{4}\right)  $, then $k=2.$ It is
easy to check that if $k=2$ and $p_{0}=4$, then (\ref{ne2}) holds$,$ while for
$k=3$ and $p_{0}=4$, (\ref{ne2}) does not hold. If $p_{0}=2$ and
$\varepsilon\in\left[  \frac{1}{2},1\right)  ,$ then $K_{\left\vert
p_{0}\right\vert }^{\prime}=\left\{  1\right\}  $. It is easy to see that
(\ref{ne2}) holds for $k=1$ and $p_{0}=2.$ Therefore in this case statement
ii) holds except for $p_{0}=2$ or $4$ and $\varepsilon\in\left[  \frac{1}%
{2},\frac{1}{2}+\frac{1}{\left\vert p_{0}\right\vert }\right)  .$
\end{proof}

\begin{remark}
If we define the set $B\left(  \varepsilon\right)  $ according to the new
upper bounds of $x$ given in Remark \ref{Rem 2}, then for all $\varepsilon
\in\left[  0,1\right)  $ and $\left(  p_{0},p_{1}\right)  \in\mathbb{Z}^{2},$
$\left\vert p_{0}\right\vert \geq2$ the following holds: $\left(  \frac
{1}{p_{0}},\frac{p_{1}}{p_{0}}\right)  \in D\left(  \varepsilon\right)  $ if
and only if $\left(  \frac{1}{p_{0}},\frac{p_{1}}{p_{0}}\right)  \in B\left(
\varepsilon\right)  .$
\end{remark}

\begin{lemma}
\label{Lem5}Let $p_{0}=2$ or $p_{0}=4$ and $\varepsilon\in\left[  \frac{1}%
{2},\frac{1}{2}+\frac{1}{\left\vert p_{0}\right\vert }\right)  $. \textit{Then
}$P\left(  x\right)  =x^{2}+p_{1}x+p_{0}$\textit{\ is a }$\mathcal{\varepsilon
-}$CNS polynomial if and only if $-p_{0}+k\leq p_{1}\leq k+1$ where
$k=\left\lfloor \mathcal{\varepsilon}\left\vert p_{0}\right\vert \right\rfloor
.$
\end{lemma}

\begin{proof}
If $p_{0}=2$ or $p_{0}=4$ and $\varepsilon\in\left[  \frac{1}{2},\frac{1}%
{2}+\frac{1}{\left\vert p_{0}\right\vert }\right)  ,$ then $\varepsilon
_{k}=\frac{1}{2},\ $where $k=1$ if $p_{0}=2,\ $and $k=2$ if $p_{0}=4.$ From
Corollary \ref{Corr-1} we derive: $P$ is a $\frac{1}{2}$-CNS polynomial if and
only if
\begin{align*}
-2  &  \leq p_{1}\leq3,\text{ for }p_{0}=4\\
-1  &  \leq p_{1}\leq2,\text{ for }p_{0}=2,
\end{align*}
or equivalently, if and only if $-p_{0}+k\leq p_{1}\leq k+1$. Now, from
Corollary \ref{TRM-veza 2} it follows that $P$ is $\varepsilon$-CNS polynomial
if and only if $-p_{0}+k\leq p_{1}\leq k+1$.
\end{proof}

\begin{proof}
[Proof of Theorem \ref{TRM-main-manje}]i) Let $\varepsilon\in\left[
0,\frac{1}{2}\right)  $ or let $\varepsilon=\frac{1}{2}$ if $\left\vert
p_{0}\right\vert $ is odd. Suppose\textit{\ }$P$\textit{\ }is a\textit{\ }%
$\mathcal{\varepsilon-}$CNS polynomial. Then, by Theorem \ref{TRM-veza}, we
have $\left(  \frac{1}{p_{0}},\frac{p_{1}}{p_{0}}\right)  \in D_{2,\varepsilon
}^{0}.$ Since by (\ref{inc-0}) we have $\mathcal{D}_{2,\varepsilon}^{0}\subset
D\left(  \varepsilon\right)  ,$ then $\left(  \frac{1}{p_{0}},\frac{p_{1}%
}{p_{0}}\right)  \in D\left(  \varepsilon\right)  $. Now, by Proposition
\ref{Lem1}, we conclude that the pair $\left(  p_{0},p_{1}\right)  $ fulfills
the conditions (\ref{prva1}) or (\ref{druga 1}). Let the pair $\left(
p_{0},p_{1}\right)  $ satisfies the conditions (\ref{prva1}) or (\ref{druga 1}%
). Then Proposition \ref{Lem1} implies $\left(  \frac{1}{p_{0}},\frac{p_{1}%
}{p_{0}}\right)  \in D\left(  \varepsilon_{k}\right)  $, and by Proposition
\ref{Lem2} ii), we conclude $\left(  \frac{1}{p_{0}},\frac{p_{1}}{p_{0}%
}\right)  \in B\left(  \varepsilon_{k}\right)  .$ Since, by (\ref{inc-0}), we
have $B\left(  \varepsilon_{k}\right)  \subset\mathcal{D}_{2,\varepsilon_{k}%
}^{0}$, then $\left(  \frac{1}{p_{0}},\frac{p_{1}}{p_{0}}\right)
\in\mathcal{D}_{2,\varepsilon_{k}}^{0}$. By Theorem \ref{TRM-veza} and
Corollary \ref{Corr-1}, we now conclude $P$\textit{\ }is an
$\mathcal{\varepsilon-}$CNS polynomial.

ii) Let $\varepsilon\in\left(  \frac{1}{2},1\right)  $ or let $\varepsilon
=\frac{1}{2}$ if $\left\vert p_{0}\right\vert $ is even. Assume that
$P$\textit{\ }is an\textit{\ }$\mathcal{\varepsilon-}$CNS polynomial.
Similarly as above, by Theorem \ref{TRM-veza} and by (\ref{inc-0}), we
conclude $\left(  \frac{1}{p_{0}},\frac{p_{1}}{p_{0}}\right)  \in D\left(
\varepsilon\right)  .$ Now, by Proposition \ref{Lem3}, we obtain that the pair
$\left(  p_{0},p_{1}\right)  $ fulfills the conditions (\ref{prva2}) or
(\ref{druga2}). Let the pair $\left(  p_{0},p_{1}\right)  $ satisfies the
conditions (\ref{prva2}) or (\ref{druga2}). Then from Proposition \ref{Lem3}
and Proposition \ref{Lem2} ii) we deduce $\left(  \frac{1}{p_{0}},\frac{p_{1}%
}{p_{0}}\right)  \in B\left(  \varepsilon_{k}\right)  ,$ except when $p_{0}=2$
or $4$ and $\varepsilon\in\left[  \frac{1}{2},\frac{1}{2}+\frac{1}{\left\vert
p_{0}\right\vert }\right)  .$ Since, $B\left(  \varepsilon_{k}\right)
\subset\mathcal{D}_{2,\varepsilon_{k}}^{0}$, we conclude, similar to above,
that $P$\textit{\ }is an $\mathcal{\varepsilon-}$CNS polynomial. If $p_{0}=2$
or $4$ and $\varepsilon\in\left[  \frac{1}{2},\frac{1}{2}+\frac{1}{\left\vert
p_{0}\right\vert }\right)  $, then the statement follows directly from
Proposition \ref{Lem5}.
\end{proof}

\appendix{}

\section{Proofs of auxiliary results from Section \ref{Sect3}\label{Sect5}}

Lemmas \ref{Lemma-18-prvi} and \ref{Lemma-18-s} are needed in the proof of
Lemma \ref{Lemma-delta-CZ}.

\begin{algorithm}
\label{Lemma-18-prvi}Let $n\geq4\ $be integer and $\varepsilon\in\Big[\frac
{2}{3(2n+1)},\frac{2}{3(2n-1)}\Big).$ Let $\Delta_{19}=\Delta(W,Z,V_{2})$ be
triangle, where $W=(\varepsilon,1),$ $Z=(\frac{2}{3}-\varepsilon,\frac{5}%
{3}-2\varepsilon)$ and $V_{2}=(\frac{2}{3}-\varepsilon,\frac{5}{3}%
-2\varepsilon-\frac{2}{n}\varepsilon)$. Then $\Delta_{19}\cap\mathcal{D}%
_{2,\varepsilon}^{0}=\Delta_{19}\setminus\overline{WZ}.$
\end{algorithm}

\begin{proof}
The algorithm of Lemma \ref{lem-alg} for $\mathcal{H}=\Delta_{19}$ and
$\varepsilon\in\Big[\frac{2}{3(2n+1)},\frac{2}{3(2n-1)}\Big),$ $n\geq4\ $leads
to the graph $G_{\varepsilon}(\Delta_{19})$ which is subgraph of the graph
$G^{\prime}(\Delta_{19})$ with $18n-3$ vertices and $32n-12$ edges as follows:%
\begin{align*}
\mathcal{V}\left(  {\small \Delta}_{19}\right)   &  =\left\{  {\small \pm
(t,-t),\pm(t,-(t-1)):t=0,...,2n}\right\} \\
&  {\small \cup}\left\{  {\small \pm(t,-(t+1)),\pm(t,-(t-2)):t=1,...,2n-1}%
\right\} \\
&  {\small \cup}\left\{  {\small \pm(t,-(t+2)):t=0,...,n-2}\right\}
\end{align*}

\begin{itemize}
\item[\textbf{a)}] $(t,-t)\rightarrow(-t,t)$ for all $t=0,...,2n,$ and also
$(t,-t)\rightarrow(-t,t-1)$ if $t=2n,$ $2n-1.$

\item[\textbf{b)}] $(-t,t)\rightarrow(t,-\left(  t-1\right)  ),$ for all
$t=1,...,2n,$ and also $(-t,t)\rightarrow(t,-t)$ if $t=1.$

\item[\textbf{c)}] $(t,-(t-1))\rightarrow(-\left(  t-1\right)  ,t-1),$ for all
$t=0,...,2n,$ and also $(t,-(t-1))\rightarrow(-\left(  t-1\right)  ,t-2)$ if
$t=n-1,...,2n;$

\item[\textbf{d)}] $(-t,t-1)\rightarrow(t-1,-\left(  t-2\right)  ),$ for all
$t=0,...,2n,$ and also $(-t,t-1)\rightarrow(t-1,-\left(  t-1\right)  )$ if
$t=0$ and $(-t,t-1)\rightarrow(t-1,-\left(  t-3\right)  )$ if $t=n+1,...2n;$

\item[\textbf{e)}] $(t,-(t+1))\rightarrow(-(t+1),t+1)$ for all $t=1,...,2n-1,$
and also $(t,-(t+1))\rightarrow(-(t+1),t+2)$ if $t=1,...,2n-3.$

\item[\textbf{f)}] $(-t,t+1)\rightarrow(t+1,-\left(  t+1\right)  )$ for all
$t=1,...,2n-1,$ and also $(-t,t+1)\rightarrow(t+1,-t)\ $if $t=2,...,2n-1;$

\item[\textbf{g)}] $(t,-(t-2))\rightarrow(-(t-2),t-3)$ and
$(t,-(t-2))\rightarrow(-(t-2),t-2),$ for all $t=1,...,2n-1$ and also
$(t,-(t-2))\rightarrow(-(t-2),t-4)$ if $t=2n-1;$

\item[\textbf{h)}] $(-t,t-2)\rightarrow(t-2,-\left(  t-3\right)  )$ and
$(-t,t-2)\rightarrow(t-2,-\left(  t-4\right)  ),$ for all $t=1,...,2n-1$, and
also $(-t,t-2)\rightarrow(t-2,-\left(  t-2\right)  )$ if $t=1.$

\item[\textbf{i)}] $(t,-(t+2))\rightarrow(-(t+2),t+2)$ and
$(t,-(t+2))\rightarrow(-(t+2),t+3)$ for all $t=0,...,n-2,$ and also
$(t,-(t+2))\rightarrow(-(t+2),t+4)$ if $t=0,...,n-4.$

\item[\textbf{j)}] $(-t,t+2)\rightarrow(t+2,-(t+2))$ for all $t=0,...,n-2$,
and also $(-t,t+2)\rightarrow(t+2,-(t+3))$ if $k=0,...,n-4$ and
$(-t,t+2)\rightarrow(t+2,-(t+1))$ if $t=2,...,n-2.$
\end{itemize}

By the definition of the edges of the graph $G^{\prime}(\Delta_{19})$ we have:
$(0,0)\overset{a)}{\rightarrow}(0,0)\ $and $(-1,1)\overset{b)}{\rightarrow
}(1,-1)\overset{a)}{\rightarrow}(-1,1)$ which implies that the graph
$G^{\prime}\left(  \Delta_{19}\right)  $ has two cycles, a trivial cycle
$\pi_{0}=(0)$ and a non-zero cycle $\pi_{2}=(-1,1)$. Our goal is to show that
$\pi_{0}$ and $\pi_{1}$ are the only cycles of the graph $G^{\prime}%
(\Delta_{19}).$ We will show the following:

\begin{claim}
\label{clm-osn}Let $v\in V\left(  {\small \Delta}_{19}\right)  $ be arbitrary
vertex. If there exist directed trail of the graph $G^{\prime}(\Delta_{19})$
that starts at vertex $v$ and includes some cycle $\pi$, then $\pi=\pi_{0}%
\ $or $\pi=\pi_{2}.$ Furthermore, each maximal directed trail of the graph
$G^{\prime}\left(  \Delta_{19}\right)  $ which starts at vertex $v,$ includes
exactly one of the cycles $\pi_{0}\ $or $\pi_{2}.$
\end{claim}

Obviously, the Claim \ref{clm-osn} implies that $\pi_{0}$ and $\pi_{1}$ are
the only cycles of the graph $G^{\prime}(\Delta_{19}).$

For all integers $k\geq0$ let us define sets%
\begin{align*}
\mathcal{V}_{k}  &  =\left\{  \pm(t,-t),\pm(t,-(t-1)):t=0,...,\min\left\{
k,2n\right\}  \right\} \\
&  {\small \cup}\left\{  \pm(t,-(t+1)),\pm(t,-(t-2)):t=\min\left\{
k,1\right\}  ,...,\min\left\{  k,2n-1\right\}  \right\} \\
&  {\small \cup}\left\{  \pm(t,-(t+2)):t=0,...,\min\left\{  k,n-2\right\}
\right\}  .
\end{align*}
Note that $\mathcal{V}_{k}\subset\mathcal{V}_{k+1}$ for all $k=0,...2n-1$ and
$\mathcal{V}_{k}=\mathcal{V}\left(  {\small \Delta}_{19}\right)  \ $if
$k\geq2n\geq8.$ We will prove the following:

\begin{claim}
\label{clm}For every integer $k\geq0$ and arbitrary vertex $v\in
\mathcal{V}_{k}$ the following holds:

\begin{itemize}
\item[i.] Every maximal directed trail of the graph $G^{\prime}(\Delta_{19})$
that starts at $v$ includes at least one of the vertices: $(0,0),$ $(1,-1),$
$(-1,1).$

\item[ii.] Every possible directed trail of the graph $G^{\prime}(\Delta
_{19})$ that starts at vertex $v$ which does not include vertex $(0,0)$ and
includes at most one of the vertices $(1,-1)\ $and $(-1,1)$ is a directed path.
\end{itemize}
\end{claim}

Note that Claim \ref{clm} implies Claim \ref{clm-osn}. Namely, if maximal
directed trail of the graph $G^{\prime}(\Delta_{19})$ that starts at some a
$v\in\mathcal{V}\left(  {\small \Delta}_{19}\right)  $ includes one of the
vertices: $(0,0),(-1,1),$ $(1,-1)$, then that trail includes one of the
subtrails: $(0,0)\overset{a)}{\rightarrow}\mathbf{(}0,0)$, $\mathbf{(}%
-1,1\mathbf{)}\overset{b)}{\rightarrow}\mathbf{(}1,-1\mathbf{)}%
\overset{a)}{\rightarrow}(-1,1),$ $(-1,1)\overset{b)}{\rightarrow
}(1,0)\overset{c)}{\rightarrow}(0,0)\overset{a)}{\rightarrow}\mathbf{(}%
0,0),\ \mathbf{(}1,-1\mathbf{)}\overset{a)}{\rightarrow}%
(-1,1)\overset{b)}{\rightarrow}\mathbf{(}1,-1\mathbf{)}$ and
$(1,-1)\overset{a)}{\rightarrow}(-1,1)\overset{b)}{\rightarrow}%
(1,0)\overset{c)}{\rightarrow}(0,0)\overset{a)}{\rightarrow}\mathbf{(}0,0).$
Since each of these subtrails includes exactly one of the cycles $\pi_{0}\ $or
$\pi_{2},$ to prove that Claim \ref{clm-osn} it is sufficient to show that
Claim \ref{clm} holds for $k=2n$, i.e for all vertices $v\in\mathcal{V}%
_{2n}=\mathcal{V}\left(  {\small \Delta}_{19}\right)  .$ We will prove Claim
\ref{clm} by induction on $k$.

\textit{Base case: }As $n\geq4$, then $2n-1\geq7$ and $n-2\geq2,\ $and
therefore%
\begin{align*}
\mathcal{V}_{0}  &  =\left\{  (0,0),\pm(0,1),\pm(0,-2)\right\}  ,\\
\mathcal{V}_{1}  &  =\left\{  (0,0),\pm(1,-1),\pm(0,1),\pm(1,0),\pm
(1,-2),\pm(1,1),\pm(0,-2),\pm(1,-3)\right\}  ,
\end{align*}
for all $n\geq4.$ Since $\mathcal{V}_{0}\subset\mathcal{V}_{1}$ we prove the
Claim \ref{clm} for each element of set $\mathcal{V}_{1}$ (more than we need).
By the definition of the edges of the graph $G^{\prime}(\Delta_{19})$ given by
\textbf{a)},..., \textbf{j)} we obtain that all maximal directed trails that
starts at each vertex from the set $\mathcal{V}_{1}\ $are given by:

\begin{itemize}
\item $\mathbf{(0,0)}$ $\overset{a)}{\rightarrow}\mathbf{(0,0)}$

\item $\mathbf{(-1,1)}\overset{b)}{\rightarrow}\mathbf{(1,-1)}$\newline%
$\mathbf{(-1,1)}\overset{b)}{\rightarrow}(1,0)\overset{c)}{\rightarrow
}\mathbf{(0,0)}$

\item $\mathbf{(1,-1)}\overset{a)}{\rightarrow}\mathbf{(-1,1)}$

\item $\mathbf{(0,1)}\overset{c)}{\rightarrow}\mathbf{(1,-1)}$

\item $\mathbf{(0,-1)}\overset{d)}{\rightarrow}\mathbf{(-1,1)}$\newline%
$\mathbf{(0,-1)}\overset{d)}{\rightarrow}(-1,2)\overset{f)}{\rightarrow
}(2,-2)\overset{a)}{\rightarrow}(-2,2)\overset{b)}{\rightarrow}%
(2,-1)\overset{c)}{\rightarrow}(\mathbf{-1,1)}$

\item $\mathbf{(1,0)}\overset{c)}{\rightarrow}\mathbf{(0,0)}$

\item $\mathbf{(-1,0)}\overset{d)}{\rightarrow}(0,1)\overset{c)}{\rightarrow
}\mathbf{(1,-1)}$

\item $\mathbf{(-1,2)}\overset{f)}{\rightarrow}(2,-2)\overset{a)}{\rightarrow
}(-2,2)\overset{b)}{\rightarrow}(2,-1)\overset{c)}{\rightarrow}\mathbf{(-1,1)}%
$

\item $\mathbf{(1,-2)}\overset{e)}{\rightarrow}(-2,2)\overset{b)}{\rightarrow
}(2,-1)\overset{c)}{\rightarrow}\mathbf{(-1,1)}$\newline$\mathbf{(1,-2)}%
\overset{e)}{\rightarrow}(-2,3)\overset{f)}{\rightarrow}%
(3,-3)\overset{a)}{\rightarrow}(-3,3)\overset{b)}{\rightarrow}%
(3,-2)\overset{c)}{\rightarrow}$ $...$\newline$\mathbf{(1,-2)}%
\overset{e)}{\rightarrow}(-2,3)\overset{f)}{\rightarrow}%
(3,-2)\overset{c)}{\rightarrow}$ $...$

\item $\mathbf{(1,1)}\overset{g)}{\rightarrow}\mathbf{(1,-1)}$\newline%
$\mathbf{(1,1)}\overset{g)}{\rightarrow}(1,-2)\overset{e)}{\rightarrow
}(-2,2)\overset{b)}{\rightarrow}(2,-1)\overset{c)}{\rightarrow}\mathbf{(-1,1)}%
$\newline$\mathbf{(1,1)}\overset{g)}{\rightarrow}%
(1,-2)\overset{e)}{\rightarrow}(-2,3)\overset{f)}{\rightarrow}%
(3,-3)\overset{a)}{\rightarrow}(-3,3)\overset{b)}{\rightarrow}%
(3,-2)\overset{c)}{\rightarrow}...$\newline$\mathbf{(1,1)}%
\overset{g)}{\rightarrow}(1,-2)\overset{e)}{\rightarrow}%
(-2,3)\overset{f)}{\rightarrow}(3,-2)\overset{c)}{\rightarrow}$ $...$

\item $\mathbf{(-1,-1)}\overset{h)}{\rightarrow}\mathbf{(-1,1)}$%
\newline$\mathbf{(-1,-1)}\overset{h)}{\rightarrow}%
(-1,2)\overset{f)}{\rightarrow}(2,-2)\overset{c)}{\rightarrow}%
(-2,2)\overset{b)}{\rightarrow}(2,-1)\overset{c)}{\rightarrow}\mathbf{(-1,1)}%
$\newline$\mathbf{(-1,-1)}\overset{h)}{\rightarrow}%
(-1,3)\overset{j)}{\rightarrow}(3,-3)\overset{a)}{\rightarrow}%
(-3,3)\overset{b)}{\rightarrow}(3,-2)\overset{c)}{\rightarrow}$ $...$%
\newline$\mathbf{(-1,-1)}\overset{h)}{\rightarrow}(-1,3)\underset{n\geq
5}{\overset{j)}{\rightarrow}}(3,-4)\overset{e)}{\rightarrow}...,$ \ if
$n\geq5$

\item $\mathbf{(0,2)}\overset{j)}{\rightarrow}(2,-2)\overset{a)}{\rightarrow
}(-2,2)\overset{b)}{\rightarrow}(2,-1)\overset{c)}{\rightarrow}\mathbf{(-1,1)}%
$\newline$\mathbf{(0,2)}\overset{j)}{\rightarrow}%
(2,-3)\overset{e)}{\rightarrow}(-3,3)\overset{b)}{\rightarrow}%
(3,-2)\overset{c)}{\rightarrow}$ $...$\newline$\mathbf{(0,2)}%
\overset{j)}{\rightarrow}(2,-3)\overset{e)}{\rightarrow}%
(-3,4)\overset{f)}{\rightarrow}(4,-4)\overset{a)}{\rightarrow}\left(
-4,4\right)  \overset{b)}{\rightarrow}(4,-3)$\newline$\mathbf{(0,2)}%
\overset{j)}{\rightarrow}(2,-3)\overset{e)}{\rightarrow}%
(-3,4)\overset{f)}{\rightarrow}(4,-3)$

\item $\mathbf{(0,-2)}\overset{i)}{\rightarrow}(-2,2)\overset{b)}{\rightarrow
}\left(  2,-1\right)  \overset{c)}{\rightarrow}\mathbf{(-1,1)}$\newline%
$\mathbf{(0,-2)}\overset{i)}{\rightarrow}(-2,3)\overset{f)}{\rightarrow
}\left(  3,-3\right)  \overset{a)}{\rightarrow}(-3,3)\overset{b)}{\rightarrow
}(3,-2)\overset{c)}{\rightarrow}$ $...$\newline$\mathbf{(0,-2)}%
\overset{i)}{\rightarrow}(-2,3)\overset{f)}{\rightarrow}\left(  3,-2\right)
\overset{c)}{\rightarrow}$ $...$\newline$\mathbf{(0,-2)}%
\overset{i)}{\rightarrow}(-2,4)\overset{j)}{\rightarrow}\left(  4,-4\right)
\overset{a)}{\rightarrow}\left(  -4,4\right)  \overset{b)}{\rightarrow
}(4,-3)\overset{c)}{\rightarrow}...$\newline$\mathbf{(0,-2)}%
\overset{i)}{\rightarrow}(-2,4)\underset{n\geq6}{\overset{j)}{\rightarrow}%
}\left(  4,-5\right)  \overset{e)}{\rightarrow}\left(  -5,5\right)
\overset{b)}{\rightarrow}\left(  5,-4\right)  \overset{c)}{\rightarrow}...$,
if $n\geq6$\newline$\mathbf{(0,-2)}\overset{i)}{\rightarrow}{\small (-2,4)}%
\underset{n\geq6}{\overset{j)}{\rightarrow}}\left(  {\small 4,-5}\right)
\overset{e)}{\rightarrow}\left(  {\small -5,6}\right)
\overset{f)}{\rightarrow}\left(  {\small 6,-6}\right)
\overset{a)}{\rightarrow}\left(  {\small -6,6}\right)
\overset{b)}{\rightarrow}\left(  {\small 6,-5}\right)
\overset{c)}{\rightarrow}...$ if ${\small n\geq6}$\newline$\mathbf{(0,-2)}%
\overset{i)}{\rightarrow}(-2,4)\underset{n\geq6}{\overset{j)}{\rightarrow}%
}\left(  4,-5\right)  \overset{e)}{\rightarrow}\left(  -5,6\right)
\overset{f)}{\rightarrow}\left(  6,-5\right)  \overset{c)}{\rightarrow}...$,
\ if $n\geq6$\newline$\mathbf{(0,-2)}\overset{i)}{\rightarrow}%
(-2,4)\overset{j)}{\rightarrow}(4,-3)\overset{c)}{\rightarrow}...$

\item $\mathbf{(1,-3)}\overset{i)}{\rightarrow}(-3,3)\overset{b)}{\rightarrow
}(3,-2)\overset{c)}{\rightarrow}...$\newline$\mathbf{(1,-3)}%
\overset{i)}{\rightarrow}(-3,4)\overset{f)}{\rightarrow}%
(4,-4)\overset{a)}{\rightarrow}\left(  -4,4\right)  \overset{b)}{\rightarrow
}(4,-3)\overset{c)}{\rightarrow}...$\newline$\mathbf{(1,-3)}%
\overset{i)}{\rightarrow}(-3,4)\overset{f)}{\rightarrow}%
(4,-3)\overset{c)}{\rightarrow}...\newline\mathbf{(1,-3)}\underset{n\geq
5}{\overset{i)}{\rightarrow}}(-3,5)\overset{j)}{\rightarrow}....$, \ if
$n\geq5$

\item $\mathbf{(-1,3)}\overset{j)}{\rightarrow}(3,-3)\overset{a)}{\rightarrow
}(-3,3)\overset{b)}{\rightarrow}(3,-2)\overset{c)}{\rightarrow}...$%
\newline$\mathbf{(-1,3)}\overset{j)}{\rightarrow}%
(3,-4)\overset{e)}{\rightarrow}...$
\end{itemize}

To complete the proof of the base case, we need to find all maximal directed
trails that starts at the vertices $(3,-2),(3,-4),$ $(4,-3)$ for all $n\geq4,$
at $\left(  5,-4\right)  $ $\left(  6,-5\right)  $ if $n\geq6$ and at
$(-3,5)\ $if $n\geq5.$ We have:

\begin{itemize}
\item $(3,-2)\overset{c)}{\rightarrow}(-2,2)\overset{b)}{\rightarrow
}(2,-1)\overset{c)}{\rightarrow}\mathbf{(-1,1)}$\newline%
$(3,-2)\underset{n=4}{\overset{c)}{\rightarrow}}(-2,1)\overset{d)}{\rightarrow
}(1,0)\overset{c)}{\rightarrow}\mathbf{(0,0)}$ \ if $n=4$

\item $(4,-3)\overset{c)}{\rightarrow}(-3,3)\overset{b)}{\rightarrow
}\mathbf{(3,-2)}\overset{c)}{\rightarrow}...\ \newline%
(4,-3)\underset{n=4,5}{\overset{c)}{\rightarrow}}%
(-3,2)\overset{d)}{\rightarrow}(2,-1)\overset{c)}{\rightarrow}\mathbf{(-1,1),}%
$ \ if $n=4,5$

\item $(5,-4)\overset{c)}{\rightarrow}\left(  -4,4\right)
\overset{b)}{\rightarrow}\mathbf{(4,-3)}\overset{c)}{\rightarrow}...\newline$
$(5,-4)\underset{n=4,5,6}{\overset{c)}{\rightarrow}}%
(-4,3)\overset{d)}{\rightarrow}\mathbf{(3,-2)}\overset{c)}{\rightarrow}...,$
\ if $n=4,5,6$

\item $(3,-4)\overset{e)}{\rightarrow}(-4,4)\overset{b)}{\rightarrow
}\mathbf{(4,-3)}\overset{c)}{\rightarrow}...\newline%
(3,-4)\overset{e)}{\rightarrow}(-4,5)\overset{f)}{\rightarrow}%
(5,-5)\overset{a)}{\rightarrow}(-5,5)\overset{b)}{\rightarrow}\mathbf{(5,-4)}%
\overset{c)}{\rightarrow}...\newline(3,-4)\overset{e)}{\rightarrow
}(-4,5)\overset{f)}{\rightarrow}\mathbf{(5,-4)}\overset{c)}{\rightarrow}...$

\item $(-3,5)\underset{n\geq5}{\overset{j)}{\rightarrow}}%
(5,-5)\overset{b)}{\rightarrow}\mathbf{(5,-4)}\overset{c)}{\rightarrow}...$if
$n\geq5\newline(-3,5)\underset{n\geq5}{\overset{j)}{\rightarrow}%
}\mathbf{(5,-4)}\overset{c)}{\rightarrow}...$if $n\geq5\newline%
(-3,5)\underset{n\geq7}{\overset{j)}{\rightarrow}}%
(5,-6)\overset{e)}{\rightarrow}....$if $n\geq7.$

\item $\left(  6,-5\right)  \overset{c)}{\rightarrow}\left(  -5,5\right)
\overset{b)}{\rightarrow}\left(  \mathbf{5,-4}\right)
\overset{c)}{\rightarrow}...\newline\left(  6,-5\right)
\underset{n=4,5,6,7}{\overset{c)}{\rightarrow}}\left(  -5,4\right)
\overset{d)}{\rightarrow}\left(  \mathbf{4,-3}\right)
\overset{c)}{\rightarrow}...$if $n=4,5,6,7\newline\left(  6,-5\right)
\underset{n=4,5,6,7}{\overset{c)}{\rightarrow}}\left(  -5,4\right)
\underset{n=4}{\overset{d)}{\rightarrow}}\left(  4,-2\right)
\overset{g)}{\rightarrow}...$if $n=4.$
\end{itemize}

Therefore, it remains to find all maximal directed trails that starts
at$\ $the vertices $\left(  6,-5\right)  $ for all $n\geq4$ and $\left(
-3,5\right)  $ if $n\geq5,$ or more precisely for the vertices $\left(
4,-2\right)  $ if $n=4$ and $(5,-6)$ if $n\geq7,$ respectively. We have:

\begin{itemize}
\item $\left(  4,-2\right)  \overset{g)}{\rightarrow}\left(  -2,2\right)
\overset{b)}{\rightarrow}\left(  2,-1\right)  \overset{c)}{\rightarrow
}\mathbf{(-1,1)}\newline\left(  4,-2\right)  \overset{g)}{\rightarrow}\left(
-2,1\right)  \overset{d)}{\rightarrow}(1,0)\overset{c)}{\rightarrow
}\mathbf{(0,0).}$

and

\item $(5,-6)\overset{e)}{\rightarrow}\left(  -6,6\right)
\overset{b)}{\rightarrow}\mathbf{(6,-5)}\overset{c)}{\rightarrow}%
...\newline(5,-6)\overset{e)}{\rightarrow}\left(  -6,7\right)
\overset{f)}{\rightarrow}(7,-7)\overset{b)}{\rightarrow}...\newline%
(5,-6)\overset{e)}{\rightarrow}\left(  -6,7\right)  \overset{f)}{\rightarrow
}(7,-6)\overset{c)}{\rightarrow}...$
\end{itemize}

Since

\begin{itemize}
\item $(7,-6)\overset{c)}{\rightarrow}\left(  -6,6\right)
\overset{b)}{\rightarrow}\mathbf{(6,-5)}\overset{c)}{\rightarrow}%
...\newline(7,-6)\underset{n=4,5,6,7,8}{\overset{c)}{\rightarrow}}\left(
-6,5\right)  \overset{d)}{\rightarrow}\left(  \mathbf{5,-4}\right)
\overset{c)}{\rightarrow}...$if $n=4,5,6,7,8$

\item $(7,-7)\overset{b)}{\rightarrow}\left(  -7,6\right)
\overset{d)}{\rightarrow}\left(  \mathbf{6,-5}\right)
\overset{c)}{\rightarrow}...\newline(7,-7)\overset{b)}{\rightarrow}\left(
-7,6\right)  \underset{n=4,5,6}{\overset{d)}{\rightarrow}}\left(  6,-4\right)
\overset{g)}{\rightarrow}...$if $n=4,5,6$
\end{itemize}

and

\begin{itemize}
\item $\left(  6,-4\right)  \overset{g)}{\rightarrow}\left(  -4,4\right)
\overset{b)}{\rightarrow}\mathbf{(4,-3)}\overset{c)}{\rightarrow}%
...\newline\left(  6,-4\right)  \overset{g)}{\rightarrow}\left(  -4,3\right)
\overset{d)}{\rightarrow}\mathbf{(3,-2)}\overset{c)}{\rightarrow}...$
\end{itemize}

we find all maximal directed trails that starts at$\ $each element of set
$\mathcal{V}_{1}$. It is obvious that every maximal directed trail of the
graph $G^{\prime}(\Delta_{19})$ which starts at each vertex $v\in
\mathcal{V}_{1}$ includes one of the following vertices: $(0,0),$ $(1,-1),$
$(-1,1).$ It is also easy to see that any possible directed trail of the graph
$G^{\prime}\left(  \Delta_{19}\right)  $ that starts at any $v\in
\mathcal{V}_{1}$ which does not include vertex $(0,0)$ and includes at most
one of the vertices $(1,-1)\ $and $(-1,1)$ is a directed path. Note that if
$v=\left(  -1,1\right)  ,$ there is exactly one trail $(-1,1)\rightarrow(1,0)$
with such property which is obviously a directed path. But if $v=\left(
1,-1\right)  ,\ \left(  0,0\right)  $ or $\left(  1,0\right)  $ there is no
directed trail with this property. Therefore, Claim \ref{clm} also valid for
the vertices $\left(  0,0\right)  ,$ $(-1,1),\ (1,-1)\ $and $\left(
1,0\right)  $. This concludes the proof of the base case.

\textit{Induction step: }Suppose, for some $k\geq0$ that Claim \ref{clm} holds
for every element of the set $\mathcal{V}_{k}.$ We have to prove that Claim
\ref{clm} holds for every element of the set $\mathcal{V}_{k+1}.$ If
$k\geq2n,$ then $\mathcal{V}_{k}=\mathcal{V}_{k+1},$ and consequently Claim
\ref{clm} holds for every element of the set $\mathcal{V}_{k+1}.$ If $0\leq
k\leq2n-1,$ let us define sets%
\begin{align}
\mathcal{V}_{k+1}^{\left(  1\right)  }  &  =\left\{  \pm(k+1,-\left(
k+1\right)  ){\small ,}\pm(k+1,-k)\right\}  \text{, if }0\leq k\leq
2n-1\label{v1}\\
\mathcal{V}_{k+1}^{\left(  2\right)  }  &  =\left\{  {\small \pm
(k+1,-(k+2)}){\small ,}\pm({\small k+1,-}\left(  {\small k-1}\right)
)\right\}  \text{, if }{\small 0\leq k\leq2n-2}\label{v2}\\
\mathcal{V}_{k+1}^{\left(  3\right)  }  &  =\left\{  \pm(k+1,-(k+3)\right\}
,\text{ if }0\leq k\leq n-3. \label{v3}%
\end{align}
Then%
\begin{align}
\mathcal{V}_{k+1}  &  =\mathcal{V}_{k}\cup\mathcal{V}_{k+1}^{\left(  1\right)
}\cup\mathcal{V}_{k+1}^{\left(  2\right)  }\cup\mathcal{V}_{k+1}^{\left(
3\right)  }\text{ \ if }0\leq k\leq n-3\label{v11}\\
\mathcal{V}_{k+1}  &  =\mathcal{V}_{k}\cup\mathcal{V}_{k+1}^{\left(  1\right)
}\cup\mathcal{V}_{k+1}^{\left(  2\right)  }\text{ \ if }n-2\leq k\leq
2n-2\label{v22}\\
\mathcal{V}_{k+1}  &  =\mathcal{V}_{k}\cup\mathcal{V}_{k+1}^{\left(  1\right)
}\text{ \ if }k=2n-1 \label{v33}%
\end{align}
Since, by the induction hypothesis, Claim \ref{clm} holds for every element of
the set $\mathcal{V}_{k},$ according to (\ref{v11}), (\ref{v22}) and
(\ref{v33}), we have to prove that Claim \ref{clm} holds for each element of
the sets $\mathcal{V}_{k+1}^{\left(  1\right)  },\mathcal{V}_{k+1}^{\left(
2\right)  }\ $and $\mathcal{V}_{k+1}^{\left(  3\right)  }$ if $k$ lies in the
ranges given in (\ref{v1}), (\ref{v2}) and (\ref{v3}) respectively. Note that
it suffices to show that for every possible maximal directed trail of the
graph $G^{\prime}(\Delta_{19})$ that starts at each vertex $v\ $belonging to
the sets $\mathcal{V}_{k+1}^{\left(  1\right)  },\mathcal{V}_{k+1}^{\left(
2\right)  }\ $and $\mathcal{V}_{k+1}^{\left(  3\right)  }$ there exist a
vertex $w\in\mathcal{V}\left(  {\small \Delta}_{19}\right)  $ that is included
in that trail for which Claim \ref{clm} holds and such that the corresponding
directed subtrail $\left(  v,w\right)  $ is a directed path. Then we can
conclude that Claim \ref{clm} also holds for the vertex $v$. Using the
definition of the edges of the graph $G^{\prime}(\Delta_{19})$, we deduce:

\noindent\textbf{Case 1: }$\mathbf{\pm(k+1,-}\left(  \mathbf{k+1}\right)
\mathbf{)}\mathbf{\in}\mathcal{V}_{k+1}^{\left(  1\right)  }\ $if
$k\in\left\{  0,...,2n-1\right\}  .$

If $k=0$, then $\pm(k+1,-\left(  k+1\right)  )=\pm(1,-1)\in\mathcal{V}_{1}$
and by the base case, Claim \ref{clm} holds.

If $1\leq k\leq2n-1,$ then $2\leq k+1\leq2n,$ and we have:

\begin{itemize}
\item $(k+1,-\left(  k+1\right)  ),$ $k+1=2,...,2n,\newline(k+1,-\left(
k+1\right)  )\overset{a)}{\rightarrow}(-\left(  k+1\right)  ,k+1)$ \ if
\ $k+1=2,...,2n$,$\newline(k+1,-\left(  k+1\right)  )\overset{a)}{\rightarrow
}(-\left(  k+1\right)  ,k)$ \ if \ $k+1=2n$, $2n-1,$

\item $(-\left(  k+1\right)  ,k+1),$ $k+1=2,...,2n,\newline(-\left(
k+1\right)  ,k+1)\overset{b)}{\rightarrow}(k+1,-k)$ \ if \ $k+1=2,...,2n.$
\end{itemize}

Since

\begin{itemize}
\item $(-\left(  k+1\right)  ,k),$ $k+1=2n,2n-1,\newline(-\left(  k+1\right)
,k)\overset{d)}{\rightarrow}\mathbf{(k,-}\left(  \mathbf{k-1}\right)
\mathbf{)}\ $if$\ k+1=2n,2n-1\Leftrightarrow k=2n-1,2n-2,\newline(-\left(
k+1\right)  ,k)\overset{d)}{\rightarrow}(\mathbf{k,-}\left(  \mathbf{k-2}%
\right)  )$ \ if \ $k+1=2n$, $2n-1\Leftrightarrow k=2n-1,2n-2,$

\item $(k+1,-k),$ $k+1=2,...,2n,\newline(k+1,-k)\overset{c)}{\rightarrow
}\mathbf{(-k,k)},$ \ if \ $k+1=2,...,2n\Leftrightarrow k=1,...,2n-1,\newline%
(k+1,-k)\overset{c)}{\rightarrow}\mathbf{(-k,k-1)},$\ if
\ $k+1=n-1,...,2n\Leftrightarrow k=n-2,...,2n-1,$
\end{itemize}

then every maximal directed trail that starts at each of the two vertices:
$v=\pm(k+1,-\left(  k+1\right)  )$ if $1\leq k\leq2n-1,$ includes one of the
vertices $w=\left(  -k,k\right)  ,$ $\left(  -k,k-1\right)  ,$ $(k,-\left(
k-1\right)  ),$ $(k,-\left(  k-2\right)  )\in\mathcal{V}_{k}$ and all
corresponding directed subtrails $\left(  v,w\right)  $ are the directed
paths. By the inductive hypothesis, we conclude that Claim \ref{clm} holds for
the vertices $\pm(k+1,-\left(  k+1\right)  )\mathbf{\in}\mathcal{V}%
_{k+1}^{\left(  1\right)  }$ too.

\noindent\textbf{Case 2: }$\pm\mathbf{(k+1,-k)\in}\mathcal{V}_{k+1}^{\left(
1\right)  }\ $if $k\in\left\{  0,...,2n-1\right\}  .$ $\newline$If $k=0$, then
$\pm(k+1,-k)=\pm(1,0)\in\mathcal{V}_{1}$ and by the base case, Claim \ref{clm} holds.

If $1\leq k\leq2n-1,$ then $2\leq k+1\leq2n\ $and we have:

\begin{itemize}
\item $(k+1,-k),$ $k+1=2,...,2n,\newline(k+1,-k)\overset{c)}{\rightarrow
}\mathbf{(-k,k)},$ \ if \ $k+1=2,...,2n\Leftrightarrow k=1,...,2n-1,\newline%
(k+1,-k)\overset{c)}{\rightarrow}\mathbf{(-k,k-1)},$ \ if
$k+1=n-1,...,2n\Leftrightarrow k=n-2,...,2n-1,$

\item $(-\left(  k+1\right)  ,k),$ $k+1=1,...,2n,\newline(-\left(  k+1\right)
,k)\overset{d)}{\rightarrow}\mathbf{(k,-}\left(  \mathbf{k-1}\right)
\mathbf{)},$ \ if \ $k+1=2,...,2n\Leftrightarrow k=1,...,2n-1,\newline%
(-\left(  k+1\right)  ,k)\overset{d)}{\rightarrow}\mathbf{(k,-}\left(
\mathbf{k-2}\right)  \mathbf{)}$ \ if $k+1=n+1,...,2n\Leftrightarrow
k=n,...,2n-1.$
\end{itemize}

Therefore, every maximal directed trail that starts at each of the two
vertices: $v=\pm(k+1,-k),$ if $1\leq k\leq2n-1$ includes the edge
$v\rightarrow w,$ where $w$ is one of the vertices
$(-k,k),(-k,k-1),(k,-\left(  k-1\right)  ),(k,-\left(  k-2\right)
)\in\mathcal{V}_{k}$ and, by the inductive hypothesis, we conclude that Claim
\ref{clm} holds for $\pm(k+1,-k)\mathbf{\in}\mathcal{V}_{k+1}^{\left(
1\right)  }$.

\noindent\textbf{Case 3: }$\mathbf{\pm(k+1,-}\left(  \mathbf{k-1}\right)
\mathbf{)\in}\mathcal{V}_{k+1}^{\left(  2\right)  }\ $if $k\in\left\{
0,...,2n-2\right\}  .$

If $k=0$, then $\pm(k+1,-\left(  k-1\right)  )=\pm(1,1)\in\mathcal{V}_{1}$ and
by the base case, Claim \ref{clm} holds. If $1\leq k\leq2n-2,$ then $2\leq
k+1\leq2n-1,$ and we have:

\begin{itemize}
\item $(k+1,-(k-1)),$ $k+1=2,...,2n-1\newline%
(k+1,-(k-1))\overset{g)}{\rightarrow}\mathbf{(-(k-1),k-1)}$ \ if
\ $k+1=2,...,2n-1\Leftrightarrow k-1=0,...,2n-3,\newline%
(k+1,-(k-1))\overset{g)}{\rightarrow}(\mathbf{-(k-1),k-2)}$ \ if
\ $k+1=2,...,2n-1\Leftrightarrow k-1=0,...,2n-3,\newline%
(k+1,-(k-1))\overset{g)}{\rightarrow}\mathbf{(-(k-1),k-3)}$ \ if
\ $k+1=2n-1\Leftrightarrow k-1=2n-3$

\item $(-\left(  k+1\right)  ,k-1),k+1=2,...,2n-1\newline(-\left(  k+1\right)
,k-1)\overset{h)}{\rightarrow}\mathbf{(k-1,-}\left(  \mathbf{k-2}\right)
\mathbf{)}$ \ if \ $k+1=2,...,2n-1\Leftrightarrow k-1=0,...,2n-3\newline%
(-\left(  k+1\right)  ,k-1)\overset{h)}{\rightarrow}\mathbf{(k-1,-}\left(
\mathbf{k-3}\right)  \mathbf{)}$ \ if \ $k+1=2,...,2n-1\Leftrightarrow
k-1=0,...,2n-3.$
\end{itemize}

Since every maximal directed trail that starts at each of the two vertices:
$v=\pm(k+1,-\left(  k-1\right)  \mathbf{)},$ if $1\leq k\leq2n-2$ includes the
edge $v\rightarrow w,$ where $w$ is one of the vertices $(-(k-1),k-1),$
$(-(k-1),k-2),$ $(-(k-1),k-3),$ $(k-1,-\left(  k-2\right)  ),$ $(k-1,-\left(
k-3\right)  )\in\mathcal{V}_{k-1}\subset\mathcal{V}_{k}$ by the inductive
hypothesis, we conclude that Claim \ref{clm} holds for $\pm
(k+1,-(k-1))\mathbf{\in}\mathcal{V}_{k+1}^{\left(  2\right)  }.$

\noindent\textbf{Case 4: }$\pm\mathbf{(k+1,-(k+2))\in}\mathcal{V}%
_{k+1}^{\left(  2\right)  }$ if $k\in\left\{  0,...,2n-2\right\}  .$

If $k=0$, then $\pm\mathbf{(}k+1,-(k+2)\mathbf{)}=\pm(1,-2)\in\mathcal{V}_{1}$
and by the base case, Claim \ref{clm} holds.

If $1\leq k\leq2n-2,$ then $2\leq k+1\leq2n-1,$ and we have:

\begin{itemize}
\item $(k+1,-(k+2)),$ $k+1=2,...,2n-1\newline%
(k+1,-(k+2))\overset{e)}{\rightarrow}(-(k+2),k+2)$ \ if
\ $k+1=2,...,2n-1\Leftrightarrow k+2=3,...,2n,\newline%
(k+1,-(k+2))\overset{e)}{\rightarrow}(-(k+2),k+3)$ \ if
\ $k+1=2,...,2n-3\Leftrightarrow k+2=3,...,2n-2,$

\item $(-\left(  k+1\right)  ,k+2),k+1=2,...,2n-1\newline(-\left(  k+1\right)
,k+2)\overset{f)}{\rightarrow}(k+2,-(k+2))$ \ if
\ $k+1=2,...,2n-1\Leftrightarrow k+2=3,...,2n,\newline(-\left(  k+1\right)
,k+2)\overset{f)}{\rightarrow}(k+2,-(k+1)),$ \ if
\ $k+1=2,...,2n-1\Leftrightarrow k+2=3,...,2n.$
\end{itemize}

Since

\begin{itemize}
\item $\mathbf{(}-(k+2),k+2\mathbf{),}$ $k+2=3,...,2n\mathbf{,}\newline%
(-(k+2),k+2)\overset{b)}{\rightarrow}\mathbf{(}k+2,-\left(  k+1\right)  )$
\ if \ $k+2=3,...,2n$

\item $(k+2,-(k+2),$ $k+2=3,...,2n,\newline%
(k+2,-(k+2))\overset{a)}{\rightarrow}(-\left(  k+2\right)  ,k+2)$\ if
\ $k+2=3,...,2n,\newline(k+2,-(k+2))\overset{a)}{\rightarrow}\mathbf{(}%
-\left(  k+2\right)  ,k+1)$\ if \ \ $k+2=2n,2n-1$

\item $\mathbf{(}k+2,-(k+1)\mathbf{)},$ $k+2=3,...,2n\newline%
(k+2,-(k+1))\overset{c)}{\rightarrow}\mathbf{(-}\left(  \mathbf{k+1}\right)
\mathbf{,k+1)}$\ if \ $k+2=3,...,2n\Leftrightarrow k=1,...,2n-2\newline%
(k+2,-(k+1))\overset{c)}{\rightarrow}\mathbf{(-}\left(  \mathbf{k+1}\right)
\mathbf{,k)}$\ if \ $k+2=n-1,...,2n\Leftrightarrow k=n-3,...,2n-2.$
\end{itemize}

and

\begin{itemize}
\item $(-\left(  k+2\right)  ,k+1),k+2=2n,2n-1\newline(-\left(  k+2\right)
,k+1\mathbf{)}\overset{d)}{\rightarrow}$\ $\mathbf{(k+1,-k)}$\textbf{\ }if
\ \ $k+2=2n,2n-1\Leftrightarrow k=2n-2,2n-3\newline\mathbf{(}-\left(
k+2\right)  ,k+1\mathbf{)}\overset{d)}{\rightarrow}$\ $\mathbf{(k+1,-}\left(
\mathbf{k-1}\right)  \mathbf{)}$\ if \ $k+2=2n,2n-1\Leftrightarrow
k=2n-2,2n-3,$
\end{itemize}

then every maximal directed trail that starts at vertex $v=(-\left(
k+1\right)  ,k+2)$ if $k=1,...,2n-2,$ includes one of the vertices:
$w=(-\left(  k+1\right)  ,k+1),$ $\pm(-\left(  k+1\right)  ,k),$
$(k+1,-\left(  k-1\right)  )\ $and all corresponding directed subtrails
$\left(  v,w\right)  $ are directed paths. The same applies to the vertices
$v=\pm(-(k+2),k+2)$ and $v=(k+2,-(k+1))$ if $k=1,...,2n-2.$ Since we have
already proven that Claim \ref{clm} holds for vertices: $(-\left(  k+1\right)
,k+1),$ $\pm(-\left(  k+1\right)  ,k),$ $(k+1,-\left(  k-1\right)  )$ if
$k\in\left\{  1,...,2n-2\right\}  $ (Cases 1, 2 and 3), we conclude that Claim
\ref{clm} holds for the vertex $(-\left(  k+1\right)  ,k+2)$ if $k=1,...,2n-2$
and also for the vertices $\pm(-(k+2),k+2)$, $(k+2,-(k+1))$ if $k=1,...,2n-2.$
To finish the proof of the Case 4, we need to prove that Claim \ref{clm} holds
for the vertex $(-(k+2),k+3)$ if $k=1,...,2n-4.$ We have:

\begin{itemize}
\item $(-(k+2),k+3),$ $k+2=3,...,2n-2,\newline%
(-(k+2),k+3)\overset{f)}{\rightarrow}(k+3,-\left(  k+3\right)  )$ \ if
\ $k+2=3,...,2n-2\Leftrightarrow k+3=4,...,2n-1\newline%
(-(k+2),k+3)\overset{f)}{\rightarrow}\mathbf{(}k+3,-\left(  k+2\right)  )$
\ if \ $k+2=3,...,2n-2\Leftrightarrow k+3=4,...,2n-1$
\end{itemize}

Since

\begin{itemize}
\item $(k+3,-\left(  k+3\right)  ),$ $k+3=4,...,2n-1\newline(k+3,-\left(
k+3\right)  )\overset{a)}{\rightarrow}(-\left(  k+3\right)  ,k+3)$ \ if
\ $k+3=4,...,2n-1\newline(k+3,-\left(  k+3\right)  )\overset{a)}{\rightarrow
}(-\left(  k+3\right)  ,k+2)$ \ if \ $k+3=2n-1$

\item $(-\left(  k+3\right)  ,k+2),k+3=2n-1\newline(-\left(  k+3\right)
,k+2)\overset{d)}{\rightarrow}$\ $\left(  \mathbf{k+2,-}\left(  \mathbf{k+1}%
\right)  \right)  $ \ if \ $k+3=2n-1\Leftrightarrow k+2=2n-2\newline(-\left(
k+3\right)  ,k+2)\overset{d)}{\rightarrow}$\ $\left(  k+2,-k\right)  $ \ if
\ $k+3=2n-1\Leftrightarrow k+2=2n-2$
\end{itemize}

and

\begin{itemize}
\item $(-\left(  k+3\right)  ,k+3),$ $k+3=4,...,2n-1\newline(-\left(
k+3\right)  ,k+3)\overset{b)}{\rightarrow}$\ $\left(  k+3,-\left(  k+2\right)
\right)  $ \ if \ $k+3=4,...,2n-1$

\item $(k+3,-\left(  k+2\right)  ),$ $k+3=4,...,2n-1\newline(k+3,-\left(
k+2\right)  )\overset{c)}{\rightarrow}\mathbf{(-}\left(  \mathbf{k+2}\right)
\mathbf{,k+2)}$\ if \ $k+3=4,...,2n-1\Leftrightarrow k=0,...,2n-4\newline%
(k+3,-\left(  k+2\right)  )\overset{c)}{\rightarrow}\mathbf{(-}\left(
k+2\right)  ,k+1\mathbf{)}$\ if \ $k+3=n-1,...,2n-1\Leftrightarrow
k+2=n-2,...,2n-2$

\item $\left(  k+2,-k\right)  ,k+2=2n-2\newline\left(  k+2,-k\right)
\overset{g)}{\rightarrow}\left(  \mathbf{-k,k-1}\right)  ,$\ if
$k+2=2n-2\Leftrightarrow k=2n-4\newline\left(  k+2,-k\right)
\overset{g)}{\rightarrow}\left(  \mathbf{-k,k}\right)  ,\ $if
$k+2=2n-2\Leftrightarrow k=2n-4$

\item $(-\left(  k+2\right)  ,k+1),$ $k+2=n-2,...,2n-2\newline(-\left(
k+2\right)  ,k+1\mathbf{)}\overset{d)}{\rightarrow}$\ $\mathbf{(k+1,-k)}%
$\textbf{\ }if \ \ $k+2=n-2,...,2n-2\Leftrightarrow k=n-4,...,2n-4\newline%
\mathbf{(}-\left(  k+2\right)  ,k+1\mathbf{)}\overset{d)}{\rightarrow}%
$\ $(\mathbf{k+1,-}\left(  \mathbf{k-1}\right)  \mathbf{)}$\ if
\ $k+2=n+1,...,2n-2\Leftrightarrow k=n-1,...,2n-4,$
\end{itemize}

then every maximal directed trail that starts at vertex $v=(-(k+2),k+3),$ if
$k=1,...,2n-4$ includes one of the vertices: $w=\left(  -k,k-1\right)  $,
$\left(  -k,k\right)  ,(k+1,-k),$ $(k+1,-\left(  k-1\right)  ),$ $(-\left(
k+2\right)  ,k+2),$ $\left(  k+2,-\left(  k+1\right)  \right)  $ and all
corresponding directed subtrails $\left(  v,w\right)  $ are directed paths.
Since, by the inductive hypothesis, Claim \ref{clm} holds for vertices
$\left(  -k,k-1\right)  $, $\left(  -k,k\right)  \in\mathcal{V}_{k}\ $and as
we have already proven that Claim \ref{clm} holds for the vertices:
$\mathbf{(}k+1,-k),(k+1,-\left(  k-1\right)  ),$ $\ (-\left(  k+2\right)
,k+2),$ $\left(  k+2,-\left(  k+1\right)  \right)  $, we conclude that Claim
\ref{clm} also holds for the vertex $(-(k+2),k+3)$ when $k=1,...,2n-4.$ It
also follows from the above that Claim \ref{clm} holds for vertices:
$\pm(k+3,-\left(  k+3\right)  ),$ $(k+3,-\left(  k+2\right)  )$ if
$k=1,...,2n-4.$

\noindent\textbf{Case 5: }$\mathbf{\pm(k+1,-}\left(  \mathbf{k+3}\right)
\mathbf{)\in}\mathcal{V}_{k+1}^{\left(  3\right)  }\ $if $k\in\left\{
0,...,n-3\right\}  .$ $\newline$If $k=0$, then $\pm\mathbf{(}%
k+1,-(k+3)\mathbf{)}=\pm(1,-3)\in\mathcal{V}_{1}$ and by base case, Claim
\ref{clm} holds.

If $1\leq k\leq n-3,$ then $2\leq k+1\leq n-2,$ and we have:

\begin{itemize}
\item $(k+1,-\left(  k+3\right)  )\overset{i)}{\rightarrow}\left(
\mathbf{-}\left(  \mathbf{k+3}\right)  \mathbf{,k+3}\right)  $ \ if
\ $k+1=2,...,n-2\Leftrightarrow k=1,...,n-3,\newline(k+1,-\left(  k+3\right)
)\overset{i)}{\rightarrow}\left(  -\left(  k+3\right)  ,k+4\right)  $ \ if
\ $k+1=2,...,n-2\Leftrightarrow k+3=4,...,n,\newline(k+1,-\left(  k+3\right)
)\overset{i)}{\rightarrow}\left(  -\left(  k+3\right)  ,k+5\right)  $ \ if
\ $k+1=2,...,n-4,n\geq6\Leftrightarrow k+3=4,...,n-2,$ $n\geq6,$

\item $(-\left(  k+1\right)  ,k+3)\overset{j)}{\rightarrow}\left(
\mathbf{k+3,-}\left(  \mathbf{k+3}\right)  \right)  $ \ if
\ $k+1=2,...,n-2\Leftrightarrow k=0,...,n-3,\newline(-\left(  k+1\right)
,k+3)\overset{j)}{\rightarrow}\left(  k+3,\mathbf{-}\left(  k+4\right)
\right)  $ \ if \ $k+1=2,...,n-4,n\geq6\Leftrightarrow k+3=3,...,n-2,n\geq
6,\newline(-\left(  k+1\right)  ,k+3)\overset{j)}{\rightarrow}\left(
\mathbf{k+3,-}\left(  \mathbf{k+2}\right)  \right)  $ \ if
\ $k+1=2,...,n-2\Leftrightarrow k=1,...,n-3,$
\end{itemize}

Since we already prove that Claim \ref{clm} holds for the vertices $\pm\left(
\mathbf{-}\left(  k+3\right)  ,k+3\right)  ,$ $\left(  k+3,-\left(
k+2\right)  \right)  $ for $k=1,...,n-3$, to finish the proof of the Case 5,
we have to prove that Claim \ref{clm} holds for the vertices $\left(  -\left(
k+3\right)  ,k+4\right)  ,$ $\left(  -\left(  k+3\right)  ,k+5\right)  $ and
$\left(  k+3,\mathbf{-}\left(  k+4\right)  \right)  $ in the ranges of $k$ and
$n$ given above. We have:

\begin{itemize}
\item $\left(  -\left(  k+3\right)  ,k+4\right)  ,$ $k+3=4,...,n\newline%
\left(  -\left(  k+3\right)  ,k+4\right)  \overset{f)}{\rightarrow}\left(
k+4,-\left(  k+4\right)  \right)  $ if$\ k+3=4,...,n\Leftrightarrow
k+4=5,...,n+1\newline\left(  -\left(  k+3\right)  ,k+4\right)
\overset{f)}{\rightarrow}\left(  k+4,-\left(  k+3\right)  \right)  $
if$\ k+3=4,...,n\Leftrightarrow k+4=5,...,n+1$

\item $\left(  -\left(  k+3\right)  ,k+5\right)  ,$ $k+3=4,...,n-2,$
$n\geq6\newline\left(  -\left(  k+3\right)  ,k+5\right)
\overset{j)}{\rightarrow}\left(  k+5,-\left(  k+5\right)  \right)
,\ $if$\ k+3=4,...,n-2,n\geq6\Leftrightarrow k+5=6,...,n,n\geq6\newline\left(
-\left(  k+3\right)  ,k+5\right)  \overset{j)}{\rightarrow}\left(
k+5,-\left(  k+4\right)  \right)  \ $if$\ k+3=4,...,n-2,n\geq6\Leftrightarrow
k+5=6,...,n,n\geq6\newline\left(  -\left(  k+3\right)  ,k+5\right)
\overset{j)}{\rightarrow}\left(  k+5,-\left(  k+6\right)  \right)  $\ if
\ $k+3=4,...,n-4,n\geq8\Leftrightarrow k+5=6,...,n-2,n\geq8$

\item $\left(  k+3,\mathbf{-}\left(  k+4\right)  \right)  $ \ if
$k+3=4,...,n-2,n\geq6\newline\left(  k+3,\mathbf{-}\left(  k+4\right)
\right)  \overset{e)}{\rightarrow}\left(  \mathbf{-}\left(  k+4\right)
,k+4\right)  $ if $k+3=4,...,n-2,n\geq6\Leftrightarrow k+4=5,...,n-1,n\geq
6\newline\left(  k+3,\mathbf{-}\left(  k+4\right)  \right)
\overset{e)}{\rightarrow}\left(  \mathbf{-}\left(  k+4\right)  ,k+5\right)  $
if $k+3=4,...,n-2,n\geq6\Leftrightarrow k+4=5,...,n-1,$ $n\geq6.$
\end{itemize}

Since

\begin{itemize}
\item $\left(  k+4,-\left(  k+4\right)  \right)  ,$ $k+4=5,...,n+1,\newline%
\left(  k+4,-\left(  k+4\right)  \right)  \overset{a)}{\rightarrow}\left(
\mathbf{-}\left(  k+4\right)  ,k+4\right)  \overset{b)}{\rightarrow}\left(
k+4,-\left(  k+3\right)  \right)  $ if $k+4=5,...,n+1$

\item $\left(  \mathbf{-}\left(  k+4\right)  ,k+4\right)  ,\ k+4=5,...,n-1,$
$n\geq6\newline\left(  \mathbf{-}\left(  k+4\right)  ,k+4\right)
\overset{b)}{\rightarrow}\left(  k+4,-\left(  k+3\right)  \right)  $ if
$k+4=5,...,n-1,$ $n\geq6$

\item $\left(  k+4,-\left(  k+3\right)  \right)  $, $k+4=5,...,n+1,\newline%
\left(  k+4,-\left(  k+3\right)  \right)  \overset{c)}{\rightarrow}\left(
\mathbf{-}\left(  \mathbf{k+3}\right)  \mathbf{,k+3}\right)  $ if
$k+4=5,...,n+1\Leftrightarrow k=1,...,n-3,\newline\left(  k+4,-\left(
k+3\right)  \right)  \overset{c)}{\rightarrow}(-\left(  k+3\right)
,k+2)\overset{d)}{\rightarrow}$\ $\left(  \mathbf{k+2,-}\left(  \mathbf{k+1}%
\right)  \right)  $ if $k+4=n+1\Leftrightarrow k+3=n,\newline\left(
k+4,-\left(  k+3\right)  \right)  \overset{c)}{\rightarrow}(-\left(
k+3\right)  ,k+2)\overset{d)}{\rightarrow}$\ $\left(  \mathbf{k+2,-}\left(
\mathbf{k+1}\right)  \right)  $ if $k+4=n,$ $n\geq5\Leftrightarrow k+3=n-1,$
$n\geq5,\newline\left(  k+4,-\left(  k+3\right)  \right)
\overset{c)}{\rightarrow}(-\left(  k+3\right)  ,k+2)\overset{d)}{\rightarrow}%
$\ $\left(  \mathbf{k+2,-}\left(  \mathbf{k+1}\right)  \right)  $ if
$k+4=n-1,$ $n\geq6\Leftrightarrow k+3=n-2,$ $n\geq6,$

\item $\left(  \mathbf{-}\left(  k+4\right)  ,k+5\right)  ,k+4=5,...,n-1,$
$n\geq6,\newline\left(  \mathbf{-}\left(  k+4\right)  ,k+5\right)
\overset{f)}{\rightarrow}\left(  k+5,\mathbf{-}\left(  k+5\right)  \right)
\overset{a)}{\rightarrow}\left(  \mathbf{-}\left(  k+5\right)  ,k+5\right)
\overset{b)}{\rightarrow}\left(  k+5,\mathbf{-}\left(  k+4\right)  \right)  $
if $k+4=5,...,n-1,$ $n\geq6\Leftrightarrow k+5=6,...,n,n\geq6\newline\left(
\mathbf{-}\left(  k+4\right)  ,k+5\right)  \overset{f)}{\rightarrow}\left(
k+5,\mathbf{-}\left(  k+4\right)  \right)  $ if $k+4=5,...,n-1,$
$n\geq6\Leftrightarrow k+5=6,...,n,$ $n\geq6$

\item $\left(  k+5,-\left(  k+5\right)  \right)  ,$ $k+5=6,...,n,n\geq
6,\newline\left(  k+5,-\left(  k+5\right)  \right)  \overset{a)}{\rightarrow
}\left(  -\left(  k+5\right)  ,k+5\right)  $ $\overset{b)}{\rightarrow}$
$\left(  k+5,\mathbf{-}\left(  k+4\right)  \right)  $ if $k+5=6,...,n,$
$n\geq6,$

\item $\left(  k+5,\mathbf{-}\left(  k+4\right)  \right)  ,$ $k+5=6,...,n,$
$n\geq6,\newline\left(  k+5,-\left(  k+4\right)  \right)
\overset{c)}{\rightarrow}\left(  -\left(  k+4\right)  ,k+4\right)
\overset{b)}{\rightarrow}\left(  k+4,-\left(  k+3\right)  \right)  $ if
$k+5=6,...,n,n\geq6$ $\Leftrightarrow k+4=5,...,n-1,$ $n\geq6,\newline\left(
k+5,-\left(  k+4\right)  \right)  \overset{c)}{\rightarrow}\left(  -\left(
k+4\right)  ,k+3\right)  \overset{d)}{\rightarrow}$ $\left(  \mathbf{k+3}%
,-\left(  \mathbf{k+2}\right)  \right)  $ if $k+5=n-1,n$ if $n\geq6$ and
$k+5=n=5\Leftrightarrow k=n-6,n-5$ if $n\geq6$ and $k=n-5=0$ if $n=5,$

\item $\left(  k+5,-\left(  k+6\right)  \right)  $, $k+5=6,...,n-2,n\geq
8\newline\left(  k+5,-\left(  k+6\right)  \right)  \overset{e)}{\rightarrow
}\left(  -\left(  k+6\right)  ,k+6\right)  \overset{b)}{\rightarrow}\left(
k+6,-\left(  k+5\right)  \right)  ,$ if $k+5=6,...,n-2,n\geq8\Leftrightarrow
k+6=7,...,n-1,n\geq8\newline\left(  k+5,-\left(  k+6\right)  \right)
\overset{e)}{\rightarrow}\left(  -\left(  k+6\right)  ,k+7\right)  ,$ if
$k+5=6,...,n-2,n\geq8\Leftrightarrow k+6=7,...,n-1,n\geq8,$
\end{itemize}

\noindent we conclude that Claim \ref{clm} holds for the vertices $v=\left(
k+4,-\left(  k+4\right)  \right)  ,\left(  k+4,-\left(  k+3\right)  \right)
,$ $\left(  k+5,-\left(  k+5\right)  \right)  ,$ $\left(  k+5,-\left(
k+4\right)  \right)  ,$ $\left(  \mathbf{-}\left(  k+4\right)  ,k+4\right)  $
and $\left(  \mathbf{-}\left(  k+4\right)  ,k+5\right)  $ in the ranges of $k$
and $n$ given above. Namely, every maximal directed trail that starts at each
of these six vertices $v$ includes one of the vertices: $w=\left(
\mathbf{-}\left(  k+3\right)  ,k+2\right)  ,$ $\left(  \mathbf{-}\left(
k+3\right)  \mathbf{,}k+3\right)  ,$ $\left(  k+2,-\left(  k+1\right)
\right)  $ and all corresponding directed subtrails $\left(  v,w\right)  $ are
directed paths. Since we have already proved that Claim \ref{clm} holds for
these tree vertices $w$ in the ranges of $k$ given above, we conclude that
Claim \ref{clm} holds for the each of these six vertices $v$ in the ranges of
$k$ given above. To complete the proof of the Case 5, we have to prove that
Claim \ref{clm} holds for the vertices $\left(  k+6,-\left(  k+5\right)
\right)  \ $and $\left(  -\left(  k+6\right)  ,k+7\right)  $ if
$k+6=7,...,n-1,n\geq8.$ We have:

\begin{itemize}
\item $\left(  k+6,-\left(  k+5\right)  \right)  ,$ $k+6=7,...,n-1,n\geq
8,\newline\left(  k+6,-\left(  k+5\right)  \right)  \overset{c)}{\rightarrow
}\left(  -\left(  k+5\right)  ,k+5\right)  \overset{b)}{\rightarrow}\left(
\mathbf{k+5,-}\left(  \mathbf{k+4}\right)  \right)  ,$ if
${\small k+6=7,...,n-1,n\geq8}\Leftrightarrow{\small k+5=6,...,n-2,}$
${\small n\geq8}\newline\left(  k+6,-\left(  k+5\right)  \right)
\overset{c)}{\rightarrow}\left(  -\left(  k+5\right)  ,k+4\right)  ,$ if
$k+6=n-1,$ $n\geq8\Leftrightarrow k+5=n-2,$ $n\geq8$

\item $\left(  -\left(  k+6\right)  ,k+7\right)  ,$ $k+6=7,...,n-1,n\geq
8.\newline\left(  -\left(  k+6\right)  ,k+7\right)  \overset{f)}{\rightarrow
}\left(  k+7,-\left(  k+7\right)  \right)  \overset{a)}{\rightarrow}\left(
-\left(  k+7\right)  ,k+7\right)  \overset{b)}{\rightarrow}\left(
k+7,-\left(  k+6\right)  \right)  $ if $k+6=7,...,n-1,n\geq8\Leftrightarrow
k+7=8,...,n,$ $n\geq8,\newline\left(  -\left(  k+6\right)  ,k+7\right)
\overset{f)}{\rightarrow}\left(  k+7,-\left(  k+6\right)  \right)  $ if
$k+6=7,...,n-1,n\geq8\Leftrightarrow k+7=8,...,n,n\geq8,$
\end{itemize}

and

\begin{itemize}
\item $\left(  -\left(  k+5\right)  ,k+4\right)  ,$ $k+5=n-2,n\geq
8\newline\left(  -\left(  k+5\right)  ,k+4\right)  \overset{d)}{\rightarrow
}\left(  \mathbf{k+4,-}\left(  \mathbf{k+3}\right)  \right)  ,$ if
$k+5=n-2,n\geq8\Leftrightarrow k+4=n-3,n\geq8$

\item $\left(  k+7,-\left(  k+6\right)  \right)  ,$ $k+7=8,...,n,$
$n\geq8\newline\left(  k+7,-\left(  k+6\right)  \right)
\overset{c)}{\rightarrow}\left(  -\left(  k+6\right)  ,k+6\right)
\overset{b)}{\rightarrow}\left(  k+6,-\left(  k+5\right)  \right)  ,$ if
$k+7=8,...,n,n\geq8\Leftrightarrow k+6=7,...,n-1,n\geq8\newline\left(
k+7,-\left(  k+6\right)  \right)  \overset{c)}{\rightarrow}\left(  -\left(
k+6\right)  ,k+5\right)  ,$ if $k+7=n-1,n,$ $n\geq8$

\item $\left(  -\left(  k+6\right)  ,k+5\right)  ,$ if $k+7=n-1,n,$
$n\geq8\newline\left(  -\left(  k+6\right)  ,k+5\right)
\overset{d)}{\rightarrow}\left(  \mathbf{k+5,-}\left(  \mathbf{k+4}\right)
\right)  $ if $k+7=n-1,n,n\geq8.$
\end{itemize}

From this, we conclude that every maximal directed trail that starts at each
of the vertices $v=\left(  k+6,-\left(  k+5\right)  \right)  \ $and $v=\left(
-\left(  k+6\right)  ,k+7\right)  $ if $k+6=7,...,n-1,n\geq8$ includes one of
the vertices: $w=\left(  k+5,-\left(  k+4\right)  \right)  ,$\textbf{\ }%
$\left(  k+4,-\left(  k+3\right)  \right)  $ and all corresponding directed
subtrails $\left(  v,w\right)  $ are directed paths. Since we already prove
that Claim \ref{clm} holds for these two vertices $w$ in the ranges of $k$ and
$n$ given above, we conclude that Claim \ref{clm} also holds for the vertices
$v=\left(  k+6,-\left(  k+5\right)  \right)  \ $and $v=\left(  -\left(
k+6\right)  ,k+7\right)  $ if $k+6=7,...,n-1,n\geq8.$ Therefore, Claim
\ref{clm} holds for $\mathbf{\pm(}k+1,-\left(  k+3\right)  \mathbf{)\in
}\mathcal{V}_{k+1}^{\left(  3\right)  }$ which finishes the proof of Claim
\ref{clm}.

We have thus shown that the graph $G^{\prime}\left(  \Delta_{19}\right)  $ has
only two cycles: a trivial cycle $\pi_{0}=(0)$ and a non-zero cycle $\pi
_{2}=(-1,1)$. Since the cycles $\pi_{0}$\ and $\pi_{2}\ $are\ also the
cycles$\ $of\ the graph $G_{\varepsilon}(\Delta_{19})$ for each $\Big[\frac
{2}{3(2n+1)},\frac{2}{3(2n-1)}\Big)$ and for the corresponding cutout polygon
$P_{\varepsilon}(\pi_{2})$ we have $\Delta_{19}\cap P_{\varepsilon}(\pi
_{2})=\overline{WZ},$ the lemma is proven.
\end{proof}

We will prove Lemma \ref{Lemma-18-s} analogously to Lemma \ref{Lemma-18-prvi}.

\begin{algorithm}
\label{Lemma-18-s}Let $n\geq4\ $be integer, $\varepsilon\in\Big[\frac
{2}{3(2n+1)},\frac{2}{3(2n-1)}\Big)\ $and$\ l=\left\lfloor \frac{2n^{2}-3n}%
{4}\right\rfloor +1.\ $Let $\Delta_{18}^{\left(  s\right)  }=\Delta
(W,V_{s},V_{s+1})$ be triangle$,$ where $s=2,...l-1,$ $W=(\varepsilon,1)\ $and
$V_{k}=(\frac{2}{3}-\varepsilon,\frac{5}{3}-2\varepsilon-\frac{k}%
{n}\varepsilon)$, $k=2,...,l.$ Then $\Delta_{18}^{\left(  s\right)  }%
\cap\mathcal{D}_{2,\varepsilon}^{0}=\Delta_{18}^{\left(  s\right)  }%
\setminus\left\{  W\right\}  \ $for all $s=2,...l-1.$
\end{algorithm}

\begin{proof}
The algorithm of Lemma \ref{lem-alg} for $\mathcal{H}=\Delta_{18}^{\left(
s\right)  }$ and $\varepsilon\in\Big[\frac{2}{3(2n+1)},\frac{2}{3(2n-1)}%
\Big),$ leads to the graph $G_{\varepsilon}(\Delta_{18}^{\left(  s\right)  })$
which is subgraph of the graph $G^{\prime}(\Delta_{18}^{\left(  s\right)  })$
with $8a_{0}^{\left(  s\right)  }+2a_{1}^{\left(  s\right)  }-1$ vertices and
$14a_{0}^{\left(  s\right)  }+4a_{1}^{\left(  s\right)  }-6$ edges where
${\small a}_{0}^{\left(  s\right)  }=\left\lceil \frac{2n^{2}}{n+s}%
\right\rceil $ and ${\small a}_{1}^{\left(  s\right)  }=\left\lceil
\frac{n^{2}}{n+s}\right\rceil .$ The vertices and edges of a graph $G^{\prime
}(\Delta_{18}^{\left(  s\right)  })$ are given as follows:%
\begin{align*}
\mathcal{V}\left(  \Delta_{18}^{\left(  s\right)  }\right)   &  =\left\{
{\small \pm(t,-t),\pm(t,-(t-1)):t=0,...,a}_{0}^{\left(  s\right)  }\right\} \\
&  {\small \cup}\left\{  {\small \pm(t,-}\left(  {\small t+1}\right)
{\small ),\pm(t,-(t-2)):t=1,...,a}_{0}^{\left(  s\right)  }-1\right\} \\
&  {\small \cup}\left\{  {\small \pm(t,-(t+2)),}\text{ }{\small t=0,...,a}%
_{1}^{\left(  s\right)  }-1\right\}  ;
\end{align*}

\begin{itemize}
\item[\textbf{a)}] $(t,-t)\rightarrow(-t,t)$ for all $t=0,...,a_{0}^{\left(
s\right)  },$ and also $(t,-t)\rightarrow(-t,t-1)$ if $t=a_{0}^{\left(
s\right)  }.$

\item[\textbf{b)}] $(-t,t)\rightarrow(t,-\left(  t-1\right)  ),$ for all
$t=1,...,a_{0}^{\left(  s\right)  },$ and also $(-t,t)\rightarrow(t,-t)$ if
$t=1.$

\item[\textbf{c)}] $(t,-(t-1))\rightarrow(-\left(  t-1\right)  ,t-1),$ for all
$t=0,...,a_{0}^{\left(  s\right)  }$ or additionally $(t,-(t-1))\rightarrow
(-\left(  t-1\right)  ,t-2)$ if $t=a_{1}^{\left(  s\right)  }-1,...,a_{0}%
^{\left(  s\right)  }\ $and $a_{1}^{\left(  s\right)  }\geq3;$

\item[\textbf{d)}] $(-t,k-1)\rightarrow(t-1,-\left(  t-2\right)  ),$ for all
$t=0,...,a_{0}^{\left(  s\right)  },$ or additionally $(-t,t-1)\rightarrow
(t-1,-\left(  t-1\right)  )$ if $t=0,1$ and $(-t,t-1)\rightarrow(t-1,-\left(
t-3\right)  )$ if $t=a_{1}^{\left(  s\right)  }+1,...,a_{0}^{\left(  s\right)
}\ $and $a_{1}^{\left(  s\right)  }\geq3;$

\item[\textbf{e)}] $(t,-(t+1))\rightarrow(-(t+1),t+1)$ for all $t=1,...,a_{0}%
^{\left(  s\right)  }-1,$ or additionally $(t,-(t+1))\rightarrow(-(t+1),t+2)$
if $t=1,...,a_{0}^{\left(  s\right)  }-2.$

\item[\textbf{f)}] $(-t,t+1)\rightarrow(t+1,-\left(  t+1\right)  )$ for all
$t=1,...,a_{0}^{\left(  s\right)  }-1$ or additionally $(-t,t+1)\rightarrow
(t+1,-t)\ $if $t=2,...,a_{0}^{\left(  s\right)  }-1;$

\item[\textbf{g)}] $(t,-(t-2))\rightarrow(-(k-2),t-3)$ and
$(t,-(t-2))\rightarrow(-(t-2),t-2),$ for all $t=1,...,a_{0}^{\left(  s\right)
}-1;$

\item[\textbf{h)}] $(-t,t-2)\rightarrow(t-2,-\left(  t-3\right)  )$ and
$(-t,t-2)\rightarrow(t-2,-\left(  t-4\right)  ),$ for all $t=1,...,a_{0}%
^{\left(  s\right)  }-1$ or additionally $(-t,t-2)\rightarrow(t-2,-\left(
t-2\right)  )$ if $t=1.$

\item[\textbf{i)}] $(t,-(t+2))\rightarrow(-(t+2),t+2)$ and
$(t,-(t+2))\rightarrow(-(t+2),t+3)$ for all $t=0,...,a_{1}^{\left(  s\right)
}-1$ or additionally $(t,-(t+2))\rightarrow(-(t+2),t+4)$ if $t=0,...,a_{1}%
^{\left(  s\right)  }-3\ $and $a_{1}^{\left(  s\right)  }\geq3.$

\item[\textbf{j)}] $(-t,t+2)\rightarrow(t+2,-(t+2))$ for all $t=0,...,a_{1}%
^{\left(  s\right)  }-1$ or additionally $(-t,t+2)\rightarrow(t+2,-(t+3))$ if
$t=0,...,a_{1}^{\left(  s\right)  }-3,$ $a_{1}^{\left(  s\right)  }\geq3$ and
$(-t,t+2)\rightarrow(t+2,-(t+1))$ if $t=2,...,a_{1}^{\left(  s\right)  }-1,$
$a_{1}^{\left(  s\right)  }\geq3.$
\end{itemize}

Note that the following holds: $4={\small a}_{0}^{\left(  l-1\right)  }%
\leq{\small a}_{0}^{\left(  s\right)  }\leq{\small a}_{0}^{\left(  2\right)
}=2n-3$ and \ $2={\small a}_{1}^{\left(  l-1\right)  }\leq{\small a}%
_{1}^{\left(  s\right)  }\leq{\small a}_{1}^{\left(  2\right)  }=n-1.$ It can
also be shown that ${\small a}_{1}^{\left(  s\right)  }<{\small a}%
_{1}^{\left(  s\right)  }+1\leq{\small a}_{0}^{\left(  s\right)
}-1<{\small a}_{0}^{\left(  s\right)  }$ for all $s=2,...l-1,$ which implies:
${\small a}_{1}^{\left(  s\right)  }\geq m\Longrightarrow{\small a}%
_{0}^{\left(  s\right)  }\geq m+2$ for all $m=2,...,n-1.$

By the definition of the edges of the graph $G^{\prime}(\Delta_{18}^{\left(
s\right)  })$ we have: $(0,0)\overset{a)}{\rightarrow}(0,0)\ $and
$(-1,1)\overset{b)}{\rightarrow}(1,-1)\overset{a)}{\rightarrow}(-1,1)$ which
implies that the graph $G^{\prime}(\Delta_{18}^{\left(  s\right)  })$ has two
cycles, a trivial cycle $\pi_{0}=(0)$ and a non-zero cycle $\pi_{2}=(-1,1)$.
Our goal is to show that $\pi_{0}$ and $\pi_{1}$ are the only cycles of the
graph $G^{\prime}(\Delta_{18}^{\left(  s\right)  }).$ We will show the following:

\begin{claim}
\label{clm-osn-s}Let $v\in\mathcal{V}\left(  \Delta_{18}^{\left(  s\right)
}\right)  $ be arbitrary vertex. If there exist directed trail of the graph
$G^{\prime}(\Delta_{18}^{\left(  s\right)  })$ that starts at vertex $v$ and
includes some cycle $\pi$, then $\pi=\pi_{0}\ $or $\pi=\pi_{2}.$ Furthermore,
each maximal directed trail of the graph $G^{\prime}(\Delta_{18}^{\left(
s\right)  })$ which starts at vertex $v,$ includes exactly one of the cycles
$\pi_{0}\ $or $\pi_{2}.$
\end{claim}

Obviously, the Claim \ref{clm-osn-s} implies that $\pi_{0}$ and $\pi_{1}$ are
the only cycles of the graph $G^{\prime}(\Delta_{18}^{\left(  s\right)  }).$

For fixed $s\in\left\{  2,...l-1\right\}  $ and all integers $k\geq0$ let us
define sets%
\begin{align*}
\mathcal{V}_{k}^{\left(  s\right)  }  &  =\left\{  \pm(t,-t),\pm
(t,-(t-1)):t=0,...,\min\left\{  k,a_{0}^{\left(  s\right)  }\right\}  \right\}
\\
&  {\small \cup}\left\{  \pm(t,-(t+1)),\pm(t,-(t-2)):t=\min\left\{
k,1\right\}  ,...,\min\left\{  k,a_{0}^{\left(  s\right)  }-1\right\}
\right\} \\
&  {\small \cup}\left\{  \pm(t,-(t+2)):t=0,...,\min\left\{  k,a_{1}^{\left(
s\right)  }-1\right\}  \right\}  \text{.}%
\end{align*}
Note that $\mathcal{V}_{k}^{\left(  s\right)  }\subset\mathcal{V}%
_{k+1}^{\left(  s\right)  }$ for all $k=0,...,a_{0}^{\left(  s\right)  }$ and
$\mathcal{V}_{k}^{\left(  s\right)  }=\mathcal{V}\left(  \Delta_{18}^{\left(
s\right)  }\right)  \ $if $k\geq a_{0}^{\left(  s\right)  }\geq4.$ We will
prove the following:

\begin{claim}
\label{clm-s}For every integer $k\geq0$ and arbitrary vertex $v\in
\mathcal{V}_{k}^{\left(  s\right)  }$ the following holds:

\begin{itemize}
\item[\textbf{i.}] Every maximal directed trail of the graph $G^{\prime
}(\Delta_{18}^{\left(  s\right)  })$ that starts at $v$ includes at least one
of the vertices: $(0,0),$ $(1,-1),$ $(-1,1).$

\item[\textbf{ii.}] Every possible directed trail of the graph $G^{\prime
}(\Delta_{18}^{\left(  s\right)  })$ that starts at vertex $v$ which does not
include vertex $(0,0)$ and includes at most one of the vertices $(1,-1)\ $and
$(-1,1)$ is a directed path.
\end{itemize}
\end{claim}

Note that Claim \ref{clm-s} implies Claim \ref{clm-osn-s}. Namely, if maximal
directed trail of the graph $G^{\prime}(\Delta_{19})$ that starts at a vertex
$v\in G^{\prime}(\Delta_{18}^{\left(  s\right)  })$ includes one of the
vertices: $(0,0),(-1,1),$ $(1,-1)$, then that trail includes one of subtrails:
$(0,0)\overset{a)}{\rightarrow}\mathbf{(}0,0)$, $\mathbf{(}-1,1\mathbf{)}%
\overset{b)}{\rightarrow}\mathbf{(}1,-1\mathbf{)}\overset{a)}{\rightarrow
}(-1,1),$ $(-1,1)\overset{b)}{\rightarrow}(1,0)\overset{c)}{\rightarrow
}(0,0)\overset{a)}{\rightarrow}\mathbf{(}0,0),\ \mathbf{(}1,-1\mathbf{)}%
\overset{a)}{\rightarrow}(-1,1)\overset{b)}{\rightarrow}\mathbf{(}%
1,-1\mathbf{)}$ and $(1,-1)\overset{a)}{\rightarrow}%
(-1,1)\overset{b)}{\rightarrow}(1,0)\overset{c)}{\rightarrow}%
(0,0)\overset{a)}{\rightarrow}\mathbf{(}0,0).$ Since each of these subtrails
includes exactly one of the cycles $\pi_{0}\ $or $\pi_{2},$ to prove that
Claim \ref{clm-osn-s} it is enough to show that Claim \ref{clm-s} holds for
$k=a_{0}^{\left(  s\right)  }$, i.e for all vertices $v\in\mathcal{V}%
_{a_{0}^{\left(  s\right)  }}^{\left(  s\right)  }=\mathcal{V}\left(
\Delta_{18}^{\left(  s\right)  }\right)  .$

We will prove Claim \ref{clm-s} by induction on $k$.

\textit{Base case: }Sincw $a_{0}^{\left(  s\right)  }\geq4$ and $a_{1}%
^{\left(  s\right)  }\geq2,\ $then
\begin{align*}
\mathcal{V}_{0}^{\left(  s\right)  }  &  =\left\{  (0,0),\pm(0,1),\pm
(0,-2)\right\} \\
\mathcal{V}_{1}^{\left(  s\right)  }  &  =\left\{  (0,0),\pm(1,-1),\pm
(0,1),\pm(1,0),\pm(1,-2),\pm(1,1),\pm(0,-2),\pm(1,-3)\right\}
\end{align*}
for all $s=2,...l-1.$ Since $\mathcal{V}_{0}^{\left(  s\right)  }%
\subset\mathcal{V}_{1}^{\left(  s\right)  }$ we prove the Claim \ref{clm-s}
for each element of set $\mathcal{V}_{1}^{\left(  s\right)  }$ (more than we
need). By the definition of the edges of the graph $G^{\prime}(\Delta
_{18}^{\left(  s\right)  })$ given by \textbf{a)},..., \textbf{j)} we obtain
that all maximal directed trails that starts at each vertex from the set
$\mathcal{V}_{1}^{\left(  s\right)  }\ $are given by:

\begin{itemize}
\item $\mathbf{(0,0)}$ $\overset{a)}{\rightarrow}\mathbf{(0,0)}$

\item $\mathbf{(-1,1)}\overset{b)}{\rightarrow}\mathbf{(1,-1)}$\newline%
$\mathbf{(-1,1)}\overset{b)}{\rightarrow}(1,0)\overset{c)}{\rightarrow
}\mathbf{(0,0)}$

\item $\mathbf{(1,-1)}\overset{a)}{\rightarrow}\mathbf{(-1,1)}$

\item $\mathbf{(0,1)}\overset{c)}{\rightarrow}\mathbf{(1,-1)}$

\item $\mathbf{(1,0)}\overset{c)}{\rightarrow}\mathbf{(0,0)}$

\item $\mathbf{(-1,0)}\overset{d)}{\rightarrow}(0,1)\overset{c)}{\rightarrow
}\mathbf{(1,-1)}$\newline$\mathbf{(-1,0)}\overset{d)}{\rightarrow
}\mathbf{(0,0)}$

\item $\mathbf{(-1,2)}\overset{f)}{\rightarrow}(2,-2)\overset{a)}{\rightarrow
}(-2,2)\overset{b)}{\rightarrow}(2,-1)\overset{c)}{\rightarrow}\mathbf{(-1,1)}%
$\newline$\mathbf{(-1,2)}\overset{f)}{\rightarrow}%
(2,-2)\overset{a)}{\rightarrow}(-2,2)\overset{b)}{\rightarrow}%
(2,-1)\underset{a_{1}^{\left(  s\right)  }=3}{\overset{c)}{\rightarrow}%
}(-1,0)\overset{d)}{\rightarrow}(0,1)\overset{c)}{\rightarrow}\mathbf{(1,-1)}%
...$ if $a_{1}^{\left(  s\right)  }=3$\newline$\mathbf{(-1,2)}%
\overset{f)}{\rightarrow}(2,-2)\overset{a)}{\rightarrow}%
(-2,2)\overset{b)}{\rightarrow}(2,-1)\underset{a_{1}^{\left(  s\right)
}=3}{\overset{c)}{\rightarrow}}(-1,0)\overset{d)}{\rightarrow}\mathbf{(0,0)}%
...$ if $a_{1}^{\left(  s\right)  }=3$

\item $\mathbf{(0,-1)}\overset{d)}{\rightarrow}\mathbf{(-1,1)}$\newline%
$\mathbf{(0,-1)}\overset{d)}{\rightarrow}\mathbf{(-1,2)}%
\overset{f)}{\rightarrow}....$\newline

\item $\mathbf{(1,-2)}\overset{e)}{\rightarrow}(-2,2)\overset{b)}{\rightarrow
}(2,-1)\overset{c)}{\rightarrow}\mathbf{(-1,1)}$\newline$\mathbf{(1,-2)}%
\overset{e)}{\rightarrow}(-2,2)\overset{b)}{\rightarrow}(2,-1)\underset{a_{1}%
^{\left(  s\right)  }=3}{\overset{c)}{\rightarrow}}%
(-1,0)\overset{d)}{\rightarrow}(0,1)\overset{c)}{\rightarrow}\mathbf{(1,-1)}%
...$if $a_{1}^{\left(  s\right)  }=3\mathbf{\newline(1,-2)}%
\overset{e)}{\rightarrow}(-2,2)\overset{b)}{\rightarrow}(2,-1)\underset{a_{1}%
^{\left(  s\right)  }=3}{\overset{c)}{\rightarrow}}%
(-1,0)\overset{d)}{\rightarrow}\mathbf{(0,0)}...$ if $a_{1}^{\left(  s\right)
}=3\mathbf{\newline(1,-2)}\overset{e)}{\rightarrow}%
(-2,3)\overset{f)}{\rightarrow}(3,-3)\overset{a)}{\rightarrow}%
(-3,3)\overset{b)}{\rightarrow}(3,-2)\overset{c)}{\rightarrow}$ $...$%
\newline$\mathbf{(1,-2)}\overset{e)}{\rightarrow}%
(-2,3)\overset{f)}{\rightarrow}(3,-2)\overset{c)}{\rightarrow}$ $...$

\item $\mathbf{(1,1)}\overset{g)}{\rightarrow}\mathbf{(1,-1)}$\newline%
$\mathbf{(1,1)}\overset{g)}{\rightarrow}\mathbf{(1,-2)}%
\overset{e)}{\rightarrow}$ $...$

\item $\mathbf{(-1,-1)}\overset{h)}{\rightarrow}\mathbf{(-1,1)}$%
\newline$\mathbf{(-1,-1)}\overset{h)}{\rightarrow}\mathbf{(-1,2)}%
\overset{f)}{\rightarrow}...$\newline$\mathbf{(-1,-1)}\overset{h)}{\rightarrow
}(-1,3)\overset{j)}{\rightarrow}(3,-3)\overset{a)}{\rightarrow}%
(-3,3)\overset{b)}{\rightarrow}(3,-2)\overset{c)}{\rightarrow}$ $...$%
\newline$\mathbf{(-1,-1)}\overset{h)}{\rightarrow}(-1,3)\underset{a_{1}%
^{\left(  s\right)  }\geq4}{\overset{j)}{\rightarrow}}%
(3,-4)\overset{e)}{\rightarrow}...,$ \ if $a_{1}^{\left(  s\right)  }\geq4$

\item $\mathbf{(0,2)}\overset{j)}{\rightarrow}(2,-2)\overset{a)}{\rightarrow
}(-2,2)\overset{b)}{\rightarrow}(2,-1)\overset{c)}{\rightarrow}\mathbf{(-1,1)}%
$\newline$\mathbf{(0,2)}\underset{a_{1}^{\left(  s\right)  }\geq
3}{\overset{j)}{\rightarrow}}(2,-3)\overset{e)}{\rightarrow}%
(-3,3)\overset{b)}{\rightarrow}(3,-2)\overset{c)}{\rightarrow}$ $...$if
$a_{1}^{\left(  s\right)  }\geq3$\newline$\mathbf{(0,2)}\underset{a_{1}%
^{\left(  s\right)  }\geq3}{\overset{j)}{\rightarrow}}%
(2,-3)\overset{e)}{\rightarrow}(-3,4)\overset{f)}{\rightarrow}%
(4,-4)\overset{a)}{\rightarrow}\left(  -4,4\right)  \overset{b)}{\rightarrow
}(4,-3)\overset{c)}{\rightarrow}...$if $a_{1}^{\left(  s\right)  }\geq
3$\newline$\mathbf{(0,2)}\underset{a_{1}^{\left(  s\right)  }\geq
3}{\overset{j)}{\rightarrow}}(2,-3)\overset{e)}{\rightarrow}%
(-3,4)\overset{f)}{\rightarrow}(4,-3)\overset{c)}{\rightarrow}...$if
$a_{1}^{\left(  s\right)  }\geq3$

\item $\mathbf{(0,-2)}\overset{i)}{\rightarrow}(-2,2)\overset{b)}{\rightarrow
}\left(  2,-1\right)  \overset{c)}{\rightarrow}\mathbf{(-1,1)}$\newline%
$\mathbf{(0,-2)}\overset{i)}{\rightarrow}(-2,3)\overset{f)}{\rightarrow
}\left(  3,-3\right)  \overset{a)}{\rightarrow}(-3,3)\overset{b)}{\rightarrow
}(3,-2)\overset{c)}{\rightarrow}$ $...$\newline$\mathbf{(0,-2)}%
\overset{i)}{\rightarrow}(-2,3)\overset{f)}{\rightarrow}\left(  3,-2\right)
\overset{c)}{\rightarrow}$ $...$\newline$\mathbf{(0,-2)}\underset{a_{1}%
^{\left(  s\right)  }\geq3}{\overset{i)}{\rightarrow}}%
(-2,4)\overset{j)}{\rightarrow}(4,-3)\overset{c)}{\rightarrow}...$, \ if
$a_{1}^{\left(  s\right)  }\geq3\newline\mathbf{(0,-2)}\underset{a_{1}%
^{\left(  s\right)  }\geq3}{\overset{i)}{\rightarrow}}%
(-2,4)\overset{j)}{\rightarrow}\left(  4,-4\right)  \overset{a)}{\rightarrow
}\left(  -4,4\right)  \overset{b)}{\rightarrow}(4,-3)\overset{c)}{\rightarrow
}...a_{1}^{\left(  s\right)  }\geq3$\newline$\mathbf{(0,-2)}\underset{a_{1}%
^{\left(  s\right)  }\geq3}{\overset{i)}{\rightarrow}}(-2,4)\underset{a_{1}%
^{\left(  s\right)  }\geq5}{\overset{j)}{\rightarrow}}\left(  4,-5\right)
\overset{e)}{\rightarrow}\left(  -5,5\right)  \overset{b)}{\rightarrow}\left(
5,-4\right)  \overset{c)}{\rightarrow}...$, if $a_{1}^{\left(  s\right)  }%
\geq5$\newline$\mathbf{(0,-2)}\underset{a_{1}^{\left(  s\right)  }%
\geq3}{\overset{i)}{\rightarrow}}(-2,4)\underset{a_{1}^{\left(  s\right)
}\geq5}{\overset{j)}{\rightarrow}}\left(  {\small 4,-5}\right)
\overset{e)}{\rightarrow}\left(  {\small -5,6}\right)
\overset{f)}{\rightarrow}\left(  {\small 6,-6}\right)
\overset{a)}{\rightarrow}\left(  {\small -6,6}\right)
\overset{b)}{\rightarrow}\left(  {\small 6,-5}\right)
\overset{c)}{\rightarrow}...$ if $a_{1}^{\left(  s\right)  }\geq5$%
\newline$\mathbf{(0,-2)}\underset{a_{1}^{\left(  s\right)  }\geq
3}{\overset{i)}{\rightarrow}}(-2,4)\underset{a_{1}^{\left(  s\right)  }%
\geq5}{\overset{j)}{\rightarrow}}\left(  4,-5\right)  \overset{e)}{\rightarrow
}\left(  -5,6\right)  \overset{f)}{\rightarrow}\left(  6,-5\right)
\overset{c)}{\rightarrow}...$, \ if $a_{1}^{\left(  s\right)  }\geq5$\newline

\item $\mathbf{(1,-3)}\overset{i)}{\rightarrow}(-3,3)\overset{b)}{\rightarrow
}(3,-2)\overset{c)}{\rightarrow}...$\newline$\mathbf{(1,-3)}%
\overset{i)}{\rightarrow}(-3,4)\overset{f)}{\rightarrow}%
(4,-4)\overset{a)}{\rightarrow}\left(  -4,4\right)  \overset{b)}{\rightarrow
}(4,-3)\overset{c)}{\rightarrow}...$\newline$\mathbf{(1,-3)}%
\overset{i)}{\rightarrow}(-3,4)\overset{f)}{\rightarrow}%
(4,-3)\overset{c)}{\rightarrow}...\newline\mathbf{(1,-3)}\underset{a_{1}%
^{\left(  s\right)  }\geq4}{\overset{i)}{\rightarrow}}%
(-3,5)\overset{j)}{\rightarrow}...$, \ if $a_{1}^{\left(  s\right)  }\geq4$

\item $\mathbf{(-1,3)}\overset{j)}{\rightarrow}(3,-3)\overset{a)}{\rightarrow
}(-3,3)\overset{b)}{\rightarrow}(3,-2)\overset{c)}{\rightarrow}...$%
\newline$\mathbf{(-1,3)}\underset{a_{1}^{\left(  s\right)  }\geq
4}{\overset{j)}{\rightarrow}}(3,-4)\overset{e)}{\rightarrow}...$, \ if
$a_{1}^{\left(  s\right)  }\geq4.$
\end{itemize}

To complete the proof of the base case, we need to find all maximal directed
trails that starts at the vertices $(3,-2),$ $(4,-3)$ if $a_{1}^{\left(
s\right)  }\geq2,$ $(3,-4),$ $(-3,5)$ if $a_{1}^{\left(  s\right)  }\geq
4\ $and at $\left(  5,-4\right)  $ $\left(  6,-5\right)  $ if $a_{1}^{\left(
s\right)  }\geq5$. We have:

\begin{itemize}
\item $(3,-2)\overset{c)}{\rightarrow}(-2,2)\overset{b)}{\rightarrow
}(2,-1)\overset{c)}{\rightarrow}\mathbf{(-1,1)}$\newline$(3,-2)\underset{a_{1}%
^{\left(  s\right)  }=3,4}{\overset{c)}{\rightarrow}}%
(-2,1)\overset{d)}{\rightarrow}(1,0)\overset{c)}{\rightarrow}\mathbf{(0,0)}$
\ if $a_{1}^{\left(  s\right)  }=3,4$

\item $(4,-3)\overset{c)}{\rightarrow}(-3,3)\overset{b)}{\rightarrow
}\mathbf{(3,-2)}\overset{c)}{\rightarrow}...\newline(4,-3)\underset{a_{1}%
^{\left(  s\right)  }=3,4,5}{\overset{c)}{\rightarrow}}%
(-3,2)\overset{d)}{\rightarrow}(2,-1)\overset{c)}{\rightarrow}\mathbf{(-1,1),}%
$ \ if $a_{1}^{\left(  s\right)  }=3,4,5$

\item $(5,-4)\overset{c)}{\rightarrow}\left(  -4,4\right)
\overset{b)}{\rightarrow}\mathbf{(4,-3)}\overset{c)}{\rightarrow}...$if
$a_{1}^{\left(  s\right)  }\geq4\newline$ $(5,-4)\underset{a_{1}^{\left(
s\right)  }=4,5,6}{\overset{c)}{\rightarrow}}(-4,3)\overset{d)}{\rightarrow
}\mathbf{(3,-2)}\overset{c)}{\rightarrow}...,$ \ if $a_{1}^{\left(  s\right)
}=4,5,6$

\item $(3,-4)\overset{e)}{\rightarrow}(-4,4)\overset{b)}{\rightarrow
}\mathbf{(4,-3)}\overset{c)}{\rightarrow}...$if $a_{1}^{\left(  s\right)
}\geq4\newline(3,-4)\overset{e)}{\rightarrow}(-4,5)\overset{f)}{\rightarrow
}(5,-5)\overset{a)}{\rightarrow}(-5,5)\overset{b)}{\rightarrow}\mathbf{(5,-4)}%
\overset{c)}{\rightarrow}...$if $a_{1}^{\left(  s\right)  }\geq4\newline%
(3,-4)\overset{e)}{\rightarrow}(-4,5)\overset{f)}{\rightarrow}\mathbf{(5,-4)}%
\overset{c)}{\rightarrow}...$if $a_{1}^{\left(  s\right)  }\geq4$

\item $(-3,5)\overset{j)}{\rightarrow}(5,-5)\overset{b)}{\rightarrow
}\mathbf{(5,-4)}\overset{c)}{\rightarrow}...$if $a_{1}^{\left(  s\right)
}\geq4\newline(-3,5)\underset{a_{1}^{\left(  s\right)  }\geq
4}{\overset{j)}{\rightarrow}}\mathbf{(5,-4)}\overset{c)}{\rightarrow}...$if
$a_{1}^{\left(  s\right)  }\geq4\newline(-3,5)\underset{a_{1}^{\left(
s\right)  }\geq6}{\overset{j)}{\rightarrow}}(5,-6)\overset{e)}{\rightarrow
}....$if $a_{1}^{\left(  s\right)  }\geq6.$

\item $\left(  6,-5\right)  \overset{c)}{\rightarrow}\left(  -5,5\right)
\overset{b)}{\rightarrow}\left(  \mathbf{5,-4}\right)
\overset{c)}{\rightarrow}...$if $a_{1}^{\left(  s\right)  }\geq4\newline%
\left(  6,-5\right)  \underset{a_{1}^{\left(  s\right)  }%
=4,5,6,7}{\overset{c)}{\rightarrow}}\left(  -5,4\right)
\overset{d)}{\rightarrow}\left(  \mathbf{4,-3}\right)
\overset{c)}{\rightarrow}...$if $a_{1}^{\left(  s\right)  }=4,5,6,7\newline%
\left(  6,-5\right)  \underset{a_{1}^{\left(  s\right)  }%
=4,5,6,7}{\overset{c)}{\rightarrow}}\left(  -5,4\right)  \underset{a_{1}%
^{\left(  s\right)  }=4}{\overset{d)}{\rightarrow}}\left(  4,-2\right)
\overset{g)}{\rightarrow}...$if $a_{1}^{\left(  s\right)  }=4.$
\end{itemize}

Therefore, it remains to find all maximal directed trails that starts
at$\ $the vertices $\left(  6,-5\right)  $ and $\left(  -3,5\right)  $ if
$a_{1}^{\left(  s\right)  }\geq4,$ or more precisely for the vertices $\left(
4,-2\right)  $ if $a_{1}^{\left(  s\right)  }=4$ and $(5,-6)$ if
$a_{1}^{\left(  s\right)  }\geq6,$ respectively. We have:

\begin{itemize}
\item $\left(  4,-2\right)  \overset{g)}{\rightarrow}\left(  -2,2\right)
\overset{b)}{\rightarrow}\left(  2,-1\right)  \overset{c)}{\rightarrow
}\mathbf{(-1,1)}...$if $a_{1}^{\left(  s\right)  }=4\newline\left(
4,-2\right)  \overset{g)}{\rightarrow}\left(  -2,1\right)
\overset{d)}{\rightarrow}(1,0)\overset{c)}{\rightarrow}\mathbf{(0,0)}...$if
$a_{1}^{\left(  s\right)  }=4$
\end{itemize}

and

\begin{itemize}
\item $(5,-6)\overset{e)}{\rightarrow}\left(  -6,6\right)
\overset{b)}{\rightarrow}\mathbf{(6,-5)}\overset{c)}{\rightarrow}...$if
$a_{1}^{\left(  s\right)  }\geq6\newline(5,-6)\overset{e)}{\rightarrow}\left(
-6,7\right)  \overset{f)}{\rightarrow}(7,-7)\overset{b)}{\rightarrow}...$if
$a_{1}^{\left(  s\right)  }\geq6\newline(5,-6)\overset{e)}{\rightarrow}\left(
-6,7\right)  \overset{f)}{\rightarrow}(7,-6)\overset{c)}{\rightarrow}...$ if
$a_{1}^{\left(  s\right)  }\geq6.$
\end{itemize}

Since

\begin{itemize}
\item $(7,-6)\overset{c)}{\rightarrow}\left(  -6,6\right)
\overset{b)}{\rightarrow}\mathbf{(6,-5)}\overset{c)}{\rightarrow}%
...a_{1}^{\left(  s\right)  }\geq6\newline(7,-6)\underset{a_{1}^{\left(
s\right)  }=6,7,8}{\overset{c)}{\rightarrow}}\left(  -6,5\right)
\overset{d)}{\rightarrow}\left(  \mathbf{5,-4}\right)
\overset{c)}{\rightarrow}...$if $a_{1}^{\left(  s\right)  }=6,7,8$

\item $(7,-7)\overset{b)}{\rightarrow}\left(  -7,6\right)
\overset{d)}{\rightarrow}\left(  \mathbf{6,-5}\right)
\overset{c)}{\rightarrow}...$if $a_{1}^{\left(  s\right)  }\geq6\newline%
(7,-7)\overset{b)}{\rightarrow}\left(  -7,6\right)  \underset{a_{1}^{\left(
s\right)  }=6}{\overset{d)}{\rightarrow}}\left(  6,-4\right)
\overset{g)}{\rightarrow}...$if $a_{1}^{\left(  s\right)  }=6$
\end{itemize}

and

\begin{itemize}
\item $\left(  6,-4\right)  \overset{g)}{\rightarrow}\left(  -4,4\right)
\overset{b)}{\rightarrow}\mathbf{(4,-3)}\overset{c)}{\rightarrow}...$if
$a_{1}^{\left(  s\right)  }=6\newline\left(  6,-4\right)
\overset{g)}{\rightarrow}\left(  -4,3\right)  \overset{d)}{\rightarrow
}\mathbf{(3,-2)}\overset{c)}{\rightarrow}...$if $a_{1}^{\left(  s\right)  }=6$
\end{itemize}

we find all maximal directed trails that starts at$\ $each element of set
$\mathcal{V}_{1}^{\left(  s\right)  }$. It is obvious that every maximal
directed trail of the graph $\mathcal{V}\left(  \Delta_{18}^{\left(  s\right)
}\right)  $ which starts at each vertex $v\in\mathcal{V}_{1}^{\left(
s\right)  }$ includes one of the following vertices: $(0,0),$ $(1,-1),$
$(-1,1).$ It is also easy to see that any possible directed trail of the graph
$\mathcal{V}\left(  \Delta_{18}^{\left(  s\right)  }\right)  $ that starts at
any $v\in\mathcal{V}_{1}^{\left(  s\right)  }$ which does not include vertex
$(0,0)$ and includes at most one of the vertices $(1,-1)\ $and $(-1,1)$ is a
directed path. Note that if $v=\left(  -1,1\right)  ,$ there is exactly one
trail $(-1,1)\overset{b)}{\rightarrow}(1,0)$ with such property which is
obviously a directed path. But if $v=\left(  1,-1\right)  ,\ \left(
0,0\right)  $ or $\left(  1,0\right)  $ there is no directed trail with this
property. Therefore, Claim \ref{clm-s} also valid for the vertices $\left(
0,0\right)  ,$ $(-1,1),\ (1,-1)\ $and $\left(  1,0\right)  $. This concludes
the proof of the base case.

\textit{Induction step: }Suppose, for some $k\geq0$ that Claim \ref{clm-s}
holds for every element of the set $\mathcal{V}_{k}^{\left(  s\right)  }.$ We
have to prove that Claim \ref{clm-s} holds for every element of the set
$\mathcal{V}_{k+1}^{\left(  s\right)  }.$ If $k\geq a_{0}^{\left(  s\right)
},$ then $\mathcal{V}_{k}^{\left(  s\right)  }=\mathcal{V}_{k+1}^{\left(
s\right)  },$ and consequently Claim \ref{clm-s} holds for every element of
the set $\mathcal{V}_{k+1}^{\left(  s\right)  }.$ If $0\leq k\leq
a_{0}^{\left(  s\right)  }-1,$ let us define sets%
\begin{align}
\mathcal{V}_{k+1,1}^{\left(  s\right)  }  &  =\left\{  \pm(k+1,-\left(
k+1\right)  ){\small ,}\pm(k+1,-k)\right\}  \text{, if }0\leq k\leq
a_{0}^{\left(  s\right)  }-1,\label{v1-s}\\
\mathcal{V}_{k+1,2}^{\left(  s\right)  }  &  =\left\{  {\small \pm
(k+1,-(k+2)),\pm(k+1,-}\left(  {\small k-1}\right)  {\small )}\right\}
\text{, if }{\small 0\leq k\leq a}_{0}^{\left(  s\right)  }{\small -2}%
,\label{v2-s}\\
\mathcal{V}_{k+1,3}^{\left(  s\right)  }  &  =\left\{  \pm(k+1,-(k+3)\right\}
,\text{ if }0\leq k\leq a_{1}^{\left(  s\right)  }-2. \label{v3-s}%
\end{align}
Then%
\begin{align}
\mathcal{V}_{k+1}^{\left(  s\right)  }  &  =\mathcal{V}_{k}\cup\mathcal{V}%
_{k+1,1}^{\left(  s\right)  }\cup\mathcal{V}_{k+1,2}^{\left(  s\right)  }%
\cup\mathcal{V}_{k+1,3}^{\left(  s\right)  }\text{ \ if }0\leq k\leq
a_{1}^{\left(  s\right)  }-2,\label{v11-s}\\
\mathcal{V}_{k+1}^{\left(  s\right)  }  &  =\mathcal{V}_{k}\cup\mathcal{V}%
_{k+1,1}^{\left(  s\right)  }\cup\mathcal{V}_{k+1,2}^{\left(  s\right)
}\text{ \ if }a_{1}^{\left(  s\right)  }-1\leq k\leq a_{0}^{\left(  s\right)
}-2,\label{v22-s}\\
\mathcal{V}_{k+1}^{\left(  s\right)  }  &  =\mathcal{V}_{k}\cup\mathcal{V}%
_{k+1,1}^{\left(  s\right)  }\text{ \ if }k=a_{0}^{\left(  s\right)  }-1.
\label{v33-s}%
\end{align}
Since, by the induction hypothesis, Claim \ref{clm-s} holds for every element
of the set $\mathcal{V}_{k}^{\left(  s\right)  },$ according to (\ref{v11-s}),
(\ref{v22-s}) and (\ref{v33-s}), we have to prove that Claim \ref{clm-s} holds
for each element of the sets $\mathcal{V}_{k+1,1}^{\left(  s\right)
},\mathcal{V}_{k+1,2}^{\left(  s\right)  }\ $and $\mathcal{V}_{k+1,3}^{\left(
s\right)  }$ if $k$ lies in the ranges given in (\ref{v1-s}), (\ref{v2-s}) and
(\ref{v3-s}) respectively. Note that it suffices to show that for every
possible maximal directed trail of the graph $G^{\prime}(\Delta_{18}^{\left(
s\right)  })$ that starts at each vertex $v\ $belonging to the sets
$\mathcal{V}_{k+1,1}^{\left(  s\right)  },\mathcal{V}_{k+1,2}^{\left(
s\right)  }\ $and $\mathcal{V}_{k+1,3}^{\left(  s\right)  }$ there exist a
vertex $w\in\mathcal{V}\left(  \Delta_{18}^{\left(  s\right)  }\right)  $ that
is included in that trail for which Claim \ref{clm-s} holds and such that the
corresponding directed subtrail $\left(  v,w\right)  $ is a directed path.
Then we can conclude that Claim \ref{clm-s} also holds for the vertex $v$.
Using the definition of the edges of the graph $G^{\prime}(\Delta
_{18}^{\left(  s\right)  })$, we deduce:

\noindent\textbf{Case 1: }$\mathbf{\pm(k+1,-}\left(  \mathbf{k+1}\right)
\mathbf{)}\mathbf{\in}\mathcal{V}_{k+1,1}^{\left(  s\right)  }\ $if
$k\in\left\{  0,...,a_{0}^{\left(  s\right)  }-1\right\}  .$ $\newline$

If $k=0$, then $\pm(k+1,-\left(  k+1\right)  )=\pm(1,-1)\in\mathcal{V}%
_{1}^{\left(  s\right)  }$ and by the base case, Claim \ref{clm} holds.

If $1\leq k\leq a_{0}^{\left(  s\right)  }-1,$ then $2\leq k+1\leq
a_{0}^{\left(  s\right)  },$ and we have:

\begin{itemize}
\item $(k+1,-\left(  k+1\right)  ),$ $k+1=2,...,a_{0}^{\left(  s\right)
},\newline(k+1,-\left(  k+1\right)  )\overset{a)}{\rightarrow}(-\left(
k+1\right)  ,k+1)$ \ if \ $k+1=2,...,a_{0}^{\left(  s\right)  }$%
,$\newline(k+1,-\left(  k+1\right)  )\overset{a)}{\rightarrow}(-\left(
k+1\right)  ,k)$ \ if \ $k+1=a_{0}^{\left(  s\right)  },$

\item $(-\left(  k+1\right)  ,k+1),$ $k+1=2,...,a_{0}^{\left(  s\right)
},\newline(-\left(  k+1\right)  ,k+1)\overset{b)}{\rightarrow}(k+1,-k)$ \ if
\ $k+1=2,...,a_{0}^{\left(  s\right)  }.$
\end{itemize}

Since

\begin{itemize}
\item $(-\left(  k+1\right)  ,k),$ $k+1=a_{0}^{\left(  s\right)  }%
,\newline(-\left(  k+1\right)  ,k)\overset{d)}{\rightarrow}(k,-\left(
k-1\right)  )\ $if$\ k+1=a_{0}^{\left(  s\right)  }\Leftrightarrow
k=a_{0}^{\left(  s\right)  }-1,\newline(-\left(  k+1\right)
,k)\overset{d)}{\rightarrow}(k,-\left(  k-2\right)  )$ \ if \ $k+1=a_{0}%
^{\left(  s\right)  },$ $a_{1}^{\left(  s\right)  }\geq3\Leftrightarrow
k=a_{0}^{\left(  s\right)  }-1,$ $a_{1}^{\left(  s\right)  }\geq3,$

\item $(k+1,-k),$ $k+1=2,...,a_{0}^{\left(  s\right)  }\newline%
(k+1,-k)\overset{c)}{\rightarrow}(-k,k),$ \ if \ $k+1=2,...,a_{0}^{\left(
s\right)  }\Leftrightarrow k=1,...,a_{0}^{\left(  s\right)  }-1\newline%
(k+1,-k)\overset{c)}{\rightarrow}(-k,k-1),$\ if \ $k+1=a_{1}^{\left(
s\right)  }+1,...,a_{0}^{\left(  s\right)  },$ $a_{1}^{\left(  s\right)  }%
\geq3\Leftrightarrow k=a_{1}^{\left(  s\right)  },...,a_{0}^{\left(  s\right)
}-1,$ $a_{1}^{\left(  s\right)  }\geq3,$
\end{itemize}

then every maximal directed trail that starts at each of the two vertices:
$v=\pm(k+1,-\left(  k+1\right)  )$ if $1\leq k\leq2n-1,$ includes one of the
vertices $w=\left(  -k,k\right)  ,$ $\left(  -k,k-1\right)  ,$ $(k,-\left(
k-1\right)  ),$ $(k,-\left(  k-2\right)  )\in\mathcal{V}_{k}^{\left(
s\right)  }$ and all corresponding directed subtrails $\left(  v,w\right)  $
are the directed paths. By the inductive hypothesis, we conclude that Claim
\ref{clm-s} holds for the vertices $\pm(k+1,-\left(  k+1\right)  )\mathbf{\in
}\mathcal{V}_{k+1,1}^{\left(  s\right)  }$ too.

\noindent\textbf{Case 2: }$\pm\mathbf{(k+1,-k)\in}\mathcal{V}_{k+1,1}^{\left(
s\right)  }\ $if $k\in\left\{  0,...,a_{0}^{\left(  s\right)  }-1\right\}  .$
$\newline$If $k=0$, then $\pm(k+1,-k)=\pm(1,0)\in\mathcal{V}_{1}^{\left(
s\right)  }$ and by the base case, Claim \ref{clm-s} holds.

If $1\leq k\leq a_{0}^{\left(  s\right)  }-1,$ then $2\leq k+1\leq
a_{0}^{\left(  s\right)  },\ $and we have:

\begin{itemize}
\item $(k+1,-k),$ $k+1=2,...,a_{0}^{\left(  s\right)  }\newline%
(k+1,-k)\overset{c)}{\rightarrow}(-k,k),$ \ if \ $k+1=2,...,a_{0}^{\left(
s\right)  }\Leftrightarrow k=1,...,a_{0}^{\left(  s\right)  }-1,\newline%
(k+1,-k)\overset{c)}{\rightarrow}(-k,k-1),$ \ if $k+1=a_{1}^{\left(  s\right)
}-1,...,a_{0}^{\left(  s\right)  },$ $a_{1}^{\left(  s\right)  }%
\geq3\Leftrightarrow k=a_{1}^{\left(  s\right)  }-2,...,a_{0}^{\left(
s\right)  }-1,$ $a_{1}^{\left(  s\right)  }\geq3,$

\item $(-\left(  k+1\right)  ,k),$ $k+1=2,...,a_{0}^{\left(  s\right)
},\newline(-\left(  k+1\right)  ,k)\overset{d)}{\rightarrow}(k,-\left(
k-1\right)  ),$ \ if \ $k+1=2,...,a_{0}^{\left(  s\right)  }\Leftrightarrow
k=1,...,a_{0}^{\left(  s\right)  }-1,\newline(-\left(  k+1\right)
,k)\overset{d)}{\rightarrow}(k,-\left(  k-2\right)  )$ \ if $k+1=a_{1}%
^{\left(  s\right)  }+1,...,a_{0}^{\left(  s\right)  },$ $a_{1}^{\left(
s\right)  }\geq3\Leftrightarrow k=a_{1}^{\left(  s\right)  },...,a_{0}%
^{\left(  s\right)  }-1,$ $a_{1}^{\left(  s\right)  }\geq3.$
\end{itemize}

Therefore, every maximal directed trail that starts at each of the two
vertices: $v=\pm(k+1,-k),$ if $1\leq k\leq a_{0}^{\left(  s\right)  }-1$
includes the edge $v\rightarrow w,$ where $w$ is one of the vertices $(-k,k),$
$(-k,k-1),$ $(k,-\left(  k-1\right)  ),$ $(k,-\left(  k-2\right)
)\in\mathcal{V}_{k}^{\left(  s\right)  }$ and, by the inductive hypothesis, we
conclude that Claim \ref{clm-s} holds for $\pm(k+1,-k)\mathbf{\in}%
\mathcal{V}_{k+1,1}^{\left(  s\right)  }$.

\noindent\textbf{Case 3: }$\mathbf{\pm(k+1,-}\left(  \mathbf{k-1}\right)
\mathbf{)\in}\mathcal{V}_{k+1,2}^{\left(  s\right)  }\ $if $k\in\left\{
0,...,a_{0}^{\left(  s\right)  }-2\right\}  .$ $\newline$

If $k=0$, then $\pm(k+1,-\left(  k-1\right)  )=\pm(1,1)\in\mathcal{V}%
_{1}^{\left(  s\right)  }$ and by base case, Claim \ref{clm-s} holds.

If $1\leq k\leq a_{0}^{\left(  s\right)  }-2,$ then $2\leq k+1\leq
a_{0}^{\left(  s\right)  }-1,$ and we have:

\begin{itemize}
\item $(k+1,-(k-1)),$ $k+1=2,...,a_{0}^{\left(  s\right)  }-1\newline%
(k+1,-(k-1))\overset{g)}{\rightarrow}(-(k-1),k-1)$ \ if \ $k+1=2,...,a_{0}%
^{\left(  s\right)  }-1\Leftrightarrow k-1=0,...,a_{0}^{\left(  s\right)
}-3,\newline(k+1,-(k-1))\overset{g)}{\rightarrow}(-(k-1),k-2)$ \ if
\ $k+1=2,...,a_{0}^{\left(  s\right)  }-1\Leftrightarrow k-1=0,...,a_{0}%
^{\left(  s\right)  }-3,\newline$

\item $(-\left(  k+1\right)  ,k-1),k+1=2,...,a_{0}^{\left(  s\right)
}-1\newline(-\left(  k+1\right)  ,k-1)\overset{h)}{\rightarrow}(k-1,-\left(
k-2\right)  )$ \ if \ $k+1=2,...,a_{0}^{\left(  s\right)  }-1\Leftrightarrow
k-1=0,...,a_{0}^{\left(  s\right)  }-3\newline(-\left(  k+1\right)
,k-1)\overset{h)}{\rightarrow}(k-1,-\left(  k-3\right)  )$ \ if
\ $k+1=2,...,a_{0}^{\left(  s\right)  }-1\Leftrightarrow k-1=0,...,a_{0}%
^{\left(  s\right)  }-3.$
\end{itemize}

Since every maximal directed trail that starts at each of the two vertices:
$v=\pm(k+1,-\left(  k-1\right)  \mathbf{)},$ if $1\leq k\leq a_{0}^{\left(
s\right)  }-2$ includes the edge $v\rightarrow w,$ where $w$ is one of the
vertices $(-(k-1),k-1),$ $(-(k-1),k-2),$ $(k-1,-\left(  k-2\right)  ),$
$(k-1,-\left(  k-3\right)  )\in\mathcal{V}_{k-1}^{\left(  s\right)  }%
\subset\mathcal{V}_{k}^{\left(  s\right)  }$, by the inductive hypothesis, we
conclude that Claim \ref{clm-s} holds for $\pm(k+1,-(k-1))\mathbf{\in
}\mathcal{V}_{k+1,2}^{\left(  s\right)  }.$

\noindent\textbf{Case 4: }$\pm\mathbf{(k+1,-(k+2))\in}\mathcal{V}%
_{k+1,2}^{\left(  2\right)  }$ if $k\in\left\{  0,...,a_{0}^{\left(  s\right)
}-2\right\}  .$ $\newline$

If $k=0$, then $\pm\mathbf{(}k+1,-(k+2)\mathbf{)}=\pm(1,-2)\in\mathcal{V}%
_{1}^{\left(  s\right)  }$ and by base case, Claim \ref{clm-s} holds.

If $1\leq k\leq a_{0}^{\left(  s\right)  }-2,$ then $2\leq k+1\leq
a_{0}^{\left(  s\right)  }-1,$ and we have:

\begin{itemize}
\item $(k+1,-(k+2)),$ $k+1=2,...,a_{0}^{\left(  s\right)  }-1\newline%
(k+1,-(k+2))\overset{e)}{\rightarrow}(-(k+2),k+2)$ \ if \ $k+1=2,...,a_{0}%
^{\left(  s\right)  }-1\Leftrightarrow k+2=3,...,a_{0}^{\left(  s\right)
},\newline(k+1,-(k+2))\overset{e)}{\rightarrow}(-(k+2),k+3)$ \ if
\ $k+1=2,...,a_{0}^{\left(  s\right)  }-2\Leftrightarrow k+2=3,...,a_{0}%
^{\left(  s\right)  }-1.$

\item $(-\left(  k+1\right)  ,k+2),$ $k+1=2,...,a_{0}^{\left(  s\right)
}-1\newline(-\left(  k+1\right)  ,k+2)\overset{f)}{\rightarrow}(k+2,-(k+2))$
\ if \ $k+1=2,...,a_{0}^{\left(  s\right)  }-1\Leftrightarrow k+2=3,...,a_{0}%
^{\left(  s\right)  },\newline(-\left(  k+1\right)
,k+2)\overset{f)}{\rightarrow}(k+2,-(k+1)),$ \ if \ $k+1=2,...,a_{0}^{\left(
s\right)  }-1\Leftrightarrow k+2=3,...,a_{0}^{\left(  s\right)  }.$
\end{itemize}

Since

\begin{itemize}
\item $\mathbf{(}-(k+2),k+2\mathbf{),}$ $k+2=3,...,a_{0}^{\left(  s\right)
}\mathbf{,}\newline(-(k+2),k+2)\overset{b)}{\rightarrow}\mathbf{(}k+2,-\left(
k+1\right)  )$ \ if \ $k+2=3,...,a_{0}^{\left(  s\right)  },$

\item $(k+2,-(k+2),$ $k+2=3,...,a_{0}^{\left(  s\right)  },\newline%
(k+2,-(k+2))\overset{a)}{\rightarrow}(-\left(  k+2\right)  ,k+2)$\ if
\ $k+2=3,...,a_{0}^{\left(  s\right)  },\newline%
(k+2,-(k+2))\overset{a)}{\rightarrow}\mathbf{(}-\left(  k+2\right)  ,k+1)$\ if
\ \ $k+2=a_{0}^{\left(  s\right)  }$

\item $\mathbf{(}k+2,-(k+1)\mathbf{)},$ $k+2=3,...,a_{0}^{\left(  s\right)
}\newline(k+2,-(k+1))\overset{c)}{\rightarrow}\mathbf{(-}\left(
\mathbf{k+1}\right)  \mathbf{,k+1)}$\ if \ $k+2=3,...,a_{0}^{\left(  s\right)
}\Leftrightarrow k=1,...,a_{0}^{\left(  s\right)  }-2\newline%
(k+2,-(k+1))\overset{c)}{\rightarrow}\mathbf{(-}\left(  \mathbf{k+1}\right)
\mathbf{,k)}$\ if \ $k+2=a_{1}^{\left(  s\right)  }-1,...,a_{0}^{\left(
s\right)  },$ $a_{1}^{\left(  s\right)  }\geq4\Leftrightarrow k=a_{1}^{\left(
s\right)  }-3,...,a_{0}^{\left(  s\right)  }-2,$ $a_{1}^{\left(  s\right)
}\geq4$
\end{itemize}

and

\begin{itemize}
\item $(-\left(  k+2\right)  ,k+1),$ $k+2=a_{0}^{\left(  s\right)  }%
\newline(-\left(  k+2\right)  ,k+1\mathbf{)}\overset{d)}{\rightarrow}%
$\ $\mathbf{(k+1,-k)}$\textbf{\ }if \ \ $k+2=a_{0}^{\left(  s\right)
}\Leftrightarrow k=a_{0}^{\left(  s\right)  }-2,\newline\mathbf{(}-\left(
k+2\right)  ,k+1\mathbf{)}\overset{d)}{\rightarrow}$\ $\mathbf{(k+1,-}\left(
\mathbf{k-1}\right)  \mathbf{)}$\ if \ $k+2=a_{0}^{\left(  s\right)  },$
$a_{1}^{\left(  s\right)  }\geq3\Leftrightarrow k=a_{0}^{\left(  s\right)
}-2,$ $a_{1}^{\left(  s\right)  }\geq3,$
\end{itemize}

then every maximal directed trail that starts at vertex $v=(-\left(
k+1\right)  ,k+2)$ if $k=1,...,a_{0}^{\left(  s\right)  }-2,$ includes one of
the vertices: $w=(-\left(  k+1\right)  ,k+1),$ $\pm(-\left(  k+1\right)  ,k),$
$(k+1,-\left(  k-1\right)  )\ $and all corresponding directed subtrails
$\left(  v,w\right)  $ are directed paths. The same applies to the vertices
$v=\pm(-(k+2),k+2)$ and $v=(k+2,-(k+1))$ if $k=1,...,a_{0}^{\left(  s\right)
}-2.$ Since we have already proven that Claim \ref{clm-s} holds for vertices:
$(-\left(  k+1\right)  ,k+1),$ $\pm(-\left(  k+1\right)  ,k),$ $(k+1,-\left(
k-1\right)  )$ if $k\in\left\{  1,...,a_{0}^{\left(  s\right)  }-2\right\}  $
(Cases 1, 2 and 3), we conclude that Claim \ref{clm-s} holds for the vertex
$(-\left(  k+1\right)  ,k+2)$ if $k=1,...,a_{0}^{\left(  s\right)  }-2$ and
also for the vertices $\pm(-(k+2),k+2)$, $(k+2,-(k+1))$ if $k=1,...,a_{0}%
^{\left(  s\right)  }-2.$ To finish the proof of the Case 4, we need to prove
that Claim \ref{clm-s} holds for the vertex $(-(k+2),k+3)$ if $k+2=3,...,a_{0}%
^{\left(  s\right)  }-1.$ We have:

\begin{itemize}
\item $(-(k+2),k+3),$ $k+2=3,...,a_{0}^{\left(  s\right)  }-1\newline%
(-(k+2),k+3)\overset{f)}{\rightarrow}(k+3,-\left(  k+3\right)  )$ \ if
\ $k+2=3,...,a_{0}^{\left(  s\right)  }-1\Leftrightarrow k+3=4,...,a_{0}%
^{\left(  s\right)  },\newline(-(k+2),k+3)\overset{f)}{\rightarrow}%
\mathbf{(}k+3,-\left(  k+2\right)  )$ \ if \ $k+2=3,...,a_{0}^{\left(
s\right)  }-1\Leftrightarrow k+3=4,...,a_{0}^{\left(  s\right)  },$
\end{itemize}

Since

\begin{itemize}
\item $(k+3,-\left(  k+3\right)  ),$ $k+3=4,...,a_{0}^{\left(  s\right)
}\newline(k+3,-\left(  k+3\right)  )\overset{a)}{\rightarrow}(-\left(
k+3\right)  ,k+3)$ \ if \ $k+3=4,...,a_{0}^{\left(  s\right)  }\newline%
(k+3,-\left(  k+3\right)  )\overset{a)}{\rightarrow}(-\left(  k+3\right)
,k+2)$ \ if \ $k+3=a_{0}^{\left(  s\right)  },$

\item $(-\left(  k+3\right)  ,k+2),$ $k+3=a_{0}^{\left(  s\right)  }%
,\newline(-\left(  k+3\right)  ,k+2)\overset{d)}{\rightarrow}$\ $\left(
\mathbf{k+2,-}\left(  \mathbf{k+1}\right)  \right)  $ \ if \ $k+3=a_{0}%
^{\left(  s\right)  }\Leftrightarrow k+2=a_{0}^{\left(  s\right)  }%
-1,\newline(-\left(  k+3\right)  ,k+2)\overset{d)}{\rightarrow}$\ $\left(
k+2,-k\right)  $ \ if \ $k+3=a_{0}^{\left(  s\right)  },$ $a_{1}^{\left(
s\right)  }\geq3\Leftrightarrow k+2=a_{0}^{\left(  s\right)  }-1,$
$a_{1}^{\left(  s\right)  }\geq3,$
\end{itemize}

and

\begin{itemize}
\item $(-\left(  k+3\right)  ,k+3),$ $k+3=4,...,a_{0}^{\left(  s\right)
},\newline(-\left(  k+3\right)  ,k+3)\overset{b)}{\rightarrow}$\ $\left(
k+3,-\left(  k+2\right)  \right)  $ \ if \ $k+3=4,...,a_{0}^{\left(  s\right)
},$

\item $(k+3,-\left(  k+2\right)  ),$ $k+3=4,...,a_{0}^{\left(  s\right)
},\newline(k+3,-\left(  k+2\right)  )\overset{c)}{\rightarrow}\mathbf{(-}%
\left(  \mathbf{k+2}\right)  \mathbf{,k+2)}$\ if \ $k+3=4,...,a_{0}^{\left(
s\right)  }\Leftrightarrow k=1,...,a_{0}^{\left(  s\right)  }-3,\newline%
(k+3,-\left(  k+2\right)  )\overset{c)}{\rightarrow}\mathbf{(-}\left(
k+2\right)  ,k+1\mathbf{)}$\ if \ $k+3=a_{1}^{\left(  s\right)  }%
-1,...,a_{0}^{\left(  s\right)  },$ $a_{1}^{\left(  s\right)  }\geq
5\Leftrightarrow k+2=a_{1}^{\left(  s\right)  }-2,...,a_{0}^{\left(  s\right)
}-1,$ $a_{1}^{\left(  s\right)  }\geq5,$

\item $\left(  k+2,-k\right)  ,$ $k+2=a_{0}^{\left(  s\right)  }-1,$
$a_{1}^{\left(  s\right)  }\geq3\newline\left(  k+2,-k\right)
\overset{g)}{\rightarrow}\left(  \mathbf{-k,k-1}\right)  ,$\ if $k+2=a_{0}%
^{\left(  s\right)  }-1,$ $a_{1}^{\left(  s\right)  }\geq3\Leftrightarrow
k=a_{0}^{\left(  s\right)  }-3,$ $a_{1}^{\left(  s\right)  }\geq
3\newline\left(  k+2,-k\right)  \overset{g)}{\rightarrow}\left(
\mathbf{-k,k}\right)  ,\ $if $k+2=a_{0}^{\left(  s\right)  }-1,$
$a_{1}^{\left(  s\right)  }\geq3\Leftrightarrow k=a_{0}^{\left(  s\right)
}-3,$ $a_{1}^{\left(  s\right)  }\geq3$

\item $(-\left(  k+2\right)  ,k+1),$ $k+2=a_{1}^{\left(  s\right)
}-2,...,a_{0}^{\left(  s\right)  }-1,$ $a_{1}^{\left(  s\right)  }%
\geq5\newline(-\left(  k+2\right)  ,k+1\mathbf{)}\overset{d)}{\rightarrow}%
$\ $\mathbf{(k+1,-k)}$\textbf{\ \ }if $k+2=a_{1}^{\left(  s\right)
}-2,...,a_{0}^{\left(  s\right)  }-1,$ $a_{1}^{\left(  s\right)  }%
\geq5\Leftrightarrow k=a_{1}^{\left(  s\right)  }-4,...,a_{0}^{\left(
s\right)  }-3,$ $a_{1}^{\left(  s\right)  }\geq5\newline\mathbf{(}-\left(
k+2\right)  ,k+1\mathbf{)}\overset{d)}{\rightarrow}$\ $(\mathbf{k+1,-}\left(
\mathbf{k-1}\right)  \mathbf{)}$\ if \ $k+2=a_{1}^{\left(  s\right)
}+1,...,a_{0}^{\left(  s\right)  }-1,$ $a_{1}^{\left(  s\right)  }%
\geq5\Leftrightarrow k=a_{1}^{\left(  s\right)  }-1,...,a_{0}^{\left(
s\right)  }-3,$ $a_{1}^{\left(  s\right)  }\geq5.$
\end{itemize}

then every maximal directed trail that starts at vertex $v=(-(k+2),k+3),$ if
$k=1,...,a_{0}^{\left(  s\right)  }-3$ includes one of the vertices:
$w=\left(  -k,k-1\right)  $, $\left(  -k,k\right)  ,(k+1,-k),$ $(k+1,-\left(
k-1\right)  ),$ $(-\left(  k+2\right)  ,k+2),$ $\left(  k+2,-\left(
k+1\right)  \right)  $ and all corresponding directed subtrails $\left(
v,w\right)  $ are directed paths. Since, by the inductive hypothesis, Claim
\ref{clm-s} holds for vertices $\left(  -k,k-1\right)  $, $\left(
-k,k\right)  \in\mathcal{V}_{k}^{\left(  2\right)  }\ $and as we have already
proven that Claim \ref{clm-s} holds for the vertices: $\mathbf{(}%
k+1,-k),(k+1,-\left(  k-1\right)  ),$ $(-\left(  k+2\right)  ,k+2),$ $\left(
k+2,-\left(  k+1\right)  \right)  $, we conclude that Claim \ref{clm-s} also
holds for the vertex $(-(k+2),k+3)$ when $k=1,...,a_{0}^{\left(  s\right)
}-3.$ It also follows from the above that Claim \ref{clm-s} holds for
vertices: $\pm(k+3,-\left(  k+3\right)  ),$ $(k+3,-\left(  k+2\right)  )$ if
$k=1,...,a_{0}^{\left(  s\right)  }-3.$

\noindent\textbf{Case 5: }$\mathbf{\pm(k+1,-}\left(  \mathbf{k+3}\right)
\mathbf{)\in}\mathcal{V}_{k+1,3}^{\left(  s\right)  }\ $if $k\in\left\{
0,...,a_{1}^{\left(  s\right)  }-2\right\}  .$ $\newline$If $a_{1}^{\left(
s\right)  }=2,$ then $k=0$ and $\pm\mathbf{(}k+1\mathbf{\mathbf{,-}}\left(
k+3\right)  \mathbf{)}=\pm(1,-3)\in\mathcal{V}_{1}^{\left(  s\right)  }$.
Therefore, by base case, Claim \ref{clm-s} holds if $a_{1}^{\left(  s\right)
}=2.$

Let $a_{1}^{\left(  s\right)  }\geq3.$ If $k=0,$ then $\pm\mathbf{(}%
k+1\mathbf{\mathbf{,-}}\left(  k+3\right)  \mathbf{)}=\pm(1,-3)\in
\mathcal{V}_{1}^{\left(  s\right)  }$ and by base case, Claim \ref{clm-s} holds.

If $1\leq k\leq a_{1}^{\left(  s\right)  }-2,$ then $2\leq k+1\leq
a_{1}^{\left(  s\right)  }-1,$ and we have:

\begin{itemize}
\item $(k+1,-\left(  k+3\right)  )\overset{i)}{\rightarrow}\left(
\mathbf{-}\left(  \mathbf{k+3}\right)  \mathbf{,k+3}\right)  $ \ if
\ $k+1=2,...,a_{1}^{\left(  s\right)  }-1\Leftrightarrow k=1,...,a_{1}%
^{\left(  s\right)  }-2\newline(k+1,-\left(  k+3\right)
)\overset{i)}{\rightarrow}\left(  -\left(  k+3\right)  ,k+4\right)  $ \ if
\ $k+1=2,...,a_{1}^{\left(  s\right)  }-1\Leftrightarrow k+3=4,...,a_{1}%
^{\left(  s\right)  }+1\newline(k+1,-\left(  k+3\right)
)\overset{i)}{\rightarrow}\left(  -\left(  k+3\right)  ,k+5\right)  $ \ if
\ $k+1=2,...,a_{1}^{\left(  s\right)  }-3,$ $a_{1}^{\left(  s\right)  }%
\geq5\Leftrightarrow k+3=4,...,a_{1}^{\left(  s\right)  }-1,$ $a_{1}^{\left(
s\right)  }\geq5$

\item $(-\left(  k+1\right)  ,k+3)\overset{j)}{\rightarrow}\left(
\mathbf{k+3,-}\left(  \mathbf{k+3}\right)  \right)  $ \ if \ $k+1=2,...,a_{1}%
^{\left(  s\right)  }-1\Leftrightarrow k=1,...,a_{1}^{\left(  s\right)
}-2\newline(-\left(  k+1\right)  ,k+3)\overset{j)}{\rightarrow}\left(
k+3,\mathbf{-}\left(  k+4\right)  \right)  $ \ if \ $k+1=2,...,a_{1}^{\left(
s\right)  }-3,$ $a_{1}^{\left(  s\right)  }\geq5\Leftrightarrow
k+3=4,...,a_{1}^{\left(  s\right)  }-1,$ $a_{1}^{\left(  s\right)  }%
\geq5\newline(-\left(  k+1\right)  ,k+3)\overset{j)}{\rightarrow}\left(
\mathbf{k+3,-}\left(  \mathbf{k+2}\right)  \right)  $ \ if \ $k+1=2,...,a_{1}%
^{\left(  s\right)  }-1\Leftrightarrow k=1,...,a_{1}^{\left(  s\right)  }-2.$
\end{itemize}

Since we already prove that Claim \ref{clm-s} holds for the vertices
$\pm\left(  \mathbf{-}\left(  k+3\right)  ,k+3\right)  ,\left(  k+3,-\left(
k+2\right)  \right)  $ if $k=1,...,a_{0}^{\left(  s\right)  }-3$, to finish
the proof of the Case 5, we have to prove that Claim \ref{clm-s} holds for the
vertices $\left(  -\left(  k+3\right)  ,k+4\right)  ,$ $\left(  -\left(
k+3\right)  ,k+5\right)  $ and $\left(  k+3,\mathbf{-}\left(  k+4\right)
\right)  $ in the ranges of $k$ and $a_{1}^{\left(  s\right)  }$ given above.
We have:

\begin{itemize}
\item $\left(  -\left(  k+3\right)  ,k+4\right)  ,$ $k+3=4,...,a_{1}^{\left(
s\right)  }+1\newline\left(  -\left(  k+3\right)  ,k+4\right)
\overset{f)}{\rightarrow}\left(  k+4,-\left(  k+4\right)  \right)  $
if$\ k+3=4,...,a_{1}^{\left(  s\right)  }+1\Leftrightarrow k+4=5,...,a_{1}%
^{\left(  s\right)  }+2,\newline\left(  -\left(  k+3\right)  ,k+4\right)
\overset{f)}{\rightarrow}\left(  k+4,-\left(  k+3\right)  \right)  $
if$\ k+3=4,...,a_{1}^{\left(  s\right)  }+1\Leftrightarrow k+4=5,...,a_{1}%
^{\left(  s\right)  }+2,$

\item $\left(  -\left(  k+3\right)  ,k+5\right)  ,$ $k+3=4,...,a_{1}^{\left(
s\right)  }-1,$ $a_{1}^{\left(  s\right)  }\geq5\newline\left(  -\left(
k+3\right)  ,k+5\right)  \overset{j)}{\rightarrow}\left(  k+5,-\left(
k+5\right)  \right)  ,\ $if$\ k+3=4,...,a_{1}^{\left(  s\right)  }-1,$
$a_{1}^{\left(  s\right)  }\geq5\Leftrightarrow k+5=6,...,a_{1}^{\left(
s\right)  }+1,$ $a_{1}^{\left(  s\right)  }\geq5,\newline\left(  -\left(
k+3\right)  ,k+5\right)  \overset{j)}{\rightarrow}\left(  k+5,-\left(
k+4\right)  \right)  \ $if$\ k+3=4,...,a_{1}^{\left(  s\right)  }-1,$
$a_{1}^{\left(  s\right)  }\geq5\Leftrightarrow k+5=6,...,a_{1}^{\left(
s\right)  }+1,$ $a_{1}^{\left(  s\right)  }\geq5,\newline\left(  -\left(
k+3\right)  ,k+5\right)  \overset{j)}{\rightarrow}\left(  k+5,-\left(
k+6\right)  \right)  $\ if \ $k+3=4,...,a_{1}^{\left(  s\right)  }-3,$
$a_{1}^{\left(  s\right)  }\geq7\Leftrightarrow k+5=6,...,a_{1}^{\left(
s\right)  }-1,$ $a_{1}^{\left(  s\right)  }\geq7,$

\item $\left(  k+3,\mathbf{-}\left(  k+4\right)  \right)  $ \ if
$k+3=4,...,a_{1}^{\left(  s\right)  }-1,$ $a_{1}^{\left(  s\right)  }%
\geq5\newline\left(  k+3,\mathbf{-}\left(  k+4\right)  \right)
\overset{e)}{\rightarrow}\left(  \mathbf{-}\left(  k+4\right)  ,k+4\right)  $
if $k+3=4,...,a_{1}^{\left(  s\right)  }-1,$ $a_{1}^{\left(  s\right)  }%
\geq5\Leftrightarrow k+4=5,...,a_{1}^{\left(  s\right)  },$ $a_{1}^{\left(
s\right)  }\geq5\newline\left(  k+3,\mathbf{-}\left(  k+4\right)  \right)
\overset{e)}{\rightarrow}\left(  \mathbf{-}\left(  k+4\right)  ,k+5\right)  $
if $k+3=4,...,a_{1}^{\left(  s\right)  }-1,$ $a_{1}^{\left(  s\right)  }%
\geq5\Leftrightarrow k+4=5,...,a_{1}^{\left(  s\right)  },$ $a_{1}^{\left(
s\right)  }\geq5.$
\end{itemize}

Since

\begin{itemize}
\item $\left(  k+4,-\left(  k+4\right)  \right)  ,$ $k+4=5,...,a_{1}^{\left(
s\right)  }+2\newline\left(  k+4,-\left(  k+4\right)  \right)
\overset{a)}{\rightarrow}\left(  \mathbf{-}\left(  k+4\right)  ,k+4\right)
\overset{b)}{\rightarrow}\left(  k+4,-\left(  k+3\right)  \right)  $ if
$k+4=5,...,a_{1}^{\left(  s\right)  }+2\newline$

\item $\left(  \mathbf{-}\left(  k+4\right)  ,k+4\right)  ,$ $k+4=5,...,a_{1}%
^{\left(  s\right)  },$ $a_{1}^{\left(  s\right)  }\geq5\newline\left(
\mathbf{-}\left(  k+4\right)  ,k+4\right)  \overset{b)}{\rightarrow}\left(
k+4,-\left(  k+3\right)  \right)  $ if $k+4=5,...,a_{1}^{\left(  s\right)  },$
$a_{1}^{\left(  s\right)  }\geq5$ $\newline$

\item $\left(  k+4,-\left(  k+3\right)  \right)  $, $k+4=5,...,a_{1}^{\left(
s\right)  }+2\newline\left(  k+4,-\left(  k+3\right)  \right)
\overset{c)}{\rightarrow}\left(  \mathbf{-}\left(  \mathbf{k+3}\right)
\mathbf{,k+3}\right)  $ if $k+4=5,...,a_{1}^{\left(  s\right)  }%
+2\Leftrightarrow k=1,...,a_{1}^{\left(  s\right)  }-2\newline\left(
k+4,-\left(  k+3\right)  \right)  \overset{c)}{\rightarrow}(-\left(
k+3\right)  ,k+2)\overset{d)}{\rightarrow}$\ $\left(  \mathbf{k+2,-}\left(
\mathbf{k+1}\right)  \right)  $ if $k+4=a_{1}^{\left(  s\right)  }-1,$
$a_{1}^{\left(  s\right)  }\geq6\Leftrightarrow k+3=a_{1}^{\left(  s\right)
}-2,$ $a_{1}^{\left(  s\right)  }\geq6,\newline\left(  k+4,-\left(
k+3\right)  \right)  \overset{c)}{\rightarrow}(-\left(  k+3\right)
,k+2)\overset{d)}{\rightarrow}$\ $\left(  \mathbf{k+2,-}\left(  \mathbf{k+1}%
\right)  \right)  $ if $k+4=a_{1}^{\left(  s\right)  },$ $a_{1}^{\left(
s\right)  }\geq5\Leftrightarrow k+3=a_{1}^{\left(  s\right)  }-1,$
$a_{1}^{\left(  s\right)  }\geq5,\newline\left(  k+4,-\left(  k+3\right)
\right)  \overset{c)}{\rightarrow}(-\left(  k+3\right)
,k+2)\overset{d)}{\rightarrow}$\ $\left(  \mathbf{k+2,-}\left(  \mathbf{k+1}%
\right)  \right)  $ if $k+4=a_{1}^{\left(  s\right)  }+1,$ $a_{1}^{\left(
s\right)  }\geq4\Leftrightarrow k+3=a_{1}^{\left(  s\right)  },$
$a_{1}^{\left(  s\right)  }\geq4,\newline\left(  k+4,-\left(  k+3\right)
\right)  \overset{c)}{\rightarrow}(-\left(  k+3\right)
,k+2)\overset{d)}{\rightarrow}$\ $\left(  \mathbf{k+2,-}\left(  \mathbf{k+1}%
\right)  \right)  $ if $k+4=a_{1}^{\left(  s\right)  }+2,$ $a_{1}^{\left(
s\right)  }\geq3\Leftrightarrow k+3=a_{1}^{\left(  s\right)  }+1,$
$a_{1}^{\left(  s\right)  }\geq3,$

\item $\left(  \mathbf{-}\left(  k+4\right)  ,k+5\right)  ,$ $k+4=5,...,a_{1}%
^{\left(  s\right)  },$ $a_{1}^{\left(  s\right)  }\geq5,\newline\left(
\mathbf{-}\left(  k+4\right)  ,k+5\right)  \overset{f)}{\rightarrow}\left(
k+5,\mathbf{-}\left(  k+5\right)  \right)  \overset{a)}{\rightarrow}\left(
\mathbf{-}\left(  k+5\right)  ,k+5\right)  \overset{b)}{\rightarrow}\left(
k+5,\mathbf{-}\left(  k+4\right)  \right)  $if $k+4=5,...,a_{1}^{\left(
s\right)  },$ $a_{1}^{\left(  s\right)  }\geq5,\newline\left(  \mathbf{-}%
\left(  k+4\right)  ,k+5\right)  \overset{f)}{\rightarrow}\left(
k+5,\mathbf{-}\left(  k+4\right)  \right)  $ if $k+4=5,...,a_{1}^{\left(
s\right)  },$ $a_{1}^{\left(  s\right)  }\geq5,$

\item $\left(  k+5,-\left(  k+5\right)  \right)  ,$ $k+5=6,...,a_{1}^{\left(
s\right)  }+1,a_{1}^{\left(  s\right)  }\geq5,\newline\left(  k+5,-\left(
k+5\right)  \right)  \overset{a)}{\rightarrow}\left(  -\left(  k+5\right)
,k+5\right)  $ $\overset{b)}{\rightarrow}$ $\left(  k+5,\mathbf{-}\left(
k+4\right)  \right)  $ if $k+5=6,...,a_{1}^{\left(  s\right)  }+1,a_{1}%
^{\left(  s\right)  }\geq5,$

\item $\left(  k+5,\mathbf{-}\left(  k+4\right)  \right)  ,$ $k+5=6,...,a_{1}%
^{\left(  s\right)  }+1,$ $a_{1}^{\left(  s\right)  }\geq5,\newline\left(
k+5,-\left(  k+4\right)  \right)  \overset{c)}{\rightarrow}\left(  -\left(
k+4\right)  ,k+4\right)  \overset{b)}{\rightarrow}\left(  k+4,-\left(
k+3\right)  \right)  $ if $k+5=6,...,a_{1}^{\left(  s\right)  }+1,a_{1}%
^{\left(  s\right)  }\geq5\newline\left(  k+5,-\left(  k+4\right)  \right)
\overset{c)}{\rightarrow}\left(  -\left(  k+4\right)  ,k+3\right)
\overset{d)}{\rightarrow}$ $\left(  \mathbf{k+3},-\left(  \mathbf{k+2}\right)
\right)  $ if $k+5=a_{1}^{\left(  s\right)  }-1$ if $a_{1}^{\left(  s\right)
}\geq7$ $\newline\left(  k+5,-\left(  k+4\right)  \right)
\overset{c)}{\rightarrow}\left(  -\left(  k+4\right)  ,k+3\right)
\overset{d)}{\rightarrow}$ $\left(  \mathbf{k+3},-\left(  \mathbf{k+2}\right)
\right)  $ if $k+5=a_{1}^{\left(  s\right)  },a_{1}^{\left(  s\right)  }+1$ if
$a_{1}^{\left(  s\right)  }\geq6,\newline\left(  k+5,-\left(  k+4\right)
\right)  \overset{c)}{\rightarrow}\left(  -\left(  k+4\right)  ,k+3\right)
\overset{d)}{\rightarrow}$ $\left(  \mathbf{k+3},-\left(  \mathbf{k+2}\right)
\right)  $ if $k+5=a_{1}^{\left(  s\right)  }+1$ if $a_{1}^{\left(  s\right)
}\geq5,$

\item $\left(  k+5,-\left(  k+6\right)  \right)  $, $k+5=6,...,a_{1}^{\left(
s\right)  }-1,$ $a_{1}^{\left(  s\right)  }\geq7,\newline\left(  k+5,-\left(
k+6\right)  \right)  \overset{e)}{\rightarrow}\left(  -\left(  k+6\right)
,k+6\right)  \overset{b)}{\rightarrow}\left(  k+6,-\left(  k+5\right)
\right)  ,$ if $k+5=6,...,a_{1}^{\left(  s\right)  }-1,$ $a_{1}^{\left(
s\right)  }\geq7,\Leftrightarrow k+6=7,...,a_{1}^{\left(  s\right)  },$
$a_{1}^{\left(  s\right)  }\geq7\newline\left(  k+5,-\left(  k+6\right)
\right)  \overset{e)}{\rightarrow}\left(  -\left(  k+6\right)  ,k+7\right)  ,$
if $k+5=6,...,a_{1}^{\left(  s\right)  }-1,$ $a_{1}^{\left(  s\right)  }%
\geq7\Leftrightarrow k+6=7,...,a_{1}^{\left(  s\right)  },$ $a_{1}^{\left(
s\right)  }\geq7,$
\end{itemize}

we conclude that Claim \ref{clm-s} holds for the vertices $v=\left(
k+4,-\left(  k+4\right)  \right)  ,$ $\left(  k+4,-\left(  k+3\right)
\right)  ,$ $\left(  k+5,-\left(  k+5\right)  \right)  ,$ $\left(
k+5,-\left(  k+4\right)  \right)  ,$ $\left(  \mathbf{-}\left(  k+4\right)
,k+4\right)  $ and $\left(  \mathbf{-}\left(  k+4\right)  ,k+5\right)  $ in
the ranges of $k$ and $a_{1}^{\left(  s\right)  }$ given above. Namely, every
maximal directed trail that starts at each of these six vertices $v$ includes
one of the vertices: $w=\left(  \mathbf{-}\left(  k+3\right)  ,k+2\right)  ,$
$\left(  \mathbf{-}\left(  k+3\right)  \mathbf{,}k+3\right)  ,$ $\left(
k+2,-\left(  k+1\right)  \right)  $ and all corresponding directed subtrails
$\left(  v,w\right)  $ are directed paths. Since we have already proved that
Claim \ref{clm-s} holds for these tree vertices $w$ in the ranges of $k$ given
above, we conclude that Claim \ref{clm-s} holds for the each of these six
vertices $v$ in the ranges of $k$ and $a_{1}^{\left(  s\right)  }$ given
above. To complete the proof of the Case 5, we have to prove that Claim
\ref{clm-s} holds for the vertices $\left(  k+6,-\left(  k+5\right)  \right)
\ $and $\left(  -\left(  k+6\right)  ,k+7\right)  $ if $k+6=7,...,a_{1}%
^{\left(  s\right)  },$ $a_{1}^{\left(  s\right)  }\geq7.$ We have:

\begin{itemize}
\item $\left(  k+6,-\left(  k+5\right)  \right)  ,$ $k+6=7,...,a_{1}^{\left(
s\right)  },$ $a_{1}^{\left(  s\right)  }\geq7\newline\left(  k+6,-\left(
k+5\right)  \right)  \overset{c)}{\rightarrow}\left(  -\left(  k+5\right)
,k+5\right)  \overset{b)}{\rightarrow}\left(  \mathbf{k+5,-}\left(
\mathbf{k+4}\right)  \right)  ,$ if $k+6=7,...,a_{1}^{\left(  s\right)  },$
$a_{1}^{\left(  s\right)  }\geq7\Leftrightarrow k+5=6,...,a_{1}^{\left(
s\right)  }-1,$ $a_{1}^{\left(  s\right)  }\geq7\newline\left(  k+6,-\left(
k+5\right)  \right)  \overset{c)}{\rightarrow}\left(  -\left(  k+5\right)
,k+4\right)  ,$ if $k+6=a_{1}^{\left(  s\right)  }-1,a_{1}^{\left(  s\right)
},$ $a_{1}^{\left(  s\right)  }\geq7\Leftrightarrow k+5=a_{1}^{\left(
s\right)  }-2,$ $a_{1}^{\left(  s\right)  }-1,$ $a_{1}^{\left(  s\right)
}\geq7$

\item $\left(  -\left(  k+6\right)  ,k+7\right)  ,k+6=7,...,a_{1}^{\left(
s\right)  },$ $a_{1}^{\left(  s\right)  }\geq7\newline\left(  -\left(
k+6\right)  ,k+7\right)  \overset{f)}{\rightarrow}\left(  k+7,-\left(
k+7\right)  \right)  \overset{a)}{\rightarrow}\left(  -\left(  k+7\right)
,k+7\right)  \overset{b)}{\rightarrow}\left(  k+7,-\left(  k+6\right)
\right)  $ if $k+6=7,...,a_{1}^{\left(  s\right)  },$ $a_{1}^{\left(
s\right)  }\geq7\Leftrightarrow k+7=8,...,a_{1}^{\left(  s\right)  }+1,$
$a_{1}^{\left(  s\right)  }\geq7\newline\left(  -\left(  k+6\right)
,k+7\right)  \overset{f)}{\rightarrow}\left(  k+7,-\left(  k+6\right)
\right)  $ if $k+6=7,...,a_{1}^{\left(  s\right)  },$ $a_{1}^{\left(
s\right)  }\geq7\Leftrightarrow k+7=8,...,a_{1}^{\left(  s\right)  }+1,$
$a_{1}^{\left(  s\right)  }\geq7$

and

\item $\left(  -\left(  k+5\right)  ,k+4\right)  ,$ $k+5=a_{1}^{\left(
s\right)  }-2,$ $a_{1}^{\left(  s\right)  }-1,$ $a_{1}^{\left(  s\right)
}\geq7\newline\left(  -\left(  k+5\right)  ,k+4\right)
\overset{d)}{\rightarrow}\left(  \mathbf{k+4,-}\left(  \mathbf{k+3}\right)
\right)  ,$ if $k+5=a_{1}^{\left(  s\right)  }-2,a_{1}^{\left(  s\right)
}-1,$ $a_{1}^{\left(  s\right)  }\geq7\Leftrightarrow k+4=a_{1}^{\left(
s\right)  }-3,a_{1}^{\left(  s\right)  }-2,$ $a_{1}^{\left(  s\right)  }\geq7$

\item $\left(  k+7,-\left(  k+6\right)  \right)  ,$ $k+7=8,...,a_{1}^{\left(
s\right)  }+1,$ $a_{1}^{\left(  s\right)  }\geq7\newline\left(  k+7,-\left(
k+6\right)  \right)  \overset{c)}{\rightarrow}\left(  -\left(  k+6\right)
,k+6\right)  \overset{b)}{\rightarrow}\left(  k+6,-\left(  k+5\right)
\right)  ,$ if $k+7=8,...,a_{1}^{\left(  s\right)  }+1,$ $a_{1}^{\left(
s\right)  }\geq7\Leftrightarrow k+6=7,...,a_{1}^{\left(  s\right)  },$
$a_{1}^{\left(  s\right)  }\geq7\newline\left(  k+7,-\left(  k+6\right)
\right)  \overset{c)}{\rightarrow}\left(  -\left(  k+6\right)  ,k+5\right)  ,$
if $k+7=a_{1}^{\left(  s\right)  }+1,$ $a_{1}^{\left(  s\right)  }%
\geq7\Leftrightarrow k+6=a_{1}^{\left(  s\right)  },$ $a_{1}^{\left(
s\right)  }\geq7\newline\left(  k+7,-\left(  k+6\right)  \right)
\overset{c)}{\rightarrow}\left(  -\left(  k+6\right)  ,k+5\right)  ,$ if
$k+7=a_{1}^{\left(  s\right)  },$ $a_{1}^{\left(  s\right)  }\geq
8\Leftrightarrow k+6=a_{1}^{\left(  s\right)  }-1,$ $a_{1}^{\left(  s\right)
}\geq8\newline\left(  k+7,-\left(  k+6\right)  \right)
\overset{c)}{\rightarrow}\left(  -\left(  k+6\right)  ,k+5\right)  ,$ if
$k+7=a_{1}^{\left(  s\right)  }-1,$ $a_{1}^{\left(  s\right)  }\geq
9\Leftrightarrow k+6=a_{1}^{\left(  s\right)  }-2,$ $a_{1}^{\left(  s\right)
}\geq9$

\item $\left(  -\left(  k+6\right)  ,k+5\right)  ,$ if $k+6=a_{1}^{\left(
s\right)  }+1,$ $a_{1}^{\left(  s\right)  },$ $a_{1}^{\left(  s\right)  }-1,$
$a_{1}^{\left(  s\right)  }\geq7\newline\left(  -\left(  k+6\right)
,k+5\right)  \overset{d)}{\rightarrow}\left(  \mathbf{k+5,-}\left(
\mathbf{k+4}\right)  \right)  $ if $k+6=a_{1}^{\left(  s\right)  },$
$a_{1}^{\left(  s\right)  }-1,$ $a_{1}^{\left(  s\right)  }-2,$ $a_{1}%
^{\left(  s\right)  }\geq7.\newline$
\end{itemize}

From this, we conclude that every maximal directed trail that starts at each
of the vertices $v=\left(  k+6,-\left(  k+5\right)  \right)  \ $and $v=\left(
-\left(  k+6\right)  ,k+7\right)  $ if $k+6=6,...,n-1,n\geq7$ includes one of
the vertices: $w=\left(  k+5,-\left(  k+4\right)  \right)  ,$\textbf{\ }%
$\left(  k+4,-\left(  k+3\right)  \right)  $ and all corresponding directed
subtrails $\left(  v,w\right)  $ are directed paths. Since we already prove
that Claim \ref{clm-s} holds for these two vertices $w$ in the ranges of $k$
and $a_{1}^{\left(  s\right)  }$ given above, we conclude that Claim
\ref{clm-s} also holds for the vertices $v=\left(  k+6,-\left(  k+5\right)
\right)  \ $and $v=\left(  -\left(  k+6\right)  ,k+7\right)  $ if
$k+6=7,...,a_{1}^{\left(  s\right)  },$ $a_{1}^{\left(  s\right)  }\geq7.$
Therefore, Claim \ref{clm-s} holds for $\mathbf{\pm(}k+1,-\left(  k+3\right)
\mathbf{)\in}\mathcal{V}_{k+1,3}^{\left(  s\right)  }$ which finishes the
proof of Claim \ref{clm-s}.

We have thus shown that for all $s=2,...,t-1$ the graph $G^{\prime}%
(\Delta_{18}^{\left(  s\right)  })$ has only two cycles: a trivial cycle
$\pi_{0}=(0)$ and a non-zero cycle $\pi_{2}=(-1,1)$. Since the cycles $\pi
_{0}$\ and $\pi_{2}\ $are\ also the cycles$\ $of\ the graph $G_{\varepsilon
}(\Delta_{18}^{\left(  s\right)  })$ for each $\varepsilon\in\Big[\frac
{2}{3(2n+1)},\frac{2}{3(2n-1)}\Big)$ and for the corresponding cutout polygon
$P_{\varepsilon}(\pi_{2})$ we have $\Delta_{18}^{\left(  s\right)  }\cap
P_{\varepsilon}(\pi_{2})=\overline{WZ}$ $,$ the lemma is proven.
\end{proof}

\begin{acknowledgement}
The authors would like to thank Professor Damir Vuki\v{c}evi\'{c} for helpful
suggestions on the proof of Lemma \ref{Lemma-18-prvi}.
\end{acknowledgement}

\end{document}